\documentclass[11pt,a4paper]{article}
\usepackage[utf8x]{inputenc}
\usepackage{amsmath}
\usepackage{amsfonts}
\usepackage{amsthm}
\usepackage{amssymb}
\usepackage{graphicx}
\usepackage[margin=0.8in]{geometry}
\usepackage{hyperref}

\usepackage{mathtools}
\usepackage{apptools}
\usepackage{xcolor}
\usepackage{enumitem}
\usepackage{mathabx}
\usepackage{cases}
\usepackage{stmaryrd}
\usepackage[framemethod=tikz]{mdframed}
\usepackage{textgreek}
\usepackage{mathdots}

\newtheorem{thm}{Theorem}
\newtheorem{lem}[thm]{Lemma}

\newtheorem{algo}{Algorithm}

\theoremstyle{definition}

\newtheorem{rmk}[thm]{Remark}

\usepackage{chngcntr}
\usepackage{apptools}
\AtAppendix{\counterwithin{thm}{section}}

\newcounter{unnumber}

\newtheorem{prff}[unnumber]{Proof of Theorem \ref{existence}}
\newenvironment{prooff}{\prff\rm}{\hfill{$\blacksquare$}\endprff}

\newtheorem{prfff}[unnumber]{Proof of Lemma \ref{lemma8}}
\newenvironment{proofff}{\prfff\rm}{\hfill{$\blacksquare$}\endprfff}

\newtheorem{prffff}[unnumber]{Proof of Lemma \ref{lem:reg}}
\newenvironment{prooffff}{\prffff\rm}{\hfill{$\blacksquare$}\endprffff}

\newtheorem{prfffff}[unnumber]{Proof of Lemma \ref{lem:trunc}}
\newenvironment{proofffff}{\prfffff\rm}{\hfill{$\blacksquare$}\endprfffff}

\setlist[enumerate]{label=$\rm{(\roman*)}$,leftmargin=\parindent}

\newcommand{\sR}{\mathbb{R}}
\newcommand{\sL}{\mathbb{L}}
\newcommand{\sX}{\mathcal{X}}
\newcommand{\sY}{\mathcal{Y}}
\newcommand{\sH}{\mathcal{H}}

\newcommand{\sol}{\mathcal{Z}}

\newcommand{\mysum}{\displaystyle \sum\limits}

\newcommand{\Id}{\mathrm{Id}}
\newcommand{\bO}{\mathcal{O}}
\newcommand{\sB}{\mathbb{B}}

\newcommand{\E}{\mathcal{E}}
\newcommand{\F}{\mathcal{F}}

\newcommand{\Gap}{\mathtt{Gap}}

\newcommand{\Lag}{\mathcal{L}}

\newcommand{\bz}{\widebar{z}}

\newcommand{\Tm}{T_{\max}}
\newcommand{\CiV}{C}

\title{Fast Optimistic Gradient Descent Ascent (OGDA) method in continuous and discrete time}
\author{Radu Ioan Bo\c{t}\footnote{Faculty of Mathematics, University of Vienna, Oskar-Morgenstern-Platz 1, 1090 Vienna, Austria, email: \url{radu.bot@univie.ac.at}. Research partially supported by FWF (Austrian Science Fund), projects W 1260 and P 34922-N.}
	\and Ern\"{o} Robert Csetnek\footnote{Faculty of Mathematics, University of Vienna, Oskar-Morgenstern-Platz 1, 1090 Vienna, Austria, email: \url{robert.csetnek@univie.ac.at}. Research partially supported by FWF (Austrian Science Fund), project P 29809-N32.}
	\and Dang-Khoa Nguyen\footnote{Faculty of Mathematics, University of Vienna, Oskar-Morgenstern-Platz 1, 1090 Vienna, Austria, email: \url{dang-khoa.nguyen@univie.ac.at}. Research supported by FWF (Austrian Science Fund), project P 34922-N.}}

\begin{document}
	
\maketitle	
	
\begin{abstract}
In the framework of real Hilbert spaces we study continuous in time dynamics as well as numerical algorithms for the problem of approaching the set of zeros of a single-valued monotone and continuous operator $V$. The starting point of our investigations is a second order dynamical system that combines a vanishing damping term with the time derivative of $V$ along the trajectory, which can be seen as an analogous of the Hessian-driven damping in case the operator is originating from a potential.  Our method exhibits fast convergence rates of order $o \left( \frac{1}{t\beta(t)} \right)$ for $\|V(z(t))\|$, where $z(\cdot)$ denotes the generated trajectory and $\beta(\cdot)$ is a positive nondecreasing function satisfiyng a growth condition, and also for the restricted gap function, which is a measure of optimality for variational inequalities. We also prove the weak convergence of the trajectory to a zero of $V$.

Temporal discretizations of the dynamical system generate implicit and explicit numerical algorithms, which can be both seen as accelerated versions of the Optimistic Gradient Descent Ascent (OGDA) method for monotone operators, for which we prove that the generated sequence of iterates $(z_k)_{k \geq 0}$ shares the asymptotic features of the continuous dynamics. In particular we show for the implicit numerical algorithm convergence rates of order $o \left( \frac{1}{k\beta_k} \right)$ for $\|V(z^k)\|$ and the restricted gap function, where $(\beta_k)_{k \geq 0}$ is a positive nondecreasing sequence satisfying a growth condition. For the explicit numerical algorithm we show by additionally assuming  that the operator $V$ is Lipschitz continuous convergence rates of order $o \left( \frac{1}{k} \right)$ for $\|V(z^k)\|$ and the restricted gap function. All convergence rate statements are last iterate convergence results; in addition to these we prove for both algorithms the convergence of the iterates to a zero of $V$. To our knowledge, our study exhibits the best known convergence rate results for monotone equations. Numerical experiments indicate the overwhelming superiority of our explicit numerical algorithm over other methods designed to solve monotone equations governed by monotone and Lipschitz continuous operators.
\end{abstract}	

\noindent \textbf{Key Words.}
monotone equation, variational inequality,
Optimistic Gradient Descent Ascent (OGDA) method,
extragradient method,
Nesterov's accelerated gradient method,
Lyapunov analysis,
convergence rates,
convergence of trajectories,
convergence of iterates
\vspace{1ex}

\noindent \textbf{AMS subject classification.} 47J20, 47H05, 65K10, 65K15, 65Y20, 90C30, 90C52
%

\section{Introduction}	
	
Let $\sH$ be a real Hilbert space and $V \colon \sH \to \sH$ a monotone and continuous operator. We are interested in developing fast converging methods aimed to find a zero of $V$, or in other words, to solve the monotone equation	
\begin{equation}
\label{intro:pb:eq}
V \left( z \right) = 0,
\end{equation}
for which assume that it has a nonempty solution set $\sol$. The monotonicity and the continuity of $V$ imply that  $z_{*}$ is a solution of \eqref{intro:pb:eq} if and only if it is a solution of the following variational inequality
\begin{equation}
\label{intro:pb:vi}
\left\langle z - z_{*} , V \left( z \right) \right\rangle \geq 0 \quad \forall z \in \sH .
\end{equation}
One of the main motivations to study \eqref{intro:pb:eq} comes from minimax problems. More precisely, consider the problem 
\begin{equation}
\label{intro:m-m}
	\min\limits_{x \in \sX} \max\limits_{y \in \sY} \Phi \left( x , y \right) ,
\end{equation}
where $\sX$ and $\sY$ are real Hilbert spaces and $\Phi \colon \sX \times \sY \to \sR$ is a continuously differentiable and convex-concave function, i.e., $\Phi \left( \cdot , y \right)$ is convex for every $y \in \sY$ and $\Phi \left( x , \cdot \right)$ is convex for every $x \in \sX$. A solution of \eqref{intro:m-m} is a saddle point $\left( x_{*} , y_{*} \right) \in \sX \times \sY$ of $\Phi$, which means that it fulfills
\begin{equation*}
\Phi \left( x_{*} , y \right) \leq \Phi \left( x_{*} , y_{*} \right) \leq \Phi \left( x , y_{*} \right) \quad \forall \left( x , y \right) \in \sX \times \sY
\end{equation*}
or, equivalently,
\begin{equation}
\label{intro:opt}
\begin{dcases}
\nabla_{x} \Phi \left( x_{*} , y_{*} \right) 	& = 0 \\
- \nabla_{y} \Phi \left( x_{*} , y_{*} \right)	& = 0.
\end{dcases} \end{equation}
Taking into account that the mapping 
\begin{equation}
\label{saddlemapping}
\left( x , y \right) \mapsto \Bigl( \nabla_{x} \Phi \left( x , y \right) , - \nabla_{y} \Phi \left( x , y \right) \Bigr)
\end{equation}
 is monotone (\cite{Rockafellar:70}), it means that the problem of finding a saddle point of $\Phi$ eventually brings us back to the problem \eqref{intro:pb:eq}.

Both \eqref{intro:pb:eq} and \eqref{intro:m-m} are fundamental models in various fields such as optimization, economics, game theory, and partial differential equations. They have recently regained significant attention, in particular in the machine learning and data science community, due to the fundamental role they play, for instance, in multi agent reinforcement learning \cite{Omidshafiei-Pazis-Amato-How-Vian}, robust adversarial learning \cite{Madry-Makelov-Schmidt-Tsipras-Vladu} and generative adversarial networks (GANs) \cite{Goodfellow-et-al,Bohm-Sedlmayer-Bot-Csetnek}.

In this paper we develop fast continuous in time dynamics as well as numerical algorithms for solving \eqref{intro:pb:eq} and investigate their asymptotic/convergence properties. First we formulate a second order dynamical system that combines a vanishing damping term with the time derivative of $V$ along the trajectory, which can be seen as an analogous of the Hessian-driven damping in case the operator is originating from a potential. A continuously differentiable and nondecreasing function $\beta \colon \left[ t_{0} , + \infty \right) \to \left( 0 , + \infty \right)$, which appears in the system, plays an important role in the analysis. If $\beta$ satisfies a specific growth condition, which is for instance satisfied by polynomials including constant functions, then the method exhibits convergence rates of order $o \left( \frac{1}{t\beta(t)} \right)$ for $\|V(z(t))\|$, where $z(t)$ denotes the generated trajectory, and for the restricted gap function associated with \eqref{intro:pb:vi}. In addition, $z(t)$
converges asymptotically weakly to a solution of \eqref{intro:pb:eq}. 

By considering a temporal discretization of the dynamical system we obtain an \textit{implicit} numerical algorithm which exhibits convergence rates of order $o \left( \frac{1}{k \beta_{k}} \right)$ for $\|V(z^k)\|$ and the restricted gap function associated with \eqref{intro:pb:vi}, where $(\beta_k)_{k \geq 0}$ is a nondecreasing sequence and $(z_k)_{k \geq 0}$ is the generated sequence of iterates. For the latter we also prove that it converges weakly  to a solution of \eqref{intro:pb:eq}.

By a further more involved discretization of the dynamical system we obtain an \textit{explicit} numerical algorithm, which, under the additional assumption that $V$ is Lipschitz continuous, exhibits convergence rates of order $o \left( \frac{1}{k} \right)$ for $\|V(z^k)\|$ and the restricted gap function associated with \eqref{intro:pb:vi}, where $(z_k)_{k \geq 0}$ is the generated sequence of iterates, which is also to converge weakly to a solution of \eqref{intro:pb:eq}.

The resulting numerical schemes can be seen as accelerated versions of the Optimistic Gradient Descent Ascent (OGDA) method (\cite{Malitsky-Tam, Popov}) formulated in terms of a general monotone operator $V$. It should be also emphasized that the convergence rate statements for both the implicit and the explicit numerical algorithm are last iterate convergence results and are, to our knowledge, the best known convergence rate results for monotone equations. 

\subsection{Related works}

In the following we discuss some discrete and continuous methods from the literature designed to solve equations governed by monotone and (Lipschitz) continuous,  and \textit{not necessarily cocoercive} operators.  It has been recognized that the simplest scheme one can think of,  namely the forward algorithm, which, for a starting point $z^0 \in {\cal H} $ and a given step size $s >0$, reads for $k \geq 0$
\begin{equation*}
z^{k+1} := z^{k} - sV \left( z^{k} \right) ,
\end{equation*}
and mimics the classical gradient descent algorithm, does not converge. Unless for the trivial case, the operator in \eqref{saddlemapping}, which arises in connection with minimax problems, is only monotone and Lipschitz continuous but not cocoercive. Therefore, it was early recognized that explicit numerical methods for monotone equations require an operator corrector term.

In case $V$ is monotone and $L$-Lipschitz continuous, for $L >0$, Korpelevich \cite{Korpelevich} and Antipin \cite{Antipin} proposed to solve \eqref{intro:pb:eq} the nowadays very popular Extragradient (EG) method, which reads for $k \geq 0$
\begin{align}
	\begin{split}
		\label{algo:EG}
		\bz^{k} & := z^{k} - sV \left( z^{k} \right) \\
		z^{k+1} & := z^{k} - sV \left( \bz^{k} \right),
	\end{split}
\end{align}
and converges for a starting point $z^0 \in {\cal H} $ and $0 < s < \frac{1}{L}$ to a zero of $V$. The last iterate convergence rate for the extragradient method was only recently derived by Gorbuno-Loizou-Gidel in \cite{Gorbunov-Loizou-Gidel}.  For $\bz \in \sH$  and $\delta > 0$ we denote $\sB \left( \bz ; \delta \right) := \left\lbrace u \in \sH \colon \left\lVert \bz - u \right\rVert \leq \delta \right\rbrace$. For $z_{*} \in \sol \textrm{ and } \delta \left( z^{0} \right) := \left\lVert z_{*} - z^{0} \right\rVert$,  the \textit{restricted gap function} associated with the variational inequality \eqref{intro:pb:vi} is defined as (see \cite{Nesterov:07})
\begin{equation*}
	\Gap \left( z \right) := \sup\limits_{u \in \sB \left( z_{*} ; \delta \left( z^{0} \right) \right)} \left\langle z - u , V \left( u \right) \right\rangle \geq 0.
\end{equation*}
In \cite{Gorbunov-Loizou-Gidel} it was shown that
\begin{equation*}
	\left\lVert V \left( z^{k} \right) \right\rVert = \bO \left( \dfrac{1}{\sqrt{k}} \right) \mbox{ and }\Gap \left( z^{k} \right) = \bO \left( \dfrac{1}{\sqrt{k}} \right) \quad \textrm{as} \ k \to + \infty .
\end{equation*}

In the same setting, Popov introduced in \cite{Popov} for minmax problems and the operator in \eqref{saddlemapping} the following algorithm which,  when formulated for \eqref{intro:pb:eq}, reads for $k \geq 1$
\begin{equation}
	\label{algo:OGDA}
	z^{k+1} := z^{k} - 2s V \left( z^{k} \right) + s V \left( z^{k-1} \right),
\end{equation}
and converges for starting points $z^0, z^1 \in {\cal H} $ and step size $0 < s < \frac{1}{2L}$ to a zero of $V$. This algorithm is usually known as the Optimistic Gradient Descent Ascent (OGDA) method, a name which we adopt also for the general formulation in \eqref{algo:OGDA}.  Recently, Chavdarova-Jordan-Zampetakis proved in \cite{Chavdarova-Jordan-Zampetakis} that for $0 < s < \frac{1}{16L}$ the scheme exhibits the following best-iterate convergence rate
\begin{equation*}
	\min\limits_{1 \leq i \leq k} \left\lVert V \left( z^{i} \right) \right\rVert = \bO \left( \dfrac{1}{\sqrt{k}} \right) \quad \textrm{ as } k \to + \infty .
\end{equation*}
We notice also that, according to Golowich-Pattathil-Daskalakis-Ozdaglar (see \cite{Golowich-Pattathil-Daskalakis,Golowich-Pattathil-Daskalakis-Ozdaglar}), the lower-bound for the restricted gap function for the algorithms \eqref{algo:EG} and \eqref{algo:OGDA} is of $\bO \left( 1 / \sqrt{k} \right)$ as $k \rightarrow +\infty$.

The solving of equation \eqref{intro:pb:eq} can be also addressed in the general framework of continuous and discrete time methods for finding the zeros of a maximally monotone operator. Attouch-Svaiter introduced in \cite{Attouch-Svaiter:11} (see also \cite{Csetnek-Malitsky-Tam:19}) a first order evolution equation linked to the Newton and the Levenberg-Marquardt methods, which when applied to \eqref{intro:pb:eq} reads
\begin{equation}
	\label{ds:Newton}
	\dot{z} \left( t \right) + \lambda(t) \dfrac{d}{dt}V(z(t)) +  \lambda(t) V(z(t)) = 0,
\end{equation} 
where $t \mapsto \lambda(t)$ is a continuous mapping,  and for which they proved that its trajectories converge weakly to a zero of $V$.  \color{black} Attouch-Peypouquet studied in \cite{Attouch-Peypouquet:19} the following second-order differential equation with vanishing damping
\begin{equation}
	\label{ds:RIPA}
	\ddot{z} \left( t \right) + \dfrac{\alpha}{t} \dot{z} \left( t \right) + A_{\gamma \left( t \right)} \left( z \left( t \right) \right) = 0,
\end{equation}
where $A \colon \sH \rightrightarrows \sH$ is a possibly set-valued maximally monotone operator,
\begin{equation*}
	A_{\gamma} := \dfrac{1}{\gamma} \left( \Id - J_{\gamma A} \right)
\end{equation*}
stands for the Yosida approximation of $A$ of index $\gamma > 0$, and $J_{\gamma A} = (\Id + \gamma A)^{-1} : {\cal H} \rightarrow {\cal H}$ for the resolvent of $\gamma A$. The dynamical system 	\eqref{ds:RIPA} gives rise via implicit discretization to the following so-called Regularized Inertial Proximal Algorithm, which for every $k \geq 1$ reads
\begin{align*}
	\begin{split}
		\bz^{k} 	& := z^{k} + \left( 1 - \dfrac{\alpha}{k} \right) \left( z^{k} - z^{k-1} \right) \\
		z^{k+1}		& := \dfrac{\gamma_{k}}{\gamma_{k} + s} \bz^{k} + \dfrac{s}{\gamma_{k} + s} J_{\left( \gamma_{k} + s \right) A} \left( \bz^{k} \right) ,
	\end{split}
\end{align*}
$z^0, z^1 \in {\cal H}$ are the starting points, $\alpha >2$, $s >0$ and $\gamma_k = (1+\varepsilon)\frac{s}{\alpha^2}k^2$ for every $k\geq 1$, with $\varepsilon >0$ fixed.  In \cite{Attouch-Peypouquet:19} it was shown that the discrete velocity $z^{k+1} -z^k$ vanishes with a rate of convergence of $\bO \left( 1 / k \right)$ as $k \rightarrow +\infty$ and that the sequence of iterates converges weakly to a zero of $A$. The continuous time approach in \eqref{ds:RIPA} has been extended by Attouch-L\'{a}szl\'{o} in \cite{Attouch-Laszlo:21} by adding a Newton-like correction term $\xi \frac{d}{dt} \left( A_{\gamma \left( t \right)} \left( z \left( t \right) \right) \right)$,  with $\xi \geq 0$, whereas the discrete counterpart of this scheme was proposed and investigated in \cite{Attouch-Laszlo:20}.

For an inertial evolution equation with asymptotically vanishing damping terms approaching the set of primal-dual solutions of a smooth convex optimization problem with linear equality constraints, that can also be seen as the solution set of a monotone operator equation,  and exhibiting fast convergence rates expressed in terms of the value functions, the feasibilty measure and the primal-dual gap we refer to the recent works \cite{Attouch-Chbani-Fadili-Riahi:22, Bot-Nguyen:21}.

We also want to mention the implicit method for finding the zeros of a maximally monotone operator proposed by Kim in \cite{Kim}, which relies on the performance estimation problem approach and makes use of computer-assisted tools.

In the case when $V$ is monotone and $L$-Lipschitz continuous, for $L >0$,  Yoon-Ryu recently proposed in \cite{Yoon-Ryu} an accelerated algorithm for solving \eqref{intro:pb:eq}, called Extra Anchored Gradient (EAG) algorithm, designed by using anchor variables, a technique that can be traced back to Halpern's algorithm (see \cite{Halpern}). The iterative scheme of the EAG algorithm reads for every $k \geq 0$
\begin{align}
	\begin{split}
		\label{algo:EAG}
		\bz^{k} 	& := z^{k} + \dfrac{1}{k+2} \left( z^{0} - z^{k} \right) - s_kV \left( z^{k} \right) \\
		z^{k+1} 	& := z^{k} + \dfrac{1}{k+2} \left( z^{0} - z^{k} \right) - s_kV \left( \bz^{k} \right),
	\end{split}
\end{align}
where $z^0 \in {\cal H}$ is the starting point and the sequence of step sizes $(s_k)_{k \geq 0}$ is either chosen to be equal to a constant in the interval $\left(0,\frac{1}{8L}\right]$ or such that 
\begin{equation}\label{stepsizeEAG}
		s_{k+1} := s_{k} \left( 1 - \dfrac{1}{\left( k+1 \right) \left( k+3 \right)} \dfrac{s_{k}^2 L^{2}}{1 - s_{k}^2 L^{2}} \right) \quad \forall k \geq 0,
	\end{equation}
where $s_0 \in \left(0,\frac{3}4{L}\right)$. This iterative scheme exhibits in both cases the convergence rate of
\begin{equation*}
	\left\lVert V \left( z^{k} \right) \right\rVert = \bO \left( \dfrac{1}{k} \right) \quad \textrm{ as } k \to + \infty.
\end{equation*}
Later, Lee-Kim proposed in  \cite{Lee-Kim} an algorithm formulated in the same spirit for the problem of finding the saddle points of a smooth nonconvex-nonconcave function.

Further variants of the anchoring based method have been proposed by Tran-Dinh in \cite{Tran-Dinh} and together with Luo in \cite{Tran-Dinh-Luo}, which all exhibit the same convergence rate for $\|V(z^k)\|$ as EAG. Tran-Dinh in \cite{Tran-Dinh} and  Park-Ryu in \cite{Park-Ryu} pointed out the existence of some connections between the anchoring approach and Nesterov's acceleration technique used for the minimization of smooth and convex functions (\cite{Nesterov:83,Nesterov:book}).

\subsection{Our contributions}

The starting point of our investigations is a second order evolution equation associated with problem \eqref{intro:pb:eq} that combines a vanishing damping term with the time derivative of $V$ along the trajectory, which will then lead via temporal discretizations to the implicit and the explicit algorithms. In \cite{Chavdarova-Jordan-Zampetakis} several dynamical systems of EG and OGDA type were proposed, mainly in the spirit of the heavy ball method, that is, with a constant damping term, exhibiting a convergence rate of $\|V(z(t))\| = \bO \left(1/\sqrt{t} \right)$ as $t \to + \infty$ and,  in case $V$ is bilinear,  weak convergence of the trajectory $z(t)$  to a zero of the operator. 

One of the main discoveries of the last decade was that asymptotically vanishing damping terms (see  \cite{Su-Boyd-Candes, Attouch-Cabot,Attouch-Peypouquet-Redont}) lead to the acceleration of the  convergence of the value functions along the trajectories of a inertial gradient systems. Moreover, when enhancing the evolution equations also with Hessian-driven damping terms, the rate of convergence of the gradient along the trajectories can be accelerated, too (\cite{Attouch-Peypouquet-Redont, Shi-Du-Jordan-Su}).  It is natural to ask whether asymptomatically vanishing damping terms have the same accelerating impact on the values of the norm of the governing operator along the trajectories of inertial dynamical systems associated with monotone (not necessarily potential) operators.

\begin{mdframed}
	The dynamical system which we associate to \eqref{intro:pb:eq} reads
	\begin{equation*}
		\begin{dcases}
			\ddot{z} \left( t \right) + \dfrac{\alpha}{t} \dot{z} \left( t \right) + \beta \left( t \right) \dfrac{d}{dt} V \left( z \left( t \right) \right) + \dfrac{1}{2} \left( \dot{\beta} \left( t \right) + \dfrac{\alpha}{t} \beta \left( t \right) \right) V \left( z \left( t \right) \right) = 0 \\
			z \left( t_{0} \right) = z^{0} \textrm{ and } \dot{z} \left( t_{0} \right) = \dot z^0
		\end{dcases}
	\end{equation*}
	where $t_{0} > 0$, $\alpha \geq 2$, $\left( z^{0} , \dot z^0 \right) \in \sH \times \sH$, $\beta \colon \left[ t_{0} , + \infty \right) \to \left( 0 , + \infty \right)$ is a continuously differentiable and nondecreasing which satisfies the following growth condition
	\begin{equation*}
		0 \leq \sup\limits_{t \geq t_{0}} \dfrac{t \dot{\beta} \left( t \right)}{\beta \left( t \right)} \leq \alpha - 2,
	\end{equation*}
and $t \mapsto V(z(t))$ is assumed to be differentiable on $[t_{0} , + \infty)$.
\end{mdframed}

For $z_{*}\in \sol$ and the dynamics generated by this dynamical system we will prove that $$\left\langle z \left( t \right) - z_{*} , V \left( z \left( t \right) \right) \right\rangle = \bO \left( \dfrac{1}{t \beta \left( t \right)} \right) \ \mbox{and} \ \left\lVert V \left( z \left( t \right) \right) \right\rVert = \bO \left( \dfrac{1}{t \beta \left( t \right)} \right).$$ 
Further, by assuming that
	\begin{equation*}
	0 \leq \sup\limits_{t \geq t_{0}} \dfrac{t \dot{\beta} \left( t \right)}{\beta \left( t \right)} < \alpha - 2,
	\end{equation*}
we will prove that the  trajectory $z \left( t \right)$ converges weakly to a solution of \eqref{intro:pb:eq} as $t \to + \infty$ and it holds
	\begin{equation*}
	\left\lVert \dot{z} \left( t \right) \right\rVert = o \left( \dfrac{1}{t} \right) \quad \textrm{ as } t \to + \infty ,
	\end{equation*}
	and
	\begin{equation*}
	\left\langle z \left( t \right) - z_{*} , V \left( z \left( t \right) \right) \right\rangle = o \left( \dfrac{1}{t \beta \left( t \right)} \right) \mbox{ and } \left\lVert V \left( z \left( t \right) \right) \right\rVert = o \left( \dfrac{1}{t \beta \left( t \right)} \right) \quad \textrm{ as } t \to + \infty .
	\end{equation*}

Polynomial parameter functions $\beta(t) = \beta_0 t^\rho$, for $\beta_0>0$ and $\rho \geq 0$, satisfy the two growth conditions for $\alpha \geq \rho + 2$ and $\alpha > \rho + 2$, respectively.

To the main contributions of this work belongs not only the improvement of the convergence rates in \cite{Chavdarova-Jordan-Zampetakis} in both continuous and discrete time,   but in particular the surprising discovery that this can be achieved by means of asymptotically vanishing damping, respectively, as we will see below, of Nesterov momentum.  This shows that the accelerating effect of inertial methods with asymptotically vanishing damping/Nesterov momentum goes beyond convex optimization and opens the gate towards new unexpected research perspectives.

\begin{rmk}{(\bf restricted gap function)}
The convergence rates for $t \mapsto \left\langle z \left( t \right) - z_{*} , V \left( z \left( t \right) \right) \right\rangle $ and $t \mapsto \left\lVert V \left( z \left( t \right) \right) \right\rVert$ can be easily transferred to the restricted gap function associated with the variational inequality \eqref{intro:pb:vi}. Indeed, for $z_* \in \sol$, let $\delta (z^0) :=  \left\lVert z^0 - z_{*} \right\rVert$, $u \in \sB \left( z_{*} ; \delta(z^0) \right)$ and $t \geq t_{0}$. It holds
	\begin{align*}
		0 \leq   \left\langle z \left( t \right) - u , V \left( u \right) \right\rangle
		 \leq \left\langle z \left( t \right) - u , V \left( z \left( t \right) \right) \right\rangle  & = \left\langle z \left( t \right) - z_{*} , V \left( z \left( t \right) \right) \right\rangle +  \left\langle z_{*} - u , V \left( z \left( t \right) \right) \right\rangle \nonumber \\
		& \leq  \left\langle z \left( t \right) - z_{*} , V \left( z \left( t \right) \right) \right\rangle +  \left\lVert u - z_{*} \right\rVert \left\lVert V \left( z \left( t \right) \right) \right\rVert ,
	\end{align*}
which implies that for every $t \geq t_{0}$
\begin{equation*}
	0 \leq  \Gap \left( z \left( t \right) \right) = \sup_{u \in \sB \left( z_{*} ; \delta(z^0) \right)} \left\langle z \left( t \right) - u , V \left( u \right) \right\rangle  \leq \left\langle z \left( t \right) - z_{*} , V \left( z \left( t \right) \right) \right\rangle + \delta(z^0)\left\lVert V \left( z \left( t \right) \right) \right\rVert,
\end{equation*}
which proofs our claim. The same remark can be obviously made in the discrete case.
\end{rmk}

Further we provide two temporal discretizations of the dynamical system, one of implicit and one of explicit type.
\begin{mdframed}
{\bf Implicit Fast OGDA:} {\it
		Let $\alpha >2$, $z^{0}, z^{1} \in \sH$, $s > 0$, and $\left(\beta_{k} \right)_{k \geq 1}$ a positive and nondecreasing sequence which satisfies
		\begin{equation*}
		0 \leq \sup_{k \geq 1} \dfrac{k \left( \beta_{k} - \beta_{k-1} \right)}{\beta_{k}} < \alpha - 2 .
		\end{equation*}
		For every $k \geq 1$ we set
		\begin{align*}	
	z^{k+1}
	:= & \ z^{k} + \left( 1 - \dfrac{\alpha}{k + \alpha} \right) \left( z^{k} - z^{k-1} \right) - \dfrac{s \left( \alpha \beta_{k} + k \left( \beta_{k} - \beta_{k-1} \right) \right)}{2 \left( k + \alpha \right)} V \left( z^{k+1} \right) \nonumber \\
	& - \dfrac{sk \beta_{k-1}}{k + \alpha} \left( V \left( z^{k+1} \right) - V \left( z^{k} \right) \right). 
\end{align*}}
\end{mdframed}
We will prove that, for $z_{*} \in \sol$, it holds
	\begin{equation*}
		\left\lVert z^{k} - z^{k-1} \right\rVert = o \left( \dfrac{1}{k} \right) \textrm{ as } k \to + \infty ,
	\end{equation*}
	and
	\begin{equation*}
		\left\langle z^{k} - z_{*} , V \left( z^{k} \right) \right\rangle =  o \left( \dfrac{1}{k \beta_{k}} \right) \mbox{ and } \left\lVert V \left( z^{k} \right) \right\rVert = o \left( \dfrac{1}{k \beta_{k}} \right) \textrm{ as } k \to + \infty,
	\end{equation*}
and that the sequence $\left(z^{k} \right) _{k \geq 0}$  converges weakly to a solution in $\sol$.

The constant sequence $\beta_{k} \equiv 1$ obviously satisfies the growth condition required in the implicit numerical scheme and for this choice the generated sequence $\left(z^{k} \right) _{k \geq 0}$ fulfills for every $k \geq 1$
	\begin{align*}	
	z^{k+1}
	& = z^{k} + \left( 1 - \dfrac{\alpha}{k + \alpha} \right) \left( z^{k} - z^{k-1} \right) - \dfrac{s \alpha }{2 \left( k + \alpha \right)} V \left( z^{k+1} \right) - \dfrac{sk }{k + \alpha} \left( V \left( z^{k+1} \right) - V \left( z^{k} \right) \right).
\end{align*}
From the general statement we have that
	\begin{equation*}
		\left\lVert z^{k} - z^{k-1} \right\rVert = o \left( \dfrac{1}{k} \right) , \  \left\langle z^{k} - z_{*} , V \left( z^{k} \right) \right\rangle = o \left( \dfrac{1}{k} \right) \ \mbox{and} \ \left\lVert V \left( z^{k} \right) \right\rVert = o \left( \dfrac{1}{k} \right) \textrm{ as } k \to + \infty,
	\end{equation*}
and $\left(z^{k} \right) _{k \geq 0}$ converges weakly to a solution in $\sol$.

A further contribution of this work is therefore this numerical algorithm with Nesterov momentum for solving \eqref{intro:pb:eq}, obtained by implicit temporal discretization of the inertial evolution equation and which reproduces all its convergence properties in discrete time.

Only for the explicit discrete scheme we will additionally assume that the operator $V$ is $L$-Lipschitz continuous, with $L>0$.

\begin{mdframed}
{\bf Explicit Fast OGDA:} {\it Let $\alpha >2$, $z^{0} , z^{1}, \bz^{0} \in \sH$, and $0 < s < \frac{1}{2L}$.
		For every $k \geq 1$ we set			
			\begin{align*}
				\bz^{k} & := z^{k} + \left( 1 - \dfrac{\alpha}{k + \alpha} \right) \left( z^{k} - z^{k-1} \right) - \dfrac{\alpha s}{2 \left( k + \alpha \right)} V \left( \bz^{k-1} \right)  \\
				z^{k+1} & := \bz^{k} - \dfrac{s}{2} \left( 1 + \dfrac{k}{k + \alpha} \right) \left( V \left( \bz^{k} \right) - V \left( \bz^{k-1} \right) \right) .
			\end{align*}}
\end{mdframed}
When taking a closer look at its equivalent formulation, which reads for every $k \geq 1$
\begin{align*}
				\bz^{k} & := z^{k} + \left( 1 - \dfrac{\alpha}{k + \alpha} \right) \left( z^{k} - z^{k-1} \right) - \dfrac{\alpha s}{2 \left( k + \alpha \right)} V \left( \bz^{k-1} \right)  \\
				z^{k+1} & := z^{k} + \left( 1 - \dfrac{\alpha}{k + \alpha} \right) \left( z^{k} - z^{k-1} \right) - \dfrac{\alpha s}{2 \left( k + \alpha \right)} V \left( \bz^{k} \right) - \dfrac{sk}{k + \alpha} \left( V \left( \bz^{k} \right) - V \left( \bz^{k-1} \right) \right),
			\end{align*}
one can notice that the iterative scheme can be seen as an accelerated version of the OGDA method. An important feature of the explicit Fast OGDA method is that it requires the evaluation of $V$ only at the elements of the sequence $\left( \bz^{k} \right) _{k \geq 0}$, while the Extragradient method   \eqref{algo:EG}  and the Extra Anchored Gradient  method \eqref{algo:EAG} require the evaluation of $V$ at both sequences $\left( z^{k}  \right) _{k \geq 0}$ and $\left( \bz^{k} \right) _{k \geq 0}$.

We will show that, for $z_{*} \in \sol$, it holds
	\begin{align*}
		& \left\lVert z^{k} - z^{k-1} \right\rVert = o \left( \dfrac{1}{k} \right), \quad \left\langle z^{k} - z_{*} , V \left( z^{k} \right) \right\rangle = o \left( \dfrac{1}{k} \right),\\ 
& \left\lVert V \left( z^{k} \right) \right\rVert = o \left( \dfrac{1}{k} \right) \ \mbox{and} \ \left\lVert V \left( \bz^{k} \right) \right\rVert = o \left( \dfrac{1}{k} \right) \textrm{ as } k \to + \infty,
	\end{align*}
and that also for this algorithm the generated sequence $\left(z^{k} \right) _{k \geq 0}$ converges weakly to a solution in $\sol$.

 Another main contribution of this work is the explicit Fast OGDA method with Nesterov momentum and operator correction terms, for which we show the best convergence rate results known in the literature of explicit algorithms for monotone inclusions and the convergence of the iterates to a zero of the operator.  We illustrate the theoretical findings with numerical experiments, which show the overwhelming superiority of our method over other numerical algorithms designed to solve monotone equations governed by monotone and Lipschitz continuous operators.  These include the algorithms designed by using ``anchoring'' techniques,  for which the tracing of the iterates back to the starting value seems to have a slowing effect on the convergence performances.

\begin{rmk}{(\bf the role of the time scaling parameter function $\beta$)}\label{remark2} The function $\beta$ which appears in the formulation of the dynamical system can be seen as a \textit{time scaling parameter function} in the spirit of recent investigations on this topic (see, for instance, \cite{Attouch-Chbani-Fadili-Riahi:20,Attouch-Chbani-Fadili-Riahi:23}) in the context of the minimization of a smooth convex function. It was shown that, when used in combination with vanishing damping (and also with Hessian-driven damping) terms, time scaling functions improve the convergence rates of the function values and of the gradient. The positive effect of the time scaling on the convergence rates can be transferred to the numerical schemes obtained via \textit{implicit discretization}, as it was recently pointed out by Attouch-Chbani-Riahi in \cite{Attouch-Chbani-Riahi:SIOPT}, and long time ago by G\"uler in \cite{Guler:91,Guler:92} for the proximal point algorithm, which may exhibit convergence rates for the objective function values of $o \left( 1/k^{\rho} \right)$ rate, for arbitrary $\rho > 0$. On the other hand, this does not hold for numerical schemes obtained via \textit{explicit discretization}, as it is the gradient method for which it is known that the convergence rate of $o \left( 1/k^{2} \right)$ for the objective function values (see \cite{Attouch-Peypouquet:16}) cannot be improved in general (\cite{Nesterov:83,Nesterov:book}). 

This explains why the discretization of the parameter function $\beta$ appears only in the implicit numerical scheme and in the corresponding convergence rates, and \textit{not} in the explicit numerical scheme.
\end{rmk}

\section{The continuous time approach}	

In this section we will analyze the continuous time scheme proposed for \eqref{intro:pb:eq}, and which we recall for convenience in the following.
\begin{mdframed}
	For $t_0 >0$ we consider on $[t_0, +\infty)$ the dynamical system
	\begin{equation}	
\label{ds}
	\begin{dcases}
	\ddot{z} \left( t \right) + \dfrac{\alpha}{t} \dot{z} \left( t \right) + \beta \left( t \right) \dfrac{d}{dt} V \left( z \left( t \right) \right) + \dfrac{1}{2} \left( \dot{\beta} \left( t \right) + \dfrac{\alpha}{t} \beta \left( t \right) \right) V \left( z \left( t \right) \right) = 0 \\
	z \left( t_{0} \right) = z^{0} \textrm{ and } \dot{z} \left( t_{0} \right) = \dot z^{0},
	\end{dcases}
	\end{equation}
	where $\alpha \geq 2$, $\left( z^{0} , \dot z^{0} \right) \in \sH \times \sH$, $\beta \colon \left[ t_{0} , + \infty \right) \to \left( 0 , + \infty \right)$ is a continuously differentiable and nondecreasing function which satisfies the following growth condition
	\begin{equation}	
	\label{as:beta}	
		0 \leq \sup\limits_{t \geq t_{0}} \dfrac{t \dot{\beta} \left( t \right)}{\beta \left( t \right)} \leq \alpha - 2,
	\end{equation}
and $t \mapsto V(z(t))$ is assumed to be differentiable on $[t_{0} , + \infty)$.
\end{mdframed}

Let $z_{*} \in \sol$ and $0 \leq \lambda \leq \alpha - 1$.
We consider the following energy function $\E_{\lambda} \colon \left[ t_{0} , + \infty \right) \to \left[ 0 , + \infty \right)$,
\begin{align}
\E_{\lambda} \left( t \right)
:= & \ \dfrac{1}{2} \left\lVert 2 \lambda \left( z \left( t \right) - z_{*} \right) + t \Bigl( 2 \dot{z} \left( t \right) + \beta \left( t \right) V \left( z \left( t \right) \right) \Bigr) \right\rVert ^{2} + 2 \lambda \left( \alpha - 1 - \lambda \right) \left\lVert z \left( t \right) - z_{*} \right\rVert ^{2} \nonumber \\
& + 2 \lambda t \beta \left( t \right) \left\langle z \left( t \right) - z_{*} , V \left( z \left( t \right) \right) \right\rangle + \dfrac{1}{2} t^{2} \beta^{2} \left( t \right) \left\lVert V \left( z \left( t \right) \right) \right\rVert ^{2}, \label{defi:E}
\end{align}
which will play a fundamental role in our analysis. By taking into consideration \eqref{intro:pb:vi}, for every $0 \leq \lambda \leq \alpha - 1$ we have
\begin{equation*}
\E_{\lambda} \left( t \right) \geq 0 \quad \forall t \geq t_{0} .
\end{equation*}
Denote
\begin{equation}
\label{defi:w}
w: \left[ t_{0} , + \infty \right) \to \mathbb{R}, \quad w \left( t \right) := \dfrac{1}{2} \left( \left( \alpha - 2 \right) \dfrac{\beta \left( t \right)}{t} - \dot{\beta} \left( t \right) \right).
\end{equation}
The growth condition \eqref{as:beta} guarantees that $w(t) \geq 0$ for every $t \geq t_0$.

First we will show that the energy dissipates with time.
\begin{lem}
	\label{lem:dif}
	Let $z \colon \left[ t_{0} , + \infty \right) \to \sH$ be a solution of \eqref{ds}, $z_* \in \sol$ and $0 \leq \lambda \leq \alpha - 1$.  Then for every $t \geq t_{0}$ it holds
	\begin{align}
	\dfrac{d}{dt} \E_{\lambda} \left( t \right)
	 \leq \ & - 2 \lambda t w \left( t \right) \left\langle z \left( t \right) - z_{*} , V \left( z \left( t \right) \right) \right\rangle
	+ t \beta \left( t \right) \Bigl( \left( \alpha - 1 - \lambda \right) \beta \left( t \right) - 2t w \left( t \right) \Bigr) \left\lVert V \left( z \left( t \right) \right) \right\rVert ^{2} \nonumber \\
	& - \left( \alpha - 1 - \lambda \right) t \left\lVert 2 \dot{z} \left( t \right) - \beta \left( t \right) V \left( z \left( t \right) \right) \right\rVert ^{2} . \label{dif:inq}
	\end{align}
\end{lem}
\begin{proof}
Let $t \geq t_{0}$ be fixed. From the definition of the dynamical system \eqref{ds} we have
	\begin{equation*}
	2 t \ddot{z} \left( t \right) + t \beta \left( t \right) \dfrac{d}{dt} V \left( z \left( t \right) \right)
	= - 2 \alpha \dot{z} \left( t \right) - t \beta \left( t \right) \dfrac{d}{dt} V \left( z \left( t \right) \right) - \left( t \dot{\beta} \left( t \right) + \alpha \beta \left( t \right) \right) V \left( z \left( t \right) \right) .
	\end{equation*}
	Therefore
\begingroup
\allowdisplaybreaks
	\begin{align}
	& \dfrac{d}{dt} \left( \dfrac{1}{2} \left\lVert 2 \lambda \left( z \left( t \right) - z_{*} \right) + t \Bigl( 2 \dot{z} \left( t \right) + \beta \left( t \right) V \left( z \left( t \right) \right) \Bigr) \right\rVert ^{2} \right) \nonumber \\
	= \ 	& \left\langle 2 \lambda \left( z \left( t \right) - z_{*} \right) + t \Bigl( 2 \dot{z} \left( t \right) + \beta \left( t \right) V \left( z \left( t \right) \right) \Bigr) , \right. \nonumber \\
	& \ \ \left. 2 \left( \lambda + 1 \right) \dot{z} \left( t \right) + \left( t \dot{\beta} \left( t \right) + \beta \left( t \right) \right) V \left( z \left( t \right) \right) + 2t \ddot{z} \left( t \right) + t \beta \left( t \right) \dfrac{d}{dt} V \left( z \left( t \right) \right) \right\rangle \nonumber \\
	= \ 	& \left\langle 2 \lambda \left( z \left( t \right) - z_{*} \right) + 2t \dot{z} \left( t \right) + t \beta \left( t \right) V \left( z \left( t \right) \right) , \right. \nonumber \\
	& \ \ \left. 2 \left( \lambda + 1 - \alpha \right) \dot{z} \left( t \right) + \left( 1 - \alpha \right) \beta \left( t \right) V \left( z \left( t \right) \right) - t \beta \left( t \right) \dfrac{d}{dt} V \left( z \left( t \right) \right) \right\rangle \nonumber \\
	= \ 	& 4 \lambda \left( \lambda + 1 - \alpha \right) \left\langle z \left( t \right) - z_{*} , \dot{z} \left( t \right) \right\rangle
	+ 2 \lambda \left( 1 - \alpha \right) \beta \left( t \right) \left\langle z \left( t \right) - z_{*} , V \left( z \left( t \right) \right) \right\rangle \nonumber \\
	& - 2 \lambda t \beta \left( t \right) \left\langle z \left( t \right) - z_{*} , \dfrac{d}{dt} V \left( z \left( t \right) \right) \right\rangle
	+ 4 \left( \lambda + 1 - \alpha \right) t \left\lVert \dot{z} \left( t \right) \right\rVert ^{2} \nonumber \\
	& + 2 t \left( \lambda + 2 - 2 \alpha \right) \beta \left( t \right) \left\langle \dot{z} \left( t \right) , V \left( z \left( t \right) \right) \right\rangle
	- 2t^{2} \beta \left( t \right) \left\langle \dot{z} \left( t \right) , \dfrac{d}{dt} V \left( z \left( t \right) \right) \right\rangle \nonumber \\
	& + \left( 1 - \alpha \right) t \beta^{2} \left( t \right) \left\lVert V \left( z \left( t \right) \right) \right\rVert ^{2}
	- t^{2} \beta^{2} \left( t \right) \left\langle V \left( z \left( t \right) \right) , \dfrac{d}{dt} V \left( z \left( t \right) \right) \right\rangle . \label{dif:E1}
	\end{align}
\endgroup
By differentiating the other terms of the energy function it yields
	\begin{align}
	& \dfrac{d}{dt} \left( 2 \lambda \left( \alpha - 1 - \lambda \right) \left\lVert z \left( t \right) - z_{*} \right\rVert ^{2} + 2 \lambda t \beta \left( t \right) \left\langle z \left( t \right) - z_{*} , V \left( z \left( t \right) \right) \right\rangle + \dfrac{1}{2} t^{2} \beta^{2} \left( t \right) \left\lVert V \left( z \left( t \right) \right) \right\rVert ^{2} \right) \nonumber \\
	= \ & 4 \lambda \left( \alpha - 1 - \lambda \right) \left\langle z \left( t \right) - z_{*} , \dot{z} \left( t \right) \right\rangle
	+ 2 \lambda \left( \beta \left( t \right) + t \dot{\beta} \left( t \right) \right) \left\langle z \left( t \right) - z_{*} , V \left( z \left( t \right) \right) \right\rangle \nonumber \\
	& + 2 \lambda t \beta \left( t \right) \left\langle \dot{z} \left( t \right) , V \left( z \left( t \right) \right) \right\rangle
	+ 2 \lambda t \beta \left( t \right) \left\langle z \left( t \right) - z_{*} , \dfrac{d}{dt} V \left( z \left( t \right) \right) \right\rangle \nonumber \\
	& + t \beta \left( t \right) \left( \beta \left( t \right) + t \dot{\beta} \left( t \right) \right) \left\lVert V \left( z \left( t \right) \right) \right\rVert ^{2}
	+ t^{2} \beta^{2} \left( t \right) \left\langle V \left( z \left( t \right) \right) , \dfrac{d}{dt} V \left( z \left( t \right) \right) \right\rangle . \label{dif:E2}
	\end{align}
By summing up \eqref{dif:E1} and \eqref{dif:E2}, and then using the definition of $w$ in \eqref{defi:w}, we conclude that
	\begin{align*}
	\dfrac{d}{dt} \E_{\lambda} \left( t \right)
	& = - 2 \lambda t w \left( t \right) \left\langle z \left( t \right) - z_{*} , V \left( z \left( t \right) \right) \right\rangle + 4 \left( \lambda + 1 - \alpha \right) t \left\lVert \dot{z} \left( t \right) \right\rVert ^{2} \nonumber \\
	& \quad + 4 \left( \lambda + 1 - \alpha \right) t \beta \left( t \right) \left\langle \dot{z} \left( t \right) , V \left( z \left( t \right) \right) \right\rangle- 2t^{2} \beta \left( t \right) \left\langle \dot{z} \left( t \right) , \dfrac{d}{dt} V \left( z \left( t \right) \right) \right\rangle \nonumber \\
	& \quad - 2t^{2} \beta \left( t \right) w \left( t \right) \left\lVert V \left( z \left( t \right) \right) \right\rVert ^{2} .
	\end{align*}
	Finally, we observe that
	\begin{align*}
	& 4 \left( \lambda + 1 - \alpha \right) t \left\lVert \dot{z} \left( t \right) \right\rVert ^{2} + 4 \left( \lambda + 1 - \alpha \right) t \beta \left( t \right) \left\langle \dot{z} \left( t \right) , V \left( z \left( t \right) \right) \right\rangle \nonumber \\
	= \ 	& t \beta \left( t \right) \Bigl( \left( \alpha - 1 - \lambda \right) \beta \left( t \right) - 2t w \left( t \right) \Bigr) \left\lVert V \left( z \left( t \right) \right) \right\rVert ^{2} - \left( \alpha - 1 - \lambda \right) t \left\lVert 2 \dot{z} \left( t \right) - \beta \left( t \right) V \left( z \left( t \right) \right) \right\rVert ^{2} .
	\end{align*}
	This, in combination with $\left\langle \dot{z} \left( t \right) , \frac{d}{dt} V \left( z \left( t \right) \right) \right\rangle \geq 0$ for every $t \geq t_{0}$, which is a consequence of the monotonicity of $V$, leads to \eqref{dif:inq}.
\end{proof}

The following theorem provides first convergence rates which follow as a direct consequence of the previous lemma. Since $\beta$ is positive and nondecreasing, we have $\lim_{t \to + \infty} t \beta \left( t \right) = + \infty$. 
\begin{thm}
	\label{thm:rate}
	Let $z \colon \left[ t_{0} , + \infty \right) \to \sH$ be a solution of \eqref{ds} and $z_* \in \sol$. For every $t \geq t_{0}$ it holds
		\begin{subequations}
		\label{rate:V-inn}
		\begin{align}
		0 \leq \left\lVert V \left( z \left( t \right) \right) \right\rVert	& \leq \sqrt{2 \E_{\alpha - 1} \left( t_{0} \right)} \cdot \dfrac{1}{t \beta \left( t \right)} , \label{rate:V} \\	
		0 \leq \left\langle z \left( t \right) - z_{*} , V \left( z \left( t \right) \right) \right\rangle		& \leq \dfrac{\E_{\alpha - 1} \left( t_{0} \right)}{2 \left( \alpha - 1 \right)} \cdot \dfrac{1}{t \beta \left( t \right)} , \label{rate:inn}
		\end{align}
	\end{subequations}
	and the following statements are true
	\begin{subequations}
		\label{rate:int}
		\begin{align}
		\int_{t_{0}}^{+ \infty} t w \left( t \right) \left\langle z \left( t \right) - z_{*} , V \left( z \left( t \right) \right) \right\rangle dt  		& < + \infty , \label{rate:int:inn} \\
		\int_{t_{0}}^{+ \infty} t^{2} \beta \left( t \right) w \left( t \right) \left\lVert V \left( z \left( t \right) \right) \right\rVert ^{2}		dt	& < + \infty.\label{rate:int:V}
		\end{align}
	\end{subequations}
If we assume in addition that
	\begin{equation}
\label{as+:beta}
	0 \leq \sup\limits_{t \geq t_{0}} \dfrac{t \dot{\beta} \left( t \right)}{\beta \left( t \right)} < \alpha - 2 ,
	\end{equation}
	then the trajectory $t \mapsto z \left( t \right)$ is bounded, it holds
	\begin{equation}
	\label{rate:int:dz}
	\int_{t_{0}}^{+ \infty} t \left\lVert \dot{z} \left( t \right) \right\rVert ^{2} < + \infty,
	\end{equation}
	and the limit $\lim_{t \to + \infty} \E_{\lambda} \left( t \right) \in \sR$ exists for every $\lambda $ satisfying $0 \leq \lambda \leq \alpha - 1$.	
\end{thm}
\begin{proof}
First we choose $\lambda := \alpha - 1$. Then inequality \eqref{dif:inq} reduces to
	\begin{equation}
	\label{rate:alpha-1}
	\dfrac{d}{dt} \E_{\alpha - 1} \left( t \right)
	\leq - 2 \left( \alpha - 1 \right) t w \left( t \right) \left\langle z \left( t \right) - z_{*} , V \left( z \left( t \right) \right) \right\rangle - 2t^{2} \beta \left( t \right) w \left( t \right) \left\lVert V \left( z \left( t \right) \right) \right\rVert ^{2} \leq 0 \quad \forall t \geq t_0.
	\end{equation}
	This means that $t \mapsto \E_{\alpha - 1} \left( t \right)$ is nonincreasing on $[t_0,+\infty)$ and, thus ,the inequalities \eqref{rate:V-inn} follow from the definition of the energy function. In addition, after integration of \eqref{rate:alpha-1}, we obtain the statements in  \eqref{rate:int}.
	
Now we suppose that \eqref{as+:beta} holds. Then there exists $0 \leq \varepsilon < \alpha - 2$ such that
\begin{equation*}
\sup\limits_{t \geq t_{0}} \dfrac{t \dot{\beta} \left( t \right)}{\beta \left( t \right)} = \alpha - 2 - \varepsilon < \alpha - 2 .
\end{equation*}
This means that
\begin{equation}
\label{defi:w-as}
w \left( t \right) = - \dfrac{1}{2} \dot{\beta} \left( t \right) + \dfrac{1}{2} \left( \alpha - 2 \right) \dfrac{\beta \left( t \right)}{t} \geq \dfrac{\varepsilon}{2} \dfrac{\beta \left( t \right)}{t} > 0 \quad \forall t \geq t_{0}.
\end{equation}
Hence
	\begin{equation}
	\label{rate_beta-w}
	t^{2} \beta \left( t \right) w \left( t \right) \geq \dfrac{\varepsilon}{2} t \beta^{2} \left( t \right) \quad \forall t \geq t_{0},
	\end{equation}
which, due to \eqref{rate:int:V}, gives
	\begin{equation}
	\label{rate:int:V-beta}
	\int_{t_{0}}^{+ \infty} t \beta^{2} \left( t \right) \left\lVert V \left( z \left( t \right) \right) \right\rVert ^{2} dt < + \infty .
	\end{equation}

In order to prove the last statements of the theorem, we notice that the estimate \eqref{dif:inq} gives for every $0 \leq \lambda \leq \alpha-1$ and every $t \geq t_0$
	\begin{subequations}
		\begin{align}
		\dfrac{d}{dt} \E_{\lambda} \left( t \right)
		& \leq \left( \alpha - 1 - \lambda \right) t \beta^{2} \left( t \right) \left\lVert V \left( z \left( t \right) \right) \right\rVert ^{2} - \left( \alpha - 1 - \lambda \right) t \left\lVert 2 \dot{z} \left( t \right) - \beta \left( t \right) V \left( z \left( t \right) \right) \right\rVert ^{2} \nonumber \\
		& \leq 2 \left( \alpha - 1 - \lambda \right) t \beta^{2} \left( t \right) \left\lVert V \left( z \left( t \right) \right) \right\rVert ^{2} - 2 \left( \alpha - 1 - \lambda \right) t \left\lVert \dot{z} \left( t \right) \right\rVert ^{2} \label{rate:inq:pre} \\
		& \leq 2 \left( \alpha - 1 - \lambda \right) t \beta^{2} \left( t \right) \left\lVert V \left( z \left( t \right) \right) \right\rVert ^{2} . \label{rate:inq}
		\end{align}
	\end{subequations}
	The assertion \eqref{rate:int:dz} follows by integration of \eqref{rate:inq:pre} for $\lambda := 0$ and by using then \eqref{rate:int:V-beta}. Finally, as $t \mapsto t \beta^{2} \left( t \right) \left\lVert V \left( z \left( t \right) \right) \right\rVert ^{2} \in \sL^{1} \left( \left[ t_{0} , + \infty \right) \right)$, we can apply Lemma \ref{lem:lim-R} to \eqref{rate:inq} in order to obtain the existence of the limit $\lim_{t \to + \infty} \E_{\lambda} \left( t \right) \in \sR$ for every $0 \leq \lambda \leq \alpha - 1$.	
\end{proof}

The existence and uniqueness of solutions for \eqref{ds} can be guaranteed in a very general setting, which includes the one of continuously differentiable operators defined on finite-dimensional spaces, that are obviously Lipschitz continuous in bounded sets.  The proof of Theorem \ref{existence} is provided in the Appendix and it relies on showing that the maximal solution given by the Cauchy-Lipschitz Theorem is a global solution.
\begin{thm}
\label{existence}
Let $\alpha >2$ and assume that $V: \sH \to \sH$ is continuously differentiable, $\beta : [t_0,+\infty) \rightarrow (0,+\infty)$ is a continuously differentiable and nondecreasing function which satisfies condition \eqref{as+:beta}  and that $V$ and  $\dot{\beta}$ are Lipschitz continuous on bounded sets. Then for every initial condition $z \left( t_{0} \right) = z^{0} \in \sH \textrm{ and } \dot{z} \left( t_{0} \right) = \dot z^{0} \in \sH$ the dynamical system \eqref{ds} has a unique global twice continuously differentiable solution $z \colon \left[ t_{0} , + \infty \right) \to \sH$.
\end{thm}

Further we prove that, under the slightly stronger growth condition \eqref{as+:beta}, the trajectories of the dynamical system \eqref{ds} converge to a zero of $V$. This phenomenon is also present at inertial gradient systems with asymptotically vanishing damping terms, where it concerns the coefficient $\alpha$, too.
\begin{thm}
	Let $\alpha >2$ and $z \colon \left[ t_{0} , + \infty \right) \to \sH$ be a solution of \eqref{ds} and assume that $\beta \colon \left[ t_{0} , + \infty \right) \to \left( 0 , + \infty \right)$ satisfies the growth condition \eqref{as+:beta}, in other words
	\begin{equation*}
	0 \leq \sup\limits_{t \geq t_{0}} \dfrac{t \dot{\beta} \left( t \right)}{\beta \left( t \right)} < \alpha - 2 .
	\end{equation*}
	Then $z \left( t \right)$ converges weakly to a solution of \eqref{intro:pb:eq} as $t \to + \infty$.
\end{thm}
\begin{proof}
Let $z_* \in {\cal Z}$ and $0 \leq \lambda_{1} < \lambda_{2} \leq \alpha - 1$ be fixed. Then by the definition of the energy function in \eqref{defi:E} we have for every $t \geq t_{0}$
	\begin{align*}
	\E_{\lambda_{2}} \left( t \right) - \E_{\lambda_{1}} \left( t \right)
	 = \ & 2 \left( \lambda_{2} - \lambda_{1} \right) t \left\langle z \left( t \right) - z_{*} , 2 \dot{z} \left( t \right) + \beta \left( t \right) V \left( z \left( t \right) \right) \right\rangle \nonumber \\
	&+ 2 \left( \lambda_{2} - \lambda_{1} \right) \lambda \beta \left( t \right) \left\langle z \left( t \right) - z_{*} , V \left( z \left( t \right) \right) \right\rangle + 2 \left( \lambda_{2} - \lambda_{1} \right) \left( \alpha - 1 \right) \left\lVert z \left( t \right) - z_{*} \right\rVert ^{2} \nonumber \\
	= \ & 4 \left( \lambda_{2} - \lambda_{1} \right) \left( t \left\langle z \left( t \right) - z_{*} , \dot{z} \left( t \right) + \beta \left( t \right) V \left( z \left( t \right) \right) \right\rangle + \dfrac{1}{2} \left( \alpha - 1 \right) \left\lVert z \left( t \right) - z_{*} \right\rVert ^{2} \right) .
	\end{align*}
For every $t \geq t_{0}$ we define
	\begin{align}
	p \left( t \right) 	& := t \left\langle z \left( t \right) - z_{*} , \dot{z} \left( t \right) + \beta \left( t \right) V \left( z \left( t \right) \right) \right\rangle + \dfrac{1}{2} \left( \alpha - 1 \right) \left\lVert z \left( t \right) - z_{*} \right\rVert ^{2} , \label{defi:p} \\
	q \left( t \right)	& := \dfrac{1}{2} \left\lVert z \left( t \right) - z_{*} \right\rVert ^{2} + \int_{t_{0}}^{t} \beta \left( s \right) \left\langle z \left( s \right) - z_{*} , V \left( z \left( s \right) \right) \right\rangle ds . \label{defi:q}
	\end{align}
	One can easily see that for every $t \geq t_{0}$
	\begin{equation*}
	\dot{q} \left( t \right) = \left\langle z \left( t \right) - z_{*} , \dot{z} \left( t \right) + \beta \left( t \right) \left\langle z \left( t \right) - z_{*} , V \left( z \left( t \right) \right) \right\rangle \right\rangle = \left\langle z \left( t \right) - z_{*} , \dot{z} \left( t \right) + \beta \left( t \right) V \left( z \left( t \right) \right) \right\rangle ,
	\end{equation*}
	and thus
	\begin{align*}
	\left( \alpha - 1 \right) q \left( t \right) + t \dot{q} \left( t \right) = p \left( t \right) + \left( \alpha - 1 \right) \int_{t_{0}}^{t} \beta \left( s \right) \left\langle z \left( s \right) - z_{*} , V \left( z \left( s \right) \right) \right\rangle ds .
	\end{align*}
	
	Since $0 \leq \lambda_{1} < \lambda_{2} \leq \alpha - 1$, Theorem \ref{thm:rate} guarantees that $\lim_{t \to + \infty} \left\lbrace \E_{\lambda_2} \left( t \right) - \E_{\lambda_1} \left( t \right) \right\rbrace \in \sR$ exists, hence, by \eqref{defi:p},
	\begin{equation}
	\label{conv:lim-p}
	\lim\limits_{t \to + \infty} p \left( t \right) \in \sR \textrm{ exists}.
	\end{equation}	
	Furthermore, the quantity $\int_{t_{0}}^{t} \beta \left( s \right) \left\langle z \left( s \right) - z_{*} , V \left( z \left( s \right) \right) \right\rangle ds$ is nondecreasing with respect to $t$, and according to \eqref{defi:w-as} for every $t \geq t_{0}$ it holds
	\begin{equation*}
	\dfrac{\varepsilon}{2} \int_{t_{0}}^{t} \beta \left( s \right) \left\langle z \left( s \right) - z_{*} , V \left( z \left( s \right) \right) \right\rangle ds \leq \int_{t_{0}}^{t} s w \left( s \right) \left\langle z \left( s \right) - z_{*} , V \left( z \left( s \right) \right) \right\rangle ds .
	\end{equation*}	
	As a consequence, we conclude from \eqref{rate:int:inn} that
	\begin{equation}
	\label{conv:lim-in}
	\lim\limits_{t \to + \infty} \int_{t_{0}}^{t} \beta \left( s \right) \left\langle z \left( s \right) - z_{*} , V \left( z \left( s \right) \right) \right\rangle ds \in \sR .
	\end{equation}
	Combining \eqref{conv:lim-p} and \eqref{conv:lim-in}, it yields that the limit $\lim_{t \to + \infty} \left\lbrace \left( \alpha - 1 \right) q \left( t \right) + t \dot{q} \left( t \right) \right\rbrace \in \sR$ exists, which, according to Lemma \ref{lem:lim-u}, guarantees that $\lim_{t \to + \infty} q \left( t \right) \in \sR$. Using the definition of $q$ in \eqref{defi:q} and once again the statement \eqref{conv:lim-in}, we see that $\lim_{t \to + \infty} \left\lVert z \left( t \right) - z_{*} \right\rVert \in \sR$. This proves the hypothesis \ref{lem:Opial:cont:i} of Opial’s Lemma (see Lemma \ref{lem:Opial:cont}).

Finally, let $\widebar{z}$ be a weak sequential cluster point of the trajectory $z \left( t \right)$ as $t \to + \infty$. This means that there exists a sequence $\left( z \left( t_{n} \right) \right) _{n \geq 0}$ such that
	\begin{equation*}
	z \left( t_{n} \right) \rightharpoonup \widebar{z} \textrm{ as } n \to + \infty,
	\end{equation*}
where $\rightharpoonup$ denotes weak convergence. On the other hand, Theorem \ref{thm:rate} ensures that
	\begin{equation*}
	V \left( z \left( t_{n} \right) \right) \to 0 \textrm{ as } n \to + \infty .
	\end{equation*}
	Since $V$ is monotone and continuous, it is maximally monotone (see, for instance, \cite[Corollary 20.28]{Bauschke-Combettes:book}). Therefore, the graph of $V$ is sequentially closed in $\sH^{\textrm{weak}} \times \sH^{\textrm{strong}}$, which means that $V(\widebar{z})=0$.  In other words, the hypothesis \ref{lem:Opial:cont:ii} of Opial’s Lemma also holds, and the proof is complete.
\end{proof}

Next we will see that under the growth condition \eqref{as+:beta} the convergence rates obtained in Theorem \ref{thm:rate} can be improved from $\bO$ to $o$, which is also a phenomenon known for inertial gradient systems with asymptotically vanishing damping terms.
\begin{thm}
	Let $\alpha >2$  and $z \colon \left[ t_{0} , + \infty \right) \to \sH$ be a solution of \eqref{ds}, $z_* \in {\cal Z}$, and assume that $\beta : [t_0,+\infty) \rightarrow (0,+\infty)$ satisfies the growth condition \eqref{as+:beta}, in other words
	\begin{equation*}
	0 \leq \sup\limits_{t \geq t_{0}} \dfrac{t \dot{\beta} \left( t \right)}{\beta \left( t \right)} < \alpha - 2 .
	\end{equation*}
Then it holds
\begin{equation*}
	\left\lVert \dot{z} \left( t \right) \right\rVert = o \left( \dfrac{1}{t} \right) \quad \textrm{ as } t \to + \infty ,
\end{equation*}
and
\begin{equation*}
 \left\langle z \left( t \right) - z_{*} , V \left( z \left( t \right) \right) \right\rangle = o \left( \dfrac{1}{t \beta \left( t \right)} \right)  \textrm{ and } \left\lVert V \left( z \left( t \right) \right) \right\rVert = o \left( \dfrac{1}{t \beta \left( t \right)} \right)  \quad \textrm{ as } t \to + \infty .
\end{equation*}
\end{thm}
\begin{proof}
	For every $0 \leq \lambda \leq \alpha - 1$ the energy function of the system can be written as
	\begin{align*}
		\E_{\lambda} \left( t \right)
		 = & \ \dfrac{1}{2} \left\lVert 2 \lambda \left( z \left( t \right) - z_{*} \right) + t \Bigl( 2 \dot{z} \left( t \right) + \beta \left( t \right) V \left( z \left( t \right) \right) \Bigr) \right\rVert ^{2} + 2 \lambda \left( \alpha - 1 - \lambda \right) \left\lVert z \left( t \right) - z_{*} \right\rVert ^{2} \nonumber \\
		& + 2 \lambda t \beta \left( t \right) \left\langle z \left( t \right) - z_{*} , V \left( z \left( t \right) \right) \right\rangle + \dfrac{1}{2} t^{2} \beta^{2} \left( t \right) \left\lVert V \left( z \left( t \right) \right) \right\rVert ^{2} \nonumber \\
		 = & \ 2 \lambda \left( \alpha - 1 \right) \left\lVert z \left( t \right) - z_{*} \right\rVert ^{2} + 4 \lambda t \left\langle z \left( t \right) - z_{*} , \dot{z} \left( t \right) + \beta \left( t \right) V \left( z \left( t \right) \right) \right\rangle \nonumber \\
		& + \dfrac{1}{2} t^{2} \left\lVert 2 \dot{z} \left( t \right) + \beta \left( t \right) V \left( z \left( t \right) \right) \right\rVert ^{2} + \dfrac{1}{2} t^{2} \beta^{2} \left( t \right) \left\lVert V \left( z \left( t \right) \right) \right\rVert ^{2} \nonumber \\
		 = &  \ 4 \lambda p \left( t \right) + t^{2} \left\lVert \dot{z} \left( t \right) + \beta \left( t \right) V \left( z \left( t \right) \right) \right\rVert ^{2} + t^{2} \left\lVert \dot{z} \left( t \right) \right\rVert ^{2} ,
	\end{align*}
where the last equation comes from the definition of $p \left( t \right)$ in \eqref{defi:p} and the formula
\begin{equation}
	\label{pre:sum-2}
	\left\lVert x \right\rVert ^{2} + \left\lVert y \right\rVert ^{2} = \frac{1}{2} \left( \left\lVert x + y \right\rVert ^{2} + \left\lVert x - y \right\rVert ^{2} \right) \quad \forall x, y \in \sH .
\end{equation}
Recalling that as both limits $\lim_{t \to + \infty} \E_{\lambda} \left( t \right) \in \sR$ and $\lim_{t \to + \infty} p \left( t \right) \in \sR$ exist (see Theorem \ref{thm:rate} and  \eqref{conv:lim-p}),  we conclude that for $h:[t_0,+\infty)\rightarrow {\mathbb R},  h(t) = t^{2} \left\lVert \dot{z} \left( t \right) + \beta \left( t \right) V \left( z \left( t \right) \right) \right\rVert ^{2} + t^{2} \left\lVert \dot{z} \left( t \right) \right\rVert ^{2}$, 
\begin{equation}\label{exists-lim-h}
	\lim\limits_{t \to + \infty} h \left( t \right) \in [0,+\infty) \textrm{ exists} .
\end{equation}
Moreover, from \eqref{rate:int:dz} and \eqref{rate:int:V-beta}, we see that
\begin{align*}
\int_{t_{0}}^{+ \infty} \dfrac{1}{t} h \left( t \right) dt \leq 3 \int_{t_{0}}^{+ \infty} t \left\lVert \dot{z} \left( t \right) \right\rVert ^{2} dt + 2 \int_{t_{0}}^{+ \infty} t \beta^{2} \left( t \right) \left\lVert V \left( z \left( t \right) \right) \right\rVert ^{2} dt < + \infty ,
\end{align*}
which in combination with \eqref{exists-lim-h} leads to $\lim_{t\rightarrow+\infty}h(t)=0$. Thus
\begin{equation*}
\lim\limits_{t \to + \infty} t \left\lVert \dot{z} \left( t \right) + \beta \left( t \right) V \left( z \left( t \right) \right) \right\rVert = \lim\limits_{t \to + \infty} t \left\lVert \dot{z} \left( t \right) \right\rVert = 0 ,
\end{equation*}
and, consequently,
\begin{equation*}
\lim\limits_{t \to + \infty} t \beta \left( t \right) \left\lVert V \left( z \left( t \right) \right) \right\rVert = 0 .
\end{equation*}
Finally, by Cauchy-Schwarz inequality and the fact that the trajectory $t \mapsto z(t)$ is bounded, we deduce that
\begin{equation*}
0 \leq t \beta \left( t \right) \left\langle z \left( t \right) - z_{*} , V \left( z \left( t \right) \right) \right\rangle \leq t \beta \left( t \right) \left\lVert z \left( t \right) - z_{*} \right\rVert \left\lVert V \left( z \left( t \right) \right) \right\rVert \quad \forall t \geq t_{0} ,
\end{equation*}
which finishes the proof.
\end{proof}

\begin{rmk}\label{rmk:intuition}
One of the anonymous referees made an excellent observation regarding the asymptotic behaviour of the trajectories on which we will elaborate in the following.
For the first order system attached to \eqref{intro:pb:eq}
\begin{equation}
	\label{ds:fo}
	\dot{u} \left( t \right) + V \left( u \left( t \right) \right) = 0,
\end{equation}
it is known that the solution trajectories converge weakly in ergodic (averaged) sense towards a zero of $V$. In other words, there exists $ z_{*} \in \sol$ such that $z \left( t \right) := \frac{1}{t} \int_{0}^{t} u \left( s \right) ds \rightharpoonup z_{*} \in \sol$ as $t \to + \infty$ (see, for instance, \cite{Baillon-Brezis,Peypouquet-Sorin}).

This leads to the natural idea of considering the averaging trajectory $z$, that fulfills
\begin{equation}
	\label{defi:u}
	\dot{z} \left( t \right) + \dfrac{1}{t} \left( z \left( t \right) - u \left( t \right) \right) = 0,
\end{equation}
and to drive the equation of its dynamics from \eqref{ds:fo}.  For more details on this very powerful approach we refer the reader to \cite{Attouch-Bot-Nguyen}.

From \eqref{defi:u} we deduce that $\dot{u} \left( t \right) = t \ddot{z} \left( t \right) + 2 \dot{z} \left( t \right)$, hence equation \eqref{ds:fo}  becomes
\begin{equation*}
	t \ddot{z} \left( t \right) + 2 \dot{z} \left( t \right) + V \left( z \left( t \right) + t \dot{z} \left( t \right) \right) = 0.
\end{equation*}
Taking the Taylor expansion
\begin{equation*}
	V \left( z \left( t \right) + t \dot{z} \left( t \right) \right) \approx V \left( z \left( t \right) \right) + t \nabla V \left( z \left( t \right) \right) \dot z \left( t \right) = V \left( z \left( t \right) \right) + t \dfrac{d}{dt}  
V \left( z \left( t \right) \right),
\end{equation*}
it leads to the second-order dynamical system with correction term $\frac{d}{dt} V \left( z \left( t \right) \right)$
\begin{equation*}
	\ddot{z} \left( t \right) + \dfrac{2}{t} \dot{z} \left( t \right) + \dfrac{d}{dt} \left( V \left( z \left( t \right) \right) \right) + \dfrac{1}{t} V \left( z \left( t \right) \right) = 0,
\end{equation*}
which is of the same type as \eqref{ds}.  This approach suggests that one can expect the non-ergodic convergence of the solution trajectory of \eqref{ds} to a zero of $V$. 

The function $\beta$ can be ``inserted'' into the system through time scaling approaches aimed to speed up its convergence behaviour (see also \cite{Attouch-Bot-Nguyen, Attouch-Chbani-Fadili-Riahi:20,Attouch-Chbani-Fadili-Riahi:23, Attouch-Chbani-Riahi:SIOPT} for related ideas).
\end{rmk}

\section{An implicit numerical algorithm}\label{sec3}

In this section we formulate and investigate an implicit type numerical algorithm which follows from a temporal discretization of the dynamical system \eqref{ds}. We recall that the latter can be equivalently written as (see the proof of Theorem \ref{existence})
\begin{equation}
\label{dis:ds-fo}
\begin{dcases}
\dot{u} \left( t \right)	& = \Bigl( t \dot{\beta} \left( t \right) + \left( 2 - \alpha \right) \beta \left( t \right) \Bigr) V \left( z \left( t \right) \right) \\
u \left( t \right) 			& = 2 \left( \alpha - 1 \right) z \left( t \right) + 2t \dot{z} \left( t \right) + 2t \beta \left( t \right) V \left( z \left( t \right) \right)
\end{dcases} ,
\end{equation}
with the initializations $z \left( t_{0} \right) = z^{0} \textrm{ and } \dot{z} \left( t_{0} \right) = \dot z^{0}$.

We fix a time step $s > 0$, set $\tau_{k} := s \left( k + 1 \right)$ and $\sigma_{k} := sk$ for every $k \geq 1$,  and approximate $z \left( \tau_{k} \right) \approx z^{k+1}$, $u \left( \tau_{k} \right) \approx u^{k+1}$, and $\beta \left( \sigma_{k} \right) \approx \beta_{k}$. The implicit finite-difference scheme for \eqref{dis:ds-fo} at time $t := \tau_{k}$ for $\left( z , u \right)$ and at time $t := \sigma_{k}$ for $\beta$ gives for every $k \geq 1$
\begin{equation}
\label{dis:fd-fo}
\begin{dcases}
\dfrac{u^{k+1} - u^{k}}{s} 	& = \Bigl( k \left( \beta_{k} - \beta_{k-1} \right) + \left( 2 - \alpha \right) \beta_{k} \Bigr) V \left( z^{k+1} \right) \\
u^{k+1} 					& = 2 \left( \alpha - 1 \right) z^{k+1} + 2 \left( k + 1 \right) \left( z^{k+1} - z^{k} \right) + 2s \left( k + 1 \right) \beta_{k} V \left( z^{k+1} \right)
\end{dcases} ,
\end{equation}
with the initialization $u^{1} := z^{0}$ and $u^{0} := z^{0} - s \dot z^{0}$. Therefore we have for every $k \geq 1$
\begin{equation*}
u^{k} = 2 \left( \alpha - 1 \right) z^{k} + 2k \left( z^{k} - z^{k-1} \right) + 2sk \beta_{k-1} V \left( z^{k} \right),
\end{equation*}
and after substraction we get
\begin{align}
u^{k+1} - u^{k}
= & \ 2 \left( k + \alpha \right) \left( z^{k+1} - z^{k} \right) - 2k \left( z^{k} - z^{k-1} \right) + 2s \Bigl( \left( k + 1 \right) \beta_{k} - k \beta_{k-1} \Bigr) V \left( z^{k+1} \right) \nonumber \\
& + 2s k \beta_{k-1} \left( V \left( z^{k+1} \right) - V \left( z^{k} \right) \right) \nonumber \\
= &  \ s \Bigl( k \left( \beta_{k} - \beta_{k-1} \right) + \left( 2 - \alpha \right) \beta_{k} \Bigr) V \left( z^{k+1} \right) , \label{dis:d-u}
\end{align}
where the last relation comes from the first equation in \eqref{dis:fd-fo}.
From here, we deduce that for every $k \geq 1$
\begin{align*}	
	z^{k+1}
	 = & \ z^{k} + \left( 1 - \dfrac{\alpha}{k + \alpha} \right) \left( z^{k} - z^{k-1} \right) - \dfrac{s \left( \alpha \beta_{k} + k \left( \beta_{k} - \beta_{k-1} \right) \right)}{2 \left( k + \alpha \right)} V \left( z^{k+1} \right) \nonumber \\
	& - \dfrac{sk \beta_{k-1}}{k + \alpha} \left( V \left( z^{k+1} \right) - V \left( z^{k} \right) \right). 
\end{align*}

For
$$s_{k} :=  \dfrac{s \left( \alpha \beta_{k} + k \left( \beta_{k} - \beta_{k-1} \right) \right)}{2 \left( k + \alpha \right)} \quad \mbox{and} \quad t_{k}  := \dfrac{sk \beta_{k-1}}{k + \alpha},$$
the algorithm can be further equivalently written as
\begin{align*}
z^{k+1}	 := (\Id + ( s_{k} + t_{k}) V)^{-1} \left( z^{k} + \left( 1 - \dfrac{\alpha}{k + \alpha} \right) \left( z^{k} - z^{k-1} \right) + t_{k} V \left( z^{k} \right) \right) \quad \forall k \geq 1,
\end{align*}
and is therefore well-defined due to the maximal monotonicity of $V$.

We also want to point out that the discrete version of the growth condition \eqref{as+:beta} reads
\begin{equation*}
	0 \leq \sup_{k \geq 1} \dfrac{k \left( \beta_{k} - \beta_{k-1} \right)}{\beta_{k}} < \alpha - 2,
\end{equation*}
where $\left( \beta_{k} \right)_{k \geq 0}$ is a positive and nondecreasing sequence. This means that there exists some $0 \leq \varepsilon < \alpha - 2$ such that 
\begin{equation}
\label{as:beta-k:eps}
\dfrac{k \left( \beta_{k} - \beta_{k-1} \right)}{\beta_{k}} \leq \alpha - 2 - \varepsilon \ \mbox{or, equivalently,} \ k \left( \beta_{k} - \beta_{k-1} \right) \leq \left( \alpha - 2 - \varepsilon \right) \beta_{k} \quad \forall k \geq 1.
\end{equation}
In addition, for every $k \geq \left\lceil \alpha \right\rceil$ it holds
\begin{equation}
\label{as:beta-k}
\beta_{k} \leq \dfrac{k}{k + 2 + \varepsilon - \alpha} \beta_{k-1} \leq \dfrac{\alpha}{2 + \varepsilon} \beta_{k-1} .
\end{equation}

To sum up,  the implicit algorithm we propose for solving \eqref{intro:pb:eq} is formulated below.
\begin{mdframed}
	\begin{algo} {\bf (Implicit Fast OGDA)}
		\label{algo:im}
		Let $\alpha >2, z^{0}, z^{1} \in \sH$, $s > 0$, and $\left(\beta_{k} \right) _{k \geq 0}$ a positive and nondecreasing sequence which satisfies
		\begin{equation}\label{betak+}
		0 \leq \sup_{k \geq 1} \dfrac{k \left( \beta_{k} - \beta_{k-1} \right)}{\beta_{k}} < \alpha - 2 .
		\end{equation}
		For every $k \geq 1$ we set
		\begin{align*}	
	z^{k+1}
	= & \ z^{k} + \left( 1 - \dfrac{\alpha}{k + \alpha} \right) \left( z^{k} - z^{k-1} \right) - \dfrac{s \left( \alpha \beta_{k} + k \left( \beta_{k} - \beta_{k-1} \right) \right)}{2 \left( k + \alpha \right)} V \left( z^{k+1} \right) \nonumber \\
	&  - \dfrac{sk \beta_{k-1}}{k + \alpha} \left( V \left( z^{k+1} \right) - V \left( z^{k} \right) \right) .
       \end{align*}
		\end{algo}
\end{mdframed}
Inspired by the continuous setting, we consider for $0 \leq \lambda \leq \alpha-1$ the following sequence defined for every $k \geq 1$
\begin{align*}
	\E_{\lambda}^{k}
	 := & \ \dfrac{1}{2} \left\lVert 2 \lambda \left( z^{k} - z_{*} \right) + 2k \left( z^{k} - z^{k-1} \right) + sk \beta_{k-1} V \left( z^{k} \right) \right\rVert ^{2}
	+ 2 \lambda \left( \alpha - 1 - \lambda \right) \left\lVert z^{k} - z_{*} \right\rVert ^{2} \nonumber \\
	& + 2 \lambda sk \beta_{k-1} \left\langle z^{k} - z_{*} , V \left( z^{k} \right) \right\rangle
	+ \dfrac{1}{2} s^{2} \left( k + \alpha \right) k \beta_{k} \beta_{k-1} \left\lVert V \left( z^{k} \right) \right\rVert ^{2} \geq 0,
\end{align*}
which is the discrete version of the energy function considered in the previous section. We have for every $k \geq 1$
	\begin{align}
	\E_{\lambda}^{k} = & \ 2 \lambda \left( \alpha - 1 \right) \left\lVert z^{k} - z_{*} \right\rVert ^{2} + 4 \lambda k \left\langle z^{k} - z_{*} , z^{k} - z^{k-1} + s \beta_{k-1} V \left( z^{k} \right) \right\rangle \nonumber \\
	& + \dfrac{1}{2} k^{2} \left\lVert 2 \left( z^{k} - z^{k-1} \right) + s \beta_{k-1} V \left( z^{k} \right) \right\rVert ^{2} + \dfrac{1}{2} s^{2} \left( k + \alpha \right) k \beta_{k} \beta_{k-1} \left\lVert V \left( z^{k} \right) \right\rVert ^{2} \label{im:defi:E-k:eq} .
\end{align}

The following lemma shows that the discrete energy dissipates with every iteration of the algorithm.  Its proof can be found in the Appendix.  Lemma \ref{lemma8} is the essential ingredient for the derivation of the convergence rates in Theorem \ref{prop:im:lim}.

\begin{lem}\label{lemma8}
	Let $z_* \in {\cal Z}$ and $\left(z^{k} \right)_{k \geq 0}$ the sequence generated by Algorithm \ref{algo:im} for $\left(\beta_{k} \right) _{k \geq 0}$ a positive and nondecreasing sequence which satisfies \eqref{betak+}.  Then for every $0 \leq \lambda \leq \alpha-1$ and every $k \geq \left\lceil \alpha \right\rceil$ it holds
	\begin{align}
		\E_{\lambda}^{k+1} - \E_{\lambda}^{k}
		& \leq 2 \lambda s \Bigl( \left( k + 2 - \alpha \right) \beta_{k} - k \beta_{k-1} \Bigr) \left\langle z^{k+1} - z_{*} , V \left( z^{k+1} \right) \right\rangle \nonumber \\
		& \quad + 2 \left( \lambda + 1 - \alpha \right) \left( 2k + \alpha + 1 \right) \left\lVert z^{k+1} - z^{k} \right\rVert ^{2} \nonumber \\
		& \quad + 2s \biggl( \Bigl( \left( \lambda + 1 - \alpha \right) \left( 2k + \alpha + 1 \right) - \lambda \Bigr) \beta_{k} - \lambda k \left( \beta_{k} - \beta_{k-1} \right) \biggr) \left\langle z^{k+1} - z^{k} , V \left( z^{k+1} \right) \right\rangle \nonumber \\
		& \quad - 2sk \left( k + \alpha \right) \beta_{k-1} \left\langle z^{k+1} - z^{k} , V \left( z^{k+1} \right) - V \left( z^{k} \right) \right\rangle \nonumber \\
		& \quad + \dfrac{1}{2} \Bigl( \CiV - \varepsilon \left( 2k + \alpha + 1 \right) \Bigr) s^{2} \beta_{k}^{2} \left\lVert V \left( z^{k+1} \right) \right\rVert ^{2} \nonumber \\
		& \quad - \dfrac{1}{2} s^{2} k \Bigl( \left( k + \alpha \right) \beta_{k} + k \beta_{k-1} \Bigr) \beta_{k-1} \left\lVert V \left( z^{k+1} \right) - V \left( z^{k} \right) \right\rVert ^{2} , \label{dec:inq}
	\end{align}
where
\begin{equation}\label{C}
	\CiV := \dfrac{\alpha}{2 + \varepsilon} \left( \alpha - 2 - \varepsilon \right) \left( 2 \alpha - 2 - \varepsilon \right) > 0
\end{equation}
and $\varepsilon$ is chosen to fulfill \eqref{as:beta-k:eps}.
\end{lem}

\begin{thm}\label{prop:im:lim}
Let $z_* \in {\cal Z}$ and $\left(z^{k} \right)_{k \geq 0}$ the sequence generated by Algorithm \ref{algo:im} for $\left(\beta_{k} \right) _{k \geq 0}$ a positive and nondecreasing sequence which satisfies \eqref{betak+}, and $0 \leq \varepsilon < \alpha-2$ be such that \eqref{as:beta-k:eps} is satisfied. Then it holds
	\begin{equation*}
 \left\langle z^{k} - z_{*} , V \left( z^{k} \right) \right\rangle = \bO \left( \dfrac{1}{k \beta_{k}} \right) \textrm{ and } \left\lVert V \left( z^{k} \right) \right\rVert = \bO \left( \dfrac{1}{k \beta_{k}} \right)  \textrm{ as } k \to + \infty .
	\end{equation*}
In addition, for every $\alpha - 1 - \frac{\varepsilon}{4} < \lambda < \alpha - 1$, the sequence $\left(\E_{\lambda}^{k} \right) _{k \geq 1}$ converges, $\left( z^{k} \right) _{k \geq 0}$ is bounded and
	\begin{subequations}
		\label{rate:sum}
		\begin{align}		
			\mysum_{k \geq 1} \beta_{k} \left\langle z^{k+1} - z_{*} , V \left( z^{k+1} \right) \right\rangle < + \infty ,  \\
			\mysum_{k \geq 1} k \left\lVert z^{k+1} - z^{k} \right\rVert ^{2} < + \infty , \\
			\mysum_{k \geq 1} k \beta_{k}^{2} \left\lVert V \left( z^{k+1} \right) \right\rVert ^{2} < + \infty .
		\end{align}
	\end{subequations}
\end{thm}
\begin{proof}
Let $0 < \alpha - 1 - \frac{\varepsilon}{4} < \lambda < \alpha - 1$. First we show that for sufficiently large $k$ it holds
	\begin{align}
		R_{k} := & \left( \lambda + 1 - \alpha \right) \left( 2k + \alpha + 1 \right) \left\lVert z^{k+1} - z^{k} \right\rVert ^{2} \nonumber \\
		& + 2s \biggl( \Bigl( \left( \lambda + 1 - \alpha \right) \left( 2k + \alpha + 1 \right) - \lambda \Bigr) \beta_{k} - \lambda k \left( \beta_{k} - \beta_{k-1} \right) \biggr) \left\langle z^{k+1} - z^{k} , V \left( z^{k+1} \right) \right\rangle \nonumber \\
		& + \dfrac{1}{4} \Bigl( \CiV - \varepsilon \left( 2k + \alpha + 1 \right) \Bigr) s^{2} \beta_{k}^{2} \left\lVert V \left( z^{k+1} \right) \right\rVert ^{2} \leq 0, \label{lemma9}
	\end{align}
where $\CiV>0$ is given by \eqref{C}.
By setting $K_{\alpha} := 2k + \alpha + 1 \geq 1$, for every $k \geq 0$ we have
	\begin{align*}
		R_{k} = & \left( \lambda + 1 - \alpha \right) K_{\alpha} \left\lVert z^{k+1} - z^{k} \right\rVert ^{2} + \dfrac{1}{4} s^{2} \left( \CiV - \varepsilon K_{\alpha} \right) \beta_{k}^{2} \left\lVert V \left( z^{k+1} \right) \right\rVert ^{2} \nonumber \\
		& + 2s \biggl( \Bigl( \left( \lambda + 1 - \alpha \right) K_{\alpha} - \lambda \Bigr) \beta_{k} - \lambda k \left( \beta_{k} - \beta_{k-1} \right) \biggr) \left\langle z^{k+1} - z^{k} , V \left( z^{k+1} \right) \right\rangle .
	\end{align*}
	To guarantee that $R_{k} \leq 0$ for sufficiently large $k$, we show that
	\begin{equation*}
		\dfrac{\Delta_{k}}{s^{2}}
		:= 4 \biggl( \Bigl( \left( \lambda + 1 - \alpha \right) K_{\alpha} - \lambda \Bigr) \beta_{k} - \lambda k \left( \beta_{k} - \beta_{k-1} \right) \biggr) ^{2} - \left( \lambda + 1 - \alpha \right) \left( \CiV - \varepsilon K_{\alpha} \right) K_{\alpha} \beta_{k}^{2} \leq 0
	\end{equation*}
sufficiently large $k$. Since $\left(\beta_{k} \right)_{k \geq 0}$ is nondecreasing and $\lambda < \alpha - 1$, it follows from \eqref{as:beta-k:eps} that for every $k \geq 1$
	\begin{align*}
		0 \geq & \Bigl( \left( \lambda + 1 - \alpha \right) K_{\alpha} - \lambda \Bigr) \beta_{k} - \lambda k \left( \beta_{k} - \beta_{k-1} \right) \geq \Bigl( \left( \lambda + 1 - \alpha \right) K_{\alpha} - \lambda \Bigr) \beta_{k} - \lambda \left( \alpha - 2 - \varepsilon \right) \beta_{k} \nonumber \\
		= & \Bigl( \left( \lambda + 1 - \alpha \right) K_{\alpha} - \lambda \left( \alpha - 1 - \varepsilon \right) \Bigr) \beta_{k} ,
	\end{align*}
	and thus
	\begin{align*}
		\dfrac{\Delta_{k}}{s^{2} \beta_{k}^{2}}
		 := & \dfrac{4}{\beta_{k}^{2}} \Bigl( \Bigl( \left( \lambda + 1 - \alpha \right) K_{\alpha} - \lambda \Bigr) \beta_{k} - \lambda k \left( \beta_{k} - \beta_{k-1} \right) \Bigr) ^{2} - \left( \lambda + 1 - \alpha \right) \left( \CiV - \varepsilon K_{\alpha} \right) K_{\alpha} \nonumber \\
		 \leq & \ 4 \Bigl( \left( \lambda + 1 - \alpha \right) K_{\alpha} - \lambda \left( \alpha - 1 - \varepsilon \right) \Bigr) ^{2} - \left( \lambda + 1 - \alpha \right) \left( \CiV - \varepsilon K_{\alpha} \right) K_{\alpha} \nonumber \\
		 = & \ 4 \left( \lambda + 1 - \alpha \right) ^{2} K_{\alpha}^{2} - 8 \lambda \left( \lambda + 1 - \alpha \right) \left( \alpha - 1 - \varepsilon \right) K_{\alpha} + 4 \lambda^{2} \left( \alpha - 1 - \varepsilon \right)^{2} \nonumber \\
		& - \left( \lambda + 1 - \alpha \right) \CiV K_{\alpha} + \varepsilon \left( \lambda + 1 - \alpha \right) K_{\alpha}^{2} \nonumber \\
		= & \left( \lambda + 1 - \alpha \right) \!\Bigl( 4 \left( \lambda + 1 - \alpha \right) + \varepsilon \Bigr) K_{\alpha}^{2} - \!\left( \lambda + 1 - \alpha \right) \Bigl( 8 \lambda \left( \alpha - 1 - \varepsilon \right) + \CiV \Bigr) K_{\alpha} + 4 \lambda^{2} \left( \alpha - 1 - \varepsilon \right)^{2} .
	\end{align*}
Since $\alpha - 1 - \frac{\varepsilon}{4} < \lambda < \alpha - 1$, we have $\left( \lambda + 1 - \alpha \right) \Bigl( 4 \left( \lambda + 1 - \alpha \right) + \varepsilon \Bigr) < 0$, hence for sufficiently large $k \geq 0$ it holds $\Delta_k \leq 0$ and, consequently, $R_k \geq 0$.

From \eqref{as:beta-k:eps} we deduce that $\left( k + 2 - \alpha \right) \beta_{k} - k \beta_{k-1} \leq - \varepsilon \beta_{k}$ for every $k \geq 1$. Hence,  for every $\alpha - 1 - \frac{\varepsilon}{4} < \lambda < \alpha - 1$, from Lemma \ref{lemma8} and \eqref{lemma9} we have that  for sufficiently large $k$ it holds
	\begin{align*}
		\E_{\lambda}^{k+1} - \E_{\lambda}^{k}
		\leq & - \varepsilon 2 \lambda s \beta_{k} \left\langle z^{k+1} - z_{*} , V \left( z^{k+1} \right) \right\rangle
		+ \left( \lambda + 1 - \alpha \right) \left( 2k + \alpha + 1 \right) \left\lVert z^{k+1} - z^{k} \right\rVert ^{2} \nonumber \\
		& - 2sk \left( k + \alpha \right) \beta_{k-1} \left\langle z^{k+1} - z^{k} , V \left( z^{k+1} \right) - V \left( z^{k} \right) \right\rangle \nonumber \\
		& + \dfrac{1}{4} \Bigl( \CiV - \varepsilon \left( 2k + \alpha + 1 \right) \Bigr) s^{2} \beta_{k}^{2} \left\lVert V \left( z^{k+1} \right) \right\rVert ^{2} \nonumber \\
		& - \dfrac{1}{2} s^{2} k \Bigl( \left( k + \alpha \right) \beta_{k} + k \beta_{k-1} \Bigr) \beta_{k-1} \left\lVert V \left( z^{k+1} \right) - V \left( z^{k} \right) \right\rVert ^{2},
	\end{align*}
	which means the sequence $\left\lbrace \E_{\lambda}^{k} \right\rbrace _{k \geq 1}$ is nonincreasing for sufficiently large $k$, thus it is convergent and the boundedness of $\left( z^{k} \right) _{k \geq 0}$ and the convergence rates follow from the definition of $\E_{\lambda}^{k}$ and \eqref{as:beta-k}. The remaining assertions follow from Lemma \ref{lem:quasi-Fej}.
\end{proof}
Next we prove the weak convergence of the generated sequence of iterates.
\begin{thm}
	\label{thm:ex:conv}
	Let $z_* \in {\cal Z}$ and $\left(z^{k} \right)_{k \geq 0}$ the sequence generated by Algorithm \ref{algo:im} for $\left(\beta_{k} \right) _{k \geq 0}$ a positive and nondecreasing sequence which satisfies \eqref{betak+}.   Then the sequence $\left(z^{k} \right) _{k \geq 0}$ converges weakly to a solution of \eqref{intro:pb:eq}.
\end{thm}
\begin{proof}
Let $0 \leq \varepsilon < \alpha-2$ such that \eqref{as:beta-k:eps} is satisfied and $0 < \alpha - 1 - \frac{\varepsilon}{4} < \lambda_{1} < \lambda_{2} < \alpha - 1$.  For every $k \geq 1$ we set
\begin{align}
	p_{k}	& := \dfrac{1}{2} \left( \alpha - 1 \right) \left\lVert z^{k} - z_{*} \right\rVert ^{2} + k \left\langle z^{k} - z_{*} , z^{k} - z^{k-1} + s \beta_{k-1} V \left( z^{k} \right) \right\rangle , \label{defi:im:p-k} \\
	q_{k}	& := \dfrac{1}{2} \left\lVert z^{k} - z_{*} \right\rVert ^{2} + s \mysum_{i = 1}^{k} \beta_{i-1} \left\langle z^{i} - z_{*} , V \left( z^{i} \right) \right\rangle, \label{defi:im:q-k}
\end{align}
and notice that
\begin{equation*}
	\left( \alpha - 1 \right) q_{k} + k \left( q_{k} - q_{k-1} \right) = p_{k} + \left( \alpha - 1 \right) s \mysum_{i = 1}^{k+1} \beta_{i-1} \left\langle z^{i} - z_{*} , V \left( z^{i} \right) \right\rangle - \dfrac{k}{2} \left\lVert z^{k} - z^{k-1} \right\rVert ^{2} .
\end{equation*}
We have that
\begin{equation}
\label{conv:lim-p-k}
\lim\limits_{k \to + \infty} p_{k} = \lim_{k \to + \infty} \frac{1}{4(\lambda_2-\lambda_1)}\left(\E_{\lambda_{2}}^{k} - \E_{\lambda_{1}}^{k} \right)  \in \sR \textrm{ exists}
\end{equation}	
and,  thanks to \eqref{rate:sum},  that the limit $\lim_{k \to + \infty} \sum_{i = 1}^{k+1} \beta_{i-1} \left\langle z^{i} - z_{*} , V \left( z^{i} \right) \right\rangle \in \sR$ exists and
\begin{equation*}
	\lim\limits_{k \to + \infty} k \left\lVert z^{k+1} - z^{k} \right\rVert ^{2} = 0 .
\end{equation*}
Consequently,
\begin{equation*}
	\lim\limits_{k \to + \infty} \left( \left( \alpha - 1 \right) q_{k} + k \left( q_{k} - q_{k-1} \right) \right) \in \sR \textrm{ exists} .
\end{equation*}
From Theorem \ref{prop:im:lim} we deduce that $\left( q_{k} \right)_{k \geq 1}$ is bounded. This allows us to apply Lemma \ref{lem:lim-u-k} and to conclude from here that $\lim_{k \to + \infty} q_{k} \in \sR$ also exists. Once again,  by the definition of $q_{k}$ and the fact that the sequence $\left(\sum_{i = 1}^{k} \beta_{i-1} \left\langle z^{i} - z_{*} , V \left( z^{i} \right) \right\rangle \right)_{k \geq 1}$ converges, it follows that $\lim_{k \to + \infty} \left\lVert z_{k} - z_{*} \right\rVert \in \sR$ exists. In other words, the hypothesis \ref{lem:Opial:dis:i} in Opial’s Lemma (see Lemma \ref{lem:Opial:dis}) is fulfilled.

Now let $\widebar{z}$ be a weak sequential cluster point of $\left(z^{k} \right) _{k \geq 0}$, meaning that there exists a subsequence $\left(z^{k_{n}} \right)_{n \geq 0}$ such that
\begin{equation*}
	z^{k_{n}} \rightharpoonup \widebar{z} \textrm{ as } n \to + \infty .
\end{equation*}
From Theorem \ref{prop:im:lim} we have
\begin{equation*}
	V \left( z^{k_{n}} \right) \to 0 \textrm{ as } n \to + \infty .
\end{equation*}
Since $V$ monotone and continuous, it s maximally monotone \cite[Corollary 20.28]{Bauschke-Combettes:book}. Therefore, the graph of $V$ is sequentially closed in $\sH^{\textrm{weak}} \times \sH^{\textrm{strong}}$, which gives that $V(\widebar{z}) =0$, thus $\widebar{z} \in \sol$.
This shows that hypothesis \ref{lem:Opial:dis:ii} of Opial’s Lemma is also fulfilled, and completes the proof.
\end{proof}

We close the section with a  result which improves the convergence rates derived in Theorem \ref{prop:im:lim} for the implicit algorithm.

\begin{thm}
	\label{thm:im:rate-o}
	Let $z_* \in {\cal Z}$ and $\left(z^{k} \right)_{k \geq 0}$ the sequence generated by Algorithm \ref{algo:im} for $\left(\beta_{k} \right) _{k \geq 0}$ a positive and nondecreasing sequence which satisfies \eqref{betak+}. 
	Then it holds
	\begin{equation*}
		\left\lVert z^{k} - z^{k-1} \right\rVert = o \left( \dfrac{1}{k} \right) \textrm{ as } k \to + \infty
	\end{equation*}
and
	\begin{equation*}
	\left\langle z^{k} - z_{*} , V \left( z^{k} \right) \right\rangle = o \left( \dfrac{1}{k \beta_{k}} \right) \ \mbox{and} \ \left\lVert V \left( z^{k} \right) \right\rVert = o \left( \dfrac{1}{k \beta_{k}} \right)  \textrm{ as } k \to + \infty .
	\end{equation*}
\end{thm}

\begin{proof}
Let $0 \leq \varepsilon < \alpha-2$ such that \eqref{as:beta-k:eps} is satisfied and $0 < \alpha - 1 - \frac{\varepsilon}{4} < \lambda  < \alpha - 1$. In the view of \eqref{defi:im:p-k}, the discrete energy sequence can be written as
	\begin{equation*}
	\E_{\lambda}^{k} = 4 \lambda p_{k} + \dfrac{1}{2} k^{2} \left\lVert 2 \left( z^{k} - z^{k-1} \right) + s \beta_{k-1} V \left( z^{k} \right) \right\rVert ^{2} + \dfrac{1}{2} s^{2} \left( k + \alpha \right) k \beta_{k} \beta_{k-1} \left\lVert V \left( z^{k} \right) \right\rVert ^{2} \quad \forall k \geq 1.
	\end{equation*}
According to Theorem \ref{prop:im:lim}, we have
\begin{equation*}
\lim\limits_{k \to + \infty} k \beta_{k} \beta_{k-1} \left\lVert V \left( z^{k} \right) \right\rVert ^{2} = 0 .
\end{equation*}
This statement together with the fact that the limits $\lim_{k \to + \infty} \E_{\lambda}^{k} \in \sR$ and $\lim_{k \to + \infty} p_{k} \in \sR$ (according to \eqref{conv:lim-p-k}) exist, allows us to deduce that for the sequence 
$$h_k:= \dfrac{k^{2}}{2} \left( \left\lVert 2 \left( z^{k} - z^{k-1} \right) + s \beta_{k-1} V \left( z^{k} \right) \right\rVert ^{2} + s^{2} \beta_{k} \beta_{k-1} \left\lVert V \left( z^{k} \right) \right\rVert ^{2} \right) \quad \forall k \geq 1,$$ 
the limit
\begin{equation*}
	\lim\limits_{k \to + \infty}  h_{k} \in [0,+\infty) \textrm{ exists}.
\end{equation*}
Furthermore, by taking into consideration the relation \eqref{as:beta-k}, Theorem \ref{prop:im:lim} also guarantees that
	\begin{align*}
	\mysum_{k \geq \left\lceil \alpha \right\rceil} \dfrac{1}{k} h_{k}
	& \leq 2 \mysum_{k \geq \left\lceil \alpha \right\rceil} k \left\lVert z^{k} - z^{k-1} \right\rVert ^{2} + s^{2} \mysum_{k \geq \left\lceil \alpha \right\rceil} k \left( \beta_{k-1} + \dfrac{\beta_{k}}{2} \right) \beta_{k-1} \left\lVert V \left( z^{k} \right) \right\rVert ^{2} \nonumber \\
	& \leq 2 \mysum_{k \geq \left\lceil \alpha \right\rceil} k \left\lVert z^{k} - z^{k-1} \right\rVert ^{2} + s^{2} \left( 1 + \dfrac{\alpha}{2 \left( 2 + \varepsilon \right)} \right) \mysum_{k \geq 1} k \beta_{k-1}^{2} \left\lVert V \left( z^{k} \right) \right\rVert ^{2} < + \infty .
	\end{align*}
	From here we conclude that $\lim_{k \to + \infty} h_{k} = 0$, and since $h_{k}$ is a sum of two nonnegative terms and, since $\left(\beta_{k} \right) _{k \geq 0}$ is nondecreasing, we further deduce
	\begin{equation*}
	\lim\limits_{k \to + \infty} k \left\lVert 2 \left( z^{k} - z^{k-1} \right) + s \beta_{k-1} V \left( z^{k} \right) \right\rVert = \lim\limits_{k \to + \infty} k \sqrt{\beta_{k} \beta_{k-1}} \left\lVert V \left( z^{k} \right) \right\rVert = \lim\limits_{k \to + \infty} k \beta_{k-1} \left\lVert V \left( z^{k} \right) \right\rVert = 0 .
	\end{equation*}
Using once again \eqref{as:beta-k}, we obtain
	\begin{equation*}
	\lim\limits_{k \to + \infty} k \beta_{k} \left\lVert V \left( z^{k} \right) \right\rVert = 0 .
	\end{equation*}	
	Since $\left(z_{k} \right)_{k \geq 0}$ is bounded, we use the Cauchy-Schwarz inequality to derive
	\begin{equation*}
	0 \leq \lim\limits_{k \to + \infty} k \beta_{k} \left\langle z^{k} - z_{*} , V \left( z^{k} \right) \right\rangle \leq \lim\limits_{k \to + \infty} k \beta_{k} \left\lVert z^{k} - z_{*} \right\rVert \left\lVert V \left( z^{k} \right) \right\rVert = 0 ,
	\end{equation*}
	and the proof is complete.
\end{proof}

\section{An explicit algorithm}

In this section, additional to its monotonicity, we will assume that the operator $V$ is $L$-Lipschitz continuous, with $L>0$. We propose and investigate an explicit numerical algorithm for solving \eqref{intro:pb:eq}, which follows from a temporal discretization of the dynamical system \eqref{ds}. 

The starting point is again its reformulation \eqref{dis:ds-fo}. We fix a time step $s > 0$, set $\tau_{k} := s \left( k + 1 \right)$ for every $k \geq 1$, and approximate $z \left( \tau_{k} \right) \approx z^{k+1}$ and $u \left( \tau_{k} \right) \approx u^{k+1}$. In addition, we choose $\beta \left( \tau_{k} \right) = 1$ for every $k \geq 1$ and refer to Remark \ref{remark2} for the explanation of why the time scaling parameter function $\beta$ is discretized via a constant sequence. The finite-difference scheme for \eqref{dis:ds-fo} at time $t:=\tau_k$ gives for every $k \geq 0$
\begin{equation}
	\label{ex:fd-fo}
	\begin{dcases}
		\dfrac{u^{k+1} - u^{k}}{s} 	& = \left( 2 - \alpha \right) V \left( \bz^{k} \right) \\
		u^{k+1} 					& = 2 \left( \alpha - 1 \right) z^{k+1} + 2 \left( k + 1 \right) \left( z^{k+1} - z^{k} \right) + 2s \left( k + 1 \right) V \left( \bz^{k} \right)
	\end{dcases}.
\end{equation}
Therefore we have for every $k \geq 1$
\begin{equation}
	\label{ex:defi:u}
	u^{k} = 2 \left( \alpha - 1 \right) z^{k} + 2k \left( z^{k} - z^{k-1} \right) + 2sk V \left( \bz^{k-1} \right),
\end{equation}
and after substraction we get
\begin{align}
	u^{k+1} - u^{k}
	& = 2 \left( k + \alpha \right) \left( z^{k+1} - z^{k} \right) - 2k \left( z^{k} - z^{k-1} \right) + 2s V \left( \bz^{k} \right) + 2sk \left( V \left( \bz^{k} \right) - V \left( \bz^{k-1} \right) \right) \nonumber \\
	& = \left( 2 - \alpha \right) s V \left( \bz^{k} \right) , \label{ex:d-u}
\end{align}
where the last relation comes from the first equation in \eqref{ex:fd-fo}.

On the other hand, the second equation in \eqref{ex:fd-fo} can be rewritten for every $k \geq 0$ as
\begin{equation}
	\label{ex:z}
	z^{k+1} = \dfrac{1}{2 \left( k + \alpha \right)} u^{k+1} + \dfrac{k+1}{k + \alpha} \left( z^{k} - s V \left( \bz^{k} \right) \right) .
\end{equation}
To get an explicit choice for $\bz^{k}$, we opt for
\begin{equation}
	\label{ex:bz}
	\bz^{k} := \dfrac{1}{2 \left( k + \alpha \right)} u^{k} + \dfrac{k+1}{k + \alpha} \left( z^{k} - \dfrac{sk}{k+1} V \left( \bz^{k-1} \right) \right) - \dfrac{\alpha s}{2 \left( k + \alpha \right)} V \left( \bz^{k-1} \right) \quad \forall k \geq 1.
\end{equation}
From here,  \eqref{ex:defi:u} gives for all $k \geq 1$
\begin{align*}
	\bz^{k} & = z^{k} + \dfrac{k}{k + \alpha} \left( z^{k} - z^{k-1} \right) - \dfrac{\alpha s}{2 \left( k + \alpha \right)} V \left( \bz^{k-1} \right) \nonumber \\
	& = z^{k} + \left( 1 - \dfrac{\alpha}{k + \alpha} \right) \left( z^{k} - z^{k-1} \right) - \dfrac{\alpha s}{2 \left( k + \alpha \right)} V \left( \bz^{k-1} \right),
\end{align*}
thus, by subtracting \eqref{ex:bz} from \eqref{ex:z}, we obtain
\begin{align}
	z^{k+1} - \bz^{k}
	& = \dfrac{1}{2 \left( k + \alpha \right)} \left( u^{k+1} - u^{k} \right) - \dfrac{s \left( k+1 \right)}{k + \alpha} V \left( \bz^{k} \right) + \dfrac{sk}{k + \alpha} V \left( \bz^{k-1} \right) - \dfrac{\alpha s}{2 \left( k + \alpha \right)} V \left( \bz^{k-1} \right) \nonumber \\
	& = - \dfrac{\alpha s}{2 \left( k + \alpha \right)} V \left( \bz^{k} \right) - \dfrac{sk}{k + \alpha} \left( V \left( \bz^{k} \right) - V \left( \bz^{k-1} \right) \right) + \dfrac{\alpha s}{2 \left( k + \alpha \right)} V \left( \bz^{k-1} \right) \nonumber \\
	& = - \dfrac{s}{2} \left( 1 + \dfrac{k}{k + \alpha} \right) \left( V \left( \bz^{k} \right) - V \left( \bz^{k-1} \right) \right) . \label{ex:Nes-scheme}
\end{align}
This gives the following important estimate, which holds for every $s >0$ such that $sL\leq 1 $ and every $k \geq 1$
\begin{equation}
\label{ex:Lip}
\left\lVert V \left( z^{k+1} \right) - V \left( \bz^{k} \right) \right\rVert
\leq L \left\lVert z^{k+1} - \bz^{k} \right\rVert
\leq sL \left\lVert V \left( \bz^{k} \right) - V \left( \bz^{k-1} \right) \right\rVert
\leq \left\lVert V \left( \bz^{k} \right) - V \left( \bz^{k-1} \right) \right\rVert .
\end{equation}

Now we can formally state our explicit numerical algorithm.
\begin{mdframed}
	\begin{algo}
		\label{algo:ex} {\bf (Explicit Fast OGDA)}
		Let $\alpha >2, z^{0} , z^{1}, \bz^{0} \in \sH$, and $0 < s < \frac{1}{2L}$. For every $k \geq 1$ we set
			\begin{align*}
				\bz^{k} & := z^{k} + \left( 1 - \dfrac{\alpha}{k + \alpha} \right) \left( z^{k} - z^{k-1} \right) - \dfrac{\alpha s}{2 \left( k + \alpha \right)} V \left( \bz^{k-1} \right) \\
				z^{k+1} & := \bz^{k} - \dfrac{s}{2} \left( 1 + \dfrac{k}{k + \alpha} \right) \left( V \left( \bz^{k} \right) - V \left( \bz^{k-1} \right) \right). 
			\end{align*}
	\end{algo}
\end{mdframed}

For $z_* \in {\cal Z}$,  $0 \leq \lambda \leq \alpha - 1$ and $z_* \in {\cal Z}$ $0 < \gamma < 2$ we define first in analogy to the implicit case the discrete energy function for every $k \geq 1$ by
\begin{align}
	\E_{\lambda}^{k}
	 := & \ \dfrac{1}{2} \left\lVert u^{k}_{\lambda} \right\rVert ^{2} + 2 \lambda \left( \alpha - 1 - \lambda \right) \left\lVert z^{k} - z_{*} \right\rVert ^{2} + 2 \left( 2 - \gamma \right) \lambda sk \left\langle z^{k} - z_{*} , V \left( \bz^{k-1} \right) \right\rangle \nonumber \\ 
	& + \dfrac{1}{2} \left( 2 - \gamma \right) s^{2} k \left( \gamma k + \alpha \right) \left\lVert V \left( \bz^{k-1} \right) \right\rVert ^{2},\label{ex:defi:E-k}
\end{align}
where
\begin{equation}
\label{ex:defi:u-k-1-lambda}
	u^{k}_{\lambda} := 2 \lambda \left( z^{k} - z_{*} \right) + 2k \left( z^{k} - z^{k-1} \right) + \gamma sk V \left( \bz^{k-1} \right).
\end{equation}
In strong contrast to the implicit case,  the discrete energy sequence $(\E_{\lambda}^{k})_{k \geq 1}$ might not dissipate with every iteration of the algorithm and be even negative.  This is the reason why we consider instead the following regularized seuquence of the energy function, defined for every $k \geq 2$ as
\begin{align}
	\F_{\lambda}^{k}
	 := & \ \E_{\lambda}^{k} - 2 \left( 2 - \gamma \right) sk^{2} \left\langle z^{k} - z^{k-1} , V \left( z^{k} \right) - V \left( \bz^{k-1} \right) \right\rangle \nonumber \\
	& + \dfrac{1}{2} \left( 2 - \gamma \right) s^{2} k \sqrt{k} \left( 2sL \sqrt{k} + \alpha \right) \left\lVert V \left( \bz^{k-1} \right) - V \left( \bz^{k-2} \right) \right\rVert ^{2} \nonumber \\
	& - \dfrac{1}{2} \lambda \left( \alpha - 2 \right) s^{2} \left( 2 - \frac{\alpha}{k + \alpha} \right) \left\lVert V \left( \bz^{k-1} \right) \right\rVert ^{2} . \label{ex:defi:F}
\end{align}
Its properties are collected in the following lemma, the proof of which is deferred to the Appendix.
\begin{lem}
	\label{lem:reg}
Let $z_* \in {\cal Z}$ and $\left(z^{k} \right)_{k \geq 0}$ be the sequence generated by Algorithm \ref{algo:ex} for $0 < \gamma < 2$ and  $0 \leq \lambda \leq \alpha - 1$. Then the following statements are true:
	\begin{enumerate}
		\item
		for every $k \geq k_{0} := \max \left\lbrace 2 , \left\lceil \frac{1}{\alpha - 2} \right\rceil \right\rbrace$ it holds
		\begin{align*}
			\F_{\lambda}^{k+1} - \F_{\lambda}^{k}
			\leq & \ 2 \lambda \left( 2 - \alpha \right) s \left\langle \bz^{k} - z_{*} , V \left( \bz^{k} \right) \right\rangle - \dfrac{1}{2} s^{2} \mu_{k} \left\lVert V \left( \bz^{k} \right) - V \left( \bz^{k-1} \right) \right\rVert ^{2} \nonumber \\
			& + 2 \Bigl( \omega_{2} k + \omega_{3} \sqrt{k} \Bigr) \left\lVert z^{k+1} - z^{k} \right\rVert ^{2}
			+ 2s \Bigl( \omega_{0} k + \omega_{1} \Bigr) \left\langle z^{k+1} - z^{k} , V \left( \bz^{k} \right) \right\rangle \nonumber \\
			& + \dfrac{1}{2} s^{2} \Bigl( \omega_{4} k + \omega_{5} \sqrt{k} \Bigr) \left\lVert V \left( \bz^{k} \right) \right\rVert ^{2},
		\end{align*}
		where		
		\begin{subequations}
			\label{reg:const}
			\begin{align}
			\mu_{k} 	& := \left( 2 - \gamma \right) \left( 2 \left( 1 - 2sL \right) \left( k + 1 \right) + \alpha^{2} \sqrt{k+1} + \alpha - 4 \right) \left( k+1 \right) - \left( 2 - \gamma \right) \left( \alpha - 2 \right) - 2 \lambda \left( \alpha - 2 \right) , \\
			\omega_{0} 	& := \left( 2 - \gamma \right) \lambda + \gamma - \alpha + \gamma \left( \lambda + 1 - \alpha \right) , \\
			\omega_{1} 	& := \gamma - \alpha + \alpha \left( \lambda + 1 - \alpha \right) < 0 , \\
			\omega_{2} 	& := 2 \left( \lambda + 1 - \alpha \right) \leq 0 , \\
			\omega_{3} 	& := \left( 2 - \gamma \right) \sqrt{\alpha - 2} > 0 , \\
			\omega_{4} 	& := 2 \gamma \left( 2 - \alpha \right) < 0 , \\
			\omega_{5} 	& := \left( 2 - \gamma \right) \alpha > 0 .
			\end{align}
		\end{subequations}
		
		\item
		if $1 < \gamma < 2$, then for every $k \geq k_{1} := \left\lceil \frac{2 \lambda \left( \alpha - 2 \right)}{\left( 2 - \gamma \right) \alpha} \right\rceil$ it holds
		\begin{align}
		\F_{\lambda}^{k}
		 \geq &  \ \dfrac{2 - \gamma}{\gamma} \left\lVert 2 \lambda \left( z^{k} - z_{*} \right) + k \left( z^{k} - z^{k-1} \right) + \gamma sk V \left( \bz^{k-1} \right) \right\rVert ^{2} \nonumber \\
		& + \dfrac{\left( 2 - \gamma \right) ^{2}}{2 \gamma} k^{2} \left\lVert z^{k} - z^{k-1} \right\rVert ^{2} + 2 \lambda \left( \alpha - 1 - \dfrac{2 \lambda}{\gamma} \right) \left\lVert z^{k} - z_{*} \right\rVert ^{2} . \label{reg:low}
		\end{align}
	\end{enumerate}
\end{lem}

In Lemma \ref{lem:reg} we have two degree of freedoms in the choice of the parameters $\gamma$ and $\lambda$.  The next result proves that there are choices for these parameters for which the discrete energy starts to dissipate with every iteration after a finite number of iterations and in the same time it is bounded from below by a nonnegative term. These two statements are fundamental in the derivation of the convergence rates and finally in the proof of the convergence of the iterates.  The proof of Lemma \ref{lem:trunc} can be found in the Appendix.

\begin{lem}
	\label{lem:trunc}
	The following statements are true:
	\begin{enumerate}
		\item if $\gamma$ and $\delta$ are such that
		\begin{equation}
		\label{trunc:gamma}
		1 + \dfrac{1}{\alpha - 1} < \gamma < 2 ,
		\end{equation}
		and
		\begin{equation}
		\label{trunc:delta}
		\max \left\lbrace \sqrt{2 \left( 1 - \dfrac{1}{\gamma} \right)} , \sqrt{\dfrac{\left( 2 - \gamma \right) \left( \alpha - 1 \right) + \left( \gamma - 1 \right) \left( \alpha - 2 \right)}{\gamma \left( \alpha - 2 \right)}}  \right\rbrace < \delta < 1,
		\end{equation}
		then there exist two parameters
		\begin{equation}
		\label{trunc:fea}
		0 \leq \underline{\lambda} \left( \alpha , \gamma \right) < \overline{\lambda} \left( \alpha , \gamma \right) \leq \dfrac{\gamma}{2} \left( \alpha - 1 \right) ,
		\end{equation}
such that for every $\lambda$ satisfying $\underline{\lambda} \left( \alpha , \gamma \right) < \lambda < \overline{\lambda} \left( \alpha , \gamma \right)$ one can find an integer $k_{2} \left( \lambda \right) \geq 1$ with the property that the following inequality holds for every $k \geq k_{2} \left( \lambda \right)$
		\begin{align}
		R_{k} := 2 \delta \Bigl( \omega_{2} k + \omega_{3} \sqrt{k} \Bigr) \left\lVert z^{k+1} - z^{k} \right\rVert ^{2}
		+ 2s \Bigl( \omega_{0} k + \omega_{1} \Bigr) \left\langle z^{k+1} - z^{k} , V \left( \bz^{k} \right) \right\rangle & \nonumber \\
		\quad + \dfrac{\delta}{2} s^{2} \Bigl( \omega_{4} k + \omega_{5} \sqrt{k} \Bigr) \left\lVert V \left( \bz^{k} \right) \right\rVert ^{2} & \leq 0 ; \label{Rk}
		\end{align}
		
		\item
		there exists a positive integer $k_3$ such that for every $k \geq k_{3}$ it holds
		\begin{equation}\label{muk}
		\mu_{k} \geq \left( 2 - \gamma \right) \left( 1 - 2sL \right) \left( k+1 \right)^{2} .
		\end{equation}
	\end{enumerate}
\end{lem}

Now we are in position to provide first convergence rates statements for Algorithm \ref{algo:ex}.
\begin{thm}
	\label{prop:ex:lim}
Let $z_* \in {\cal Z}$ and $\left(z^{k} \right)_{k \geq 0}$ be the sequence generated by Algorithm \ref{algo:ex}. Then the following statements are true:
	\begin{enumerate}		
		\item
		\label{ex:lim:i}
		it holds
		\begin{subequations}
			\label{ex:lim:sum}
			\begin{align}		
			\mysum_{k \geq 1} \left\langle \bz^{k} - z_{*} , V \left( \bz^{k} \right) \right\rangle & < + \infty , \label{ex:lim:vi} \\
			\mysum_{k \geq 1} k^{2} \left\lVert V \left( \bz^{k} \right) - V \left( \bz^{k-1} \right) \right\rVert ^{2} & < + \infty , \label{ex:lim:dV} \\
			\mysum_{k \geq 1} k \left\lVert z^{k+1} - z^{k} \right\rVert ^{2} & < + \infty , \label{ex:lim:dz} \\
			\mysum_{k \geq 1} k \left\lVert V \left( \bz^{k} \right) \right\rVert ^{2} & < + \infty ; \label{ex:lim:V}
			\end{align}
		\end{subequations}
	
		\item
		\label{ex:lim:ii}
		the sequence $\left(z^{k} \right) _{k \geq 0}$ is bounded and it holds 
\begin{align*}
& \left\lVert z^{k} - z^{k-1} \right\rVert = \bO \left( \dfrac{1}{k} \right), \quad \left\langle z^{k} - z_{*} , V \left( z^{k} \right) \right\rangle = \bO \left( \dfrac{1}{k} \right), \nonumber \\
&  \left\lVert V \left( z^{k} \right) \right\rVert = \bO \left( \dfrac{1}{k} \right),\quad  \left\lVert V \left( \bz^{k} \right) \right\rVert = \bO \left( \dfrac{1}{k} \right)  \textrm{ as } k \to + \infty; 
		\end{align*}
	
		\item
		\label{ex:lim:iii} if $1 + \dfrac{1}{\alpha - 1} < \gamma < 2$, then there exist $0 \leq \underline{\lambda} \left( \alpha , \gamma \right) < \overline{\lambda} \left( \alpha , \gamma \right) \leq \frac{\gamma}{2} \left( \alpha - 1 \right)$ such that for every $\underline{\lambda} \left( \alpha , \gamma \right) < \lambda < \overline{\lambda} \left( \alpha , \gamma \right)$ both sequences $\left( \E_{\lambda}^{k} \right)_{k \geq 1}$ and $\left(\F_{\lambda}^{k} \right) _{k \geq 2}$ converge.
	\end{enumerate}
\end{thm}
\begin{proof}
Let $1 + \dfrac{1}{\alpha - 1} < \gamma < 2$ and $0 < \delta < 1$ such that \eqref{trunc:delta} holds. According to Lemma  \ref{lem:trunc} there exist $\underline{\lambda} \left( \alpha , \gamma \right) < \overline{\lambda} \left( \alpha , \gamma \right)$ such that \eqref{trunc:fea} holds. We choose $\underline{\lambda} \left( \alpha , \gamma \right) < \lambda < \overline{\lambda} \left( \alpha , \gamma \right)$ and get, according to the same result, an integer $k_2(\lambda) \geq 1$ such that for every $k \geq k_2(\lambda)$ the inequality \eqref{Rk} holds. In addition, according to Lemma  \ref{lem:trunc}(ii), we get a positive integer $k_3$ such that \eqref{muk} holds for every $k \geq k_3$.

This means that for every $k \geq k_{4} \left( \lambda \right) := \max \left\lbrace k_{0} , k_{2} \left( \lambda \right) , k_{3} \right\rbrace$, where $k_0$ is the positive integer provided by Lemma \eqref{lem:reg}(i), we have
	\begin{align*}
	\F_{\lambda}^{k+1} - \F_{\lambda}^{k}
	\leq & \ 2 \lambda \left( 2 - \alpha \right) s \left\langle \bz^{k} - z_{*} , V \left( \bz^{k} \right) \right\rangle - \dfrac{1}{2} \left( 2 - \gamma \right) \left( 1 - 2sL \right) s^{2} \left( k+1 \right)^{2} \left\lVert V \left( \bz^{k} \right) - V \left( \bz^{k-1} \right) \right\rVert ^{2} \nonumber \\
	& + 2 \left( 1 - \delta \right) \Bigl( \omega_{2} k + \omega_{3} \sqrt{k} \Bigr) \left\lVert z^{k+1} - z^{k} \right\rVert ^{2} + \dfrac{1}{2} \left( 1 - \delta \right) s^{2} \Bigl( \omega_{4} k + \omega_{5} \sqrt{k} \Bigr) \left\lVert V \left( \bz^{k} \right) \right\rVert ^{2} .
	\end{align*}
	Since $\omega_{2} < 0, \omega_{4} < 0$ and $\omega_{3}, \omega_{5} \geq 0$, we can choose $k_{5} := \left\lceil \max \left\lbrace - \frac{2 \omega_{3}}{\omega_{2}} , - \frac{2 \omega_{5}}{\omega_{4}} \right\rbrace \right\rceil > 0$, which then means that for every $k \geq k_6:=\max \left\lbrace k_{4} \left( \lambda \right) , k_{5} \right\rbrace$
	\begin{align}
	\F_{\lambda}^{k+1} - \F_{\lambda}^{k}
	& \leq 2 \lambda \left( 2 - \alpha \right) s \left\langle \bz^{k} - z_{*} , V \left( \bz^{k} \right) \right\rangle - \dfrac{1}{2} \left( 2 - \gamma \right) s^{2} \left( 1 - 2sL \right) \left( k+1 \right)^{2} \left\lVert V \left( \bz^{k} \right) - V \left( \bz^{k-1} \right) \right\rVert ^{2} \nonumber \\
	& \quad + \left( 1 - \delta \right) \omega_{2} k \left\lVert z^{k+1} - z^{k} \right\rVert ^{2}
	+ \dfrac{1}{4} \left( 1 - \delta \right) s^{2} \omega_{4} k \left\lVert V \left( \bz^{k} \right) \right\rVert ^{2} .\label{ineqFk}
	\end{align}
In view of \eqref{reg:low} and by taking into account that $\lambda < \frac{\gamma}{2} \left( \alpha - 1 \right)$, we get that $\F_{\lambda}^{k} \geq 0$ starting from the index $k_{1}$, thus the sequence $\left( \F_{\lambda}^{k} \right) _{k \geq 2}$ is bounded from below. Under these premises, we can apply Lemma \ref{lem:quasi-Fej} to \eqref{ineqFk}, and obtain \ref{ex:lim:i} as well as that the sequence $\left(\F_{\lambda}^{k} \right)_{k \geq 1}$ converges.
	
According to \eqref{ineqFk}, we also have that  $\left( \F_{\lambda}^{k} \right) _{k \geq k_{6}}$ is nonincreasing, which, according to \eqref{reg:low}, implies that following estimate holds for every $k \geq k_{6}$
	\begin{align*}
	\dfrac{2 - \gamma}{\gamma} \left\lVert 2 \lambda \left( z^{k} - z_{*} \right) + k \left( z^{k} - z^{k-1} \right) + \gamma sk V \left( \bz^{k-1} \right) \right\rVert ^{2} & \nonumber \\
	+ \dfrac{\left( 2 - \gamma \right) ^{2}}{2 \gamma} k^{2} \left\lVert z^{k} - z^{k-1} \right\rVert ^{2} + 2 \lambda \left( \alpha - 1 - \dfrac{2 \lambda}{\gamma} \right) \left\lVert z^{k} - z_{*} \right\rVert ^{2} & \leq \F_{\lambda}^{k} \leq \F_{\lambda}^{k_{6}} < + \infty .
	\end{align*}
This yields that the sequences $\left(2 \lambda \left( z^{k} - z_{*} \right) + k \left( z^{k} - z^{k-1} \right) + \gamma sk V \left( \bz^{k-1} \right) \right) _{k \geq 1}$, $\left(k \left( z^{k} - z^{k-1} \right) \right) _{k \geq 1}$, and $\left(z^{k} \right)_{k \geq 0}$ are bounded.  In particular, for every $k \geq k_{6}$ we have
	\begin{align*}
	\left\lVert 2 \lambda \left( z^{k} - z_{*} \right) + k \left( z^{k} - z^{k-1} \right) + \gamma sk V \left( \bz^{k-1} \right) \right\rVert
	& \leq C_{0} := \sqrt{\dfrac{\gamma \F_{\lambda}^{k_{6}}}{2 - \gamma}} < + \infty , \nonumber \\
	k \left\lVert z^{k} - z^{k-1} \right\rVert
	& \leq C_{1} := \dfrac{\sqrt{2 \gamma \F_{\lambda}^{k_{6}}}}{2 - \gamma} < + \infty , \nonumber \\
	\left\lVert z^{k} - z_{*} \right\rVert
	& \leq C_{2} := \sqrt{\dfrac{\gamma \F_{\lambda}^{k_{6}}}{2 \lambda \left( \gamma \left( \alpha - 1 \right) - 2 \lambda \right)}} < + \infty .
	\end{align*}
	Using the triangle inequality, we deduce from here that  for every $k \geq k_{6}$
	\begin{align}\label{ineq-C3}
	\left\lVert V \left( \bz^{k-1} \right) \right\rVert
	& \leq \dfrac{1}{\gamma s k} \left( \left\lVert 2 \lambda \left( z^{k} - z_{*} \right) + k \left( z^{k} - z^{k-1} \right) + \gamma sk V \left( \bz^{k-1} \right) \right\rVert + 2 \lambda \left\lVert z^{k} - z_{*} \right\rVert \right) \nonumber \\
	& \quad + \dfrac{1}{\gamma s} \left\lVert z^{k} - z^{k-1} \right\rVert \leq \dfrac{C_{3}}{k} ,
	\end{align}
	where
	\begin{equation*}
	C_{3} := \dfrac{1}{\gamma s} \left( C_{0} + C_{1} + 2 \lambda \left( \alpha , \gamma \right) C_{2} \right) > 0 .
	\end{equation*}
	The statement \eqref{ex:lim:dV} yields
	\begin{equation}
	\lim\limits_{k \to + \infty} k \left\lVert V \left( \bz^{k} \right) - V \left( \bz^{k-1} \right) \right\rVert = 0
	\quad \Rightarrow \quad
	C_{4} := \sup_{k \geq 1} \left\lbrace k \left\lVert V \left( \bz^{k} \right) - V \left( \bz^{k-1} \right) \right\rVert \right\rbrace < + \infty , \label{ex:lim:dV-0}
	\end{equation}
	which, together with \eqref{ex:Lip} implies that  for every $k \geq k_{6}$
	\begin{align}
	\left\lVert V \left( z^{k+1} \right) \right\rVert
	\leq \left\lVert V \left( z^{k+1} \right) - V \left( \bz^{k} \right) \right\rVert + \left\lVert V \left( \bz^{k} \right) \right\rVert
	& \leq \left\lVert V \left( \bz^{k} \right) - V \left( \bz^{k-1} \right) \right\rVert + \left\lVert V \left( \bz^{k} \right) \right\rVert \label{ex:lim:V-dV} \leq \dfrac{C_{5}}{k} , 
	\end{align}
	where
	\begin{equation*}
	C_{5} := C_{3} + C_{4} > 0 .
	\end{equation*}
	The last assertion in \ref{ex:lim:ii} follows from the Cauchy-Schwarz inequality and the boundedness of $\left( z^{k} \right)_{k \geq 0}$, namely, for for every $k \geq k_{6}$ it holds
	\begin{equation*}
	0 \leq \left\langle z^{k} - z_{*} , V \left( z^{k} \right) \right\rangle \leq \left\lVert z^{k} - z_{*} \right\rVert \left\lVert V \left( z^{k} \right) \right\rVert \leq \dfrac{C_{2} C_{5}}{k-1} .
	\end{equation*}
To complete the proof of \ref{ex:lim:iii}, we are going to show that in fact
	\begin{equation*}
		\lim\limits_{k \to + \infty} \E_{\lambda}^{k} = \lim\limits_{k \to + \infty} \F_{\lambda}^{k} \in \sR .
	\end{equation*}
Indeed, we already have seen that
\begin{equation*}
	\lim\limits_{k \to + \infty} \left( k + 1 \right) \left\lVert V \left( \bz^{k} \right) - V \left( \bz^{k-1} \right) \right\rVert = \lim\limits_{k \to + \infty} \left\lVert V \left( \bz^{k} \right) \right\rVert = 0 ,
\end{equation*}
which, by the Cauchy-Schwarz inequality and \eqref{ex:Lip} yields
\begin{align*}
	0 \leq \lim\limits_{k \to + \infty} k^{2} \left\lvert \left\langle z^{k} - z^{k-1} , V \left( z^{k} \right) - V \left( \bz^{k-1} \right) \right\rangle \right\rvert
	& \leq C_{1} \lim\limits_{k \to + \infty} k \left\lVert V \left( z^{k} \right) - V \left( \bz^{k-1} \right) \right\rVert \nonumber \\
	& \leq C_{1} \lim\limits_{k \to + \infty} k \left\lVert V \left( \bz^{k-1} \right) - V \left( \bz^{k-2} \right) \right\rVert = 0 .
\end{align*}
From here we obtain the desired statement.
\end{proof}

The following theorem addresses the convergence of the sequence of iterates to an element in ${\cal Z}$.

\begin{thm}
  Let $z_* \in {\cal Z}$ and $\left(z^{k} \right)_{k \geq 0}$ be the sequence generated by Algorithm \ref{algo:ex}. Then the sequence $\left( z^{k} \right) _{k \geq 0}$ converges weakly to a solution of \eqref{intro:pb:eq}.
\end{thm}
\begin{proof}
Let $1 + \dfrac{1}{\alpha - 1} < \gamma < 2$ and $\underline{\lambda} \left( \alpha , \gamma \right) < \overline{\lambda} \left( \alpha , \gamma \right)$ be the parameters provided by Lemma  \ref{lem:trunc} such that \eqref{trunc:fea} holds and with the property that for every $\underline{\lambda} \left( \alpha , \gamma \right) < \lambda < \overline{\lambda} \left( \alpha , \gamma \right)$ there exists an integer $k_2(\lambda) \geq 1$ such that for every $k \geq k_2(\lambda)$ the inequality \eqref{Rk} holds.  The proof relies on Opial's Lemma  and follows the line of the proof of Theorem \ref{thm:ex:conv},  by defining this time for every 
$k \geq 1$
    \begin{align}
    p_{k}    & := \dfrac{1}{2} \left( \alpha - 1 \right) \left\lVert z^{k} - z_{*} \right\rVert ^{2} + k \left\langle z^{k} - z_{*} , z^{k} - z^{k-1} + sV \left( \bz^{k-1} \right) \right\rangle , \label{ex:defi:p-k} \\
    q_{k}    & := \dfrac{1}{2} \left\lVert z^{k} - z_{*} \right\rVert ^{2} + s \mysum_{i = 1}^{k} \left\langle z^{i} - z_{*} , V \left( \bz^{i-1} \right) \right\rangle . \label{ex:defi:q-k}
    \end{align}
One can notice that the limit
\begin{equation*}
        \lim_{k \to + \infty} \sum_{i = 1}^{k} \left\langle z^{i} - z_{*} , V \left( \bz^{i-1} \right) \right\rangle = \lim_{k \to + \infty} \sum_{i = 1}^{k} \left\langle z^{i} - \bz^{i-1} , V \left( \bz^{i-1} \right) \right\rangle + \lim_{k \to + \infty} \sum_{i = 1}^{k} \left\langle \bz^{i-1} - z_{*} , V \left( \bz^{i-1} \right) \right\rangle \in \sR 
\end{equation*}
exists due to \eqref{ex:lim:vi} and the fact that the series $\sum_{k \geq 2} \left \langle z^{k} - \bz^{k-1} , V \left( \bz^{k-1} \right) \right\rangle$ is absolutely convergent, which follows from
        \begin{align*}
        \mysum_{i = 2}^{k} \left\lvert \left\langle z^{i} - \bz^{i-1} , V \left( \bz^{i-1} \right) \right\rangle \right\rvert \leq \mysum_{i = 2}^{k} \left\lVert V \left( \bz^{i-1} \right) \right\rVert \left\lVert z^{i} - \bz^{i-1} \right\rVert  \leq sC_{3} C_{4} \mysum_{i = 2}^{\infty} \dfrac{1}{i \left( i-1 \right)} < + \infty \quad\forall k \geq 2,
        \end{align*}
where we make use of \eqref{ex:Lip}, \eqref{ex:lim:dV-0}, and \eqref{ineq-C3}, and of the constants $C_3, C_4$ defined in the proof of Theorem \ref{prop:ex:lim}. 
\end{proof}

As for the implicit algorithm, we can improve also for the explicit algorithm the convergence rates. 

\begin{thm}
     Let $z_* \in {\cal Z}$ and $\left(z^{k} \right)_{k \geq 0}$ be the sequence generated by Algorithm \ref{algo:ex}. Then it holds
\begin{align*}
& \left\lVert z^{k} - z^{k-1} \right\rVert = o \left( \dfrac{1}{k} \right), \quad \left\langle z^{k} - z_{*} , V \left( z^{k} \right) \right\rangle = o \left( \dfrac{1}{k} \right), \nonumber \\
&  \left\lVert V \left( z^{k} \right) \right\rVert = o \left( \dfrac{1}{k} \right),\quad  \left\lVert V \left( \bz^{k} \right) \right\rVert = o \left( \dfrac{1}{k} \right)  \textrm{ as } k \to + \infty.
		\end{align*}
\end{thm}
\begin{proof}
Let $1 + \dfrac{1}{\alpha - 1} < \gamma < 2$ and $\underline{\lambda} \left( \alpha , \gamma \right) < \overline{\lambda} \left( \alpha , \gamma \right)$ be the parameters provided by Lemma  \ref{lem:trunc} such that \eqref{trunc:fea} holds and with the property that for every $\underline{\lambda} \left( \alpha , \gamma \right) < \lambda < \overline{\lambda} \left( \alpha , \gamma \right)$ there exists an integer $k_2(\lambda) \geq 1$ such that for every $k \geq k_2(\lambda)$ the inequality \eqref{Rk} holds. 

We fix $\underline{\lambda} \left( \alpha , \gamma \right) < \lambda < \overline{\lambda} \left( \alpha , \gamma \right)$ and recall that according to Theorem \ref{prop:ex:lim}(iii) the sequence $(\E_{\lambda}^{k})_{k \geq 1}$ converges.

From \eqref{ex:defi:u-k-1-lambda} and \eqref{ex:defi:E-k} we have that for every $k \geq 1$
    \begin{align*}
    \E_{\lambda}^{k}
     = &  \ \dfrac{1}{2} \left\lVert 2 \lambda \left( z^{k} - z_{*} \right) + 2k \left( z^{k} - z^{k-1} \right) + \gamma sk V \left( \bz^{k-1} \right) \right\rVert ^{2} + 2 \lambda \left( \alpha - 1 - \lambda \right) \left\lVert z^{k} - z_{*} \right\rVert ^{2} \nonumber \\
    & + 2 \left( 2 - \gamma \right) \lambda sk \left\langle z^{k} - z_{*} , V \left( \bz^{k-1} \right) \right\rangle + \dfrac{1}{2} \left( 2 - \gamma \right) s^{2} k \left( \gamma k + \alpha \right) \left\lVert V \left( \bz^{k-1} \right) \right\rVert ^{2} \nonumber \\
     = & \ 2 \lambda \left( \alpha - 1 \right) \left\lVert z^{k} - z_{*} \right\rVert ^{2} + 4 \lambda k \left\langle z^{k} - z_{*} , z^{k} - z^{k-1} + sV \left( \bz^{k-1} \right) \right\rangle + \dfrac{1}{2} \left( 2 - \gamma \right) \alpha s^{2} k \left\lVert V \left( \bz^{k-1} \right) \right\rVert ^{2} \nonumber \\
    &  + \dfrac{k^{2}}{2} \left( \left\lVert 2 \left( z^{k} - z^{k-1} \right) + \gamma s V \left( \bz^{k-1} \right) \right\rVert ^{2} + \left( 2 - \gamma \right) \gamma s^{2} \left\lVert V \left( \bz^{k-1} \right) \right\rVert ^{2} \right) .
    \end{align*}
We set for every $k \geq 1$
    \begin{equation*}
    h_{k} := \dfrac{k^{2}}{2} \left( \left\lVert 2 \left( z^{k} - z^{k-1} \right) + \gamma s V \left( \bz^{k-1} \right) \right\rVert ^{2} + \left( 2 - \gamma \right) \gamma s^{2} \left\lVert V \left( \bz^{k-1} \right) \right\rVert ^{2} \right) ,
    \end{equation*}
and notice that, in view of \eqref{ex:defi:p-k}, we have
    \begin{equation*}
    \E_{\lambda}^{k} = 4 \lambda p_{k} + \dfrac{1}{2} \left( 2 - \gamma \right) \alpha s^{2} k \left\lVert V \left( \bz^{k-1} \right) \right\rVert ^{2} + h_{k} .
    \end{equation*}
    Theorem \ref{prop:ex:lim} asserts that
    \begin{equation*}
    \lim\limits_{k \to \infty} k \left\lVert V \left( \bz^{k-1} \right) \right\rVert ^{2} = 0 ,
    \end{equation*}
    which, together with $\lim_{k \to + \infty} \E_{\lambda}^{k} \in \sR$ and $\lim_{k \to + \infty} p_{k} \in \sR$, yields
    \begin{equation*}
    \lim\limits_{k \to + \infty} h_{k} \in \sR \textrm{ exists}.
    \end{equation*}
In addition, \eqref{ex:lim:dz} and \eqref{ex:lim:V} in Theorem \ref{prop:ex:lim} guarantee that
    \begin{align*}
    \mysum_{k \geq 1} \dfrac{1}{k} h_{k} \leq 4 \mysum_{k \geq 1} k \left\lVert z^{k} - z^{k-1} \right\rVert ^{2} + \dfrac{1}{2} \left( 2 + \gamma \right) \gamma s^{2} \mysum_{k \geq 1} k \left\lVert V \left( \bz^{k-1} \right) \right\rVert ^{2} < + \infty .
    \end{align*}
    Consequently, $\lim_{k \to + \infty} h_{k} = 0$, which yields
    \begin{equation*}
    \lim\limits_{k \to \infty} k \left\lVert 2 \left( z^{k} - z^{k-1} \right) + \gamma s V \left( \bz^{k-1} \right) \right\rVert = \lim\limits_{k \to \infty} k \left\lVert V \left( \bz^{k-1} \right) \right\rVert = 0 .
    \end{equation*}
    This immediately implies $\lim_{k \to + \infty} k\left\lVert z^{k} - z^{k-1} \right\rVert = 0$. The fact that $\lim_{k \to + \infty} k \left\lVert V \left( z^{k} \right) \right\rVert = 0$ follows from \eqref{ex:lim:dV-0} and \eqref{ex:lim:V-dV}, since
    \begin{equation*}
    0 \leq \lim\limits_{k \to + \infty} k \left\lVert V \left( z^{k} \right) \right\rVert \leq \lim\limits_{k \to + \infty} k\left\lVert V \left( \bz^{k} \right) - V \left( \bz^{k-1} \right) \right\rVert + \lim\limits_{k \to + \infty} k\left\lVert V \left( \bz^{k} \right) \right\rVert = 0 .
 \end{equation*}
Finally, using the Cauchy-Schwarz inequality and the fact that $\left(z^{k} \right) _{k \geq 0}$ is bounded, we obtain that $\lim_{k \to + \infty} k \left\langle z^{k} - z_{*} , V \left( z^{k} \right) \right\rangle = 0$.
\end{proof}

\section{Numerical experiments}

In this section we perform numerical experiments to illustrate the convergence rates derived for the explicit Fast OGDA method and to compare our algorithm with other numerical schemes from the literature designed to solve equations governed by a monotone and Lipschitz continuous operator. To this end we consider a minmax problem studied in \cite{Ouyang-Xu},  which has then been used in \cite{Yoon-Ryu} to illustrate the performances of anchoring based numerical methods. This reads
\begin{equation}
	\label{num:pb}
	\min\limits_{x \in \sR^{n}} \max\limits_{y \in \sR^{n}} \Lag \left( x , y \right) := \dfrac{1}{2} \left\langle x , Hx \right\rangle - \left\langle x , h \right\rangle - \left\langle y , Ax - b \right\rangle,
\end{equation}
where
\begin{align*}
	A :=  \dfrac{1}{4} \begin{pmatrix}
		& & & -1 & 1 \\
		& & \iddots & \iddots \\
		& -1 & 1 \\
		-1  & 1 \\
		1
	\end{pmatrix} \in \sR^{n \times n} , \quad
	H := 2A^{T}A , \quad
	b := \dfrac{1}{4} \begin{pmatrix}
		1 \\ 1 \\ \vdots \\ 1 \\ 1
	\end{pmatrix} \in \sR^{n} \ \mbox{and} \
	h := \dfrac{1}{4} \begin{pmatrix}
		0 \\ 0 \\ \vdots \\ 0 \\ 1
	\end{pmatrix} \in \sR^{n} .
\end{align*}
We notice that $\Lag$ is nothing else than the Lagrangian of a linearly constrained quadratic minimization problem. It has been shown in \cite{Ouyang-Xu} that $\left\lVert A \right\rVert \leq \frac{1}{2}$, thus $\left\lVert H \right\rVert \leq \frac{1}{2}$, and, consequently, for the monotone mapping $\left( x , y \right) \mapsto \Bigl( \nabla_{x} \Lag \left( x , y \right) , - \nabla_{y} \Lag \left( x , y \right) \Bigr)$ we can take $L=1$ as Lipschitz constant.

In the following we summarize all the algorithms we use in the numerical experiments and the corresponding step sizes:
\begin{enumerate}
	\item OGDA: Optimistic Gradient Descent Ascent  method \eqref{algo:OGDA} (see \cite{Popov}) with $s := \frac{0.48}{L}$ ;	
	\item EG: Extragradient method \eqref{algo:EG} (see \cite{Korpelevich,Antipin}) with $s := \frac{0.96}{L}$;
	\item EAG-V: Extra Anchored Gradient method \eqref{algo:EAG} (see \cite{Yoon-Ryu}) with variable step sizes $\left( s_{k} \right) _{k \geq 0}$ satisfying \eqref{stepsizeEAG};
	\item Nesterov-EAG: Nesterov’s accelerated variant of the Extra Anchored Gradient method, which has been proposed in \cite{Tran-Dinh} and can be obtained from \eqref{algo:EAG} be taking in the first update line the sequence $\left( \frac{k+1}{L \left( k+2 \right)} \right) _{k \geq 0}$ as step sizes, and in the second one the constant step size $\frac{1}{L}$ (see \cite[Theorem 5.1, Lemma 5.1, Theorem 5.2]{Tran-Dinh});
\item Halpern-OGDA:  OGDA mixed with the Halpern anchoring scheme, which has been proposed in \cite{Tran-Dinh-Luo} and can be obtained from the variant of \eqref{algo:EAG} with variable step sizes  by replacing in the first update line $V \left( z^{k} \right)$ by $V \left( \bz^{k-1} \right)$;
	\item Fast OGDA: our explicit algorithm with $s := \frac{0.48}{L}$ and various choices of $\alpha$.
\end{enumerate}


\begin{figure}[h!]
	\centering
	\includegraphics[width=0.7 \textwidth]{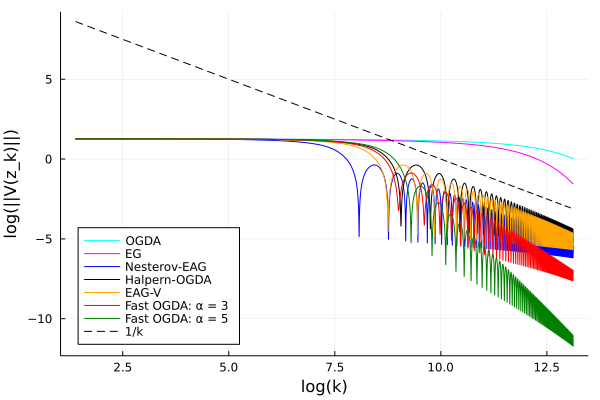}
	\caption{Explicit Fast OGDA outperforms all other explicit methods}
	\label{fig:eff}
\end{figure}

For the first numerical experiments we consider the same setting as in \cite{Yoon-Ryu}, namely, we take $n=200$, which means that the underlying space is ${\mathbb R}^{400}$, and allow a maximum number of iterations of $5 \times 10^{5}$. Figure \ref{fig:eff} presents the convergence behaviour of the different methods when solving \eqref{num:pb} in logarithmic scale.  One can see that the anchoring based methods perform better than the classical algorithms EG and OGDA, and that Nesterov-EAG performs better than Halpern-OGDA, which reconfirms a finding of \cite{Tran-Dinh} and is not surprising when one takes into account that the first allows for larger step sizes than EAG-V (and Halpnern-OGDA). On the other hand, Fast OGDA outperforms all the other methods in spite of the fact that the step size is restricted to $\left(0, \frac{1}{2L}\right)$.

\begin{figure}[h!]
	\centering
	\includegraphics[width=0.7 \textwidth]{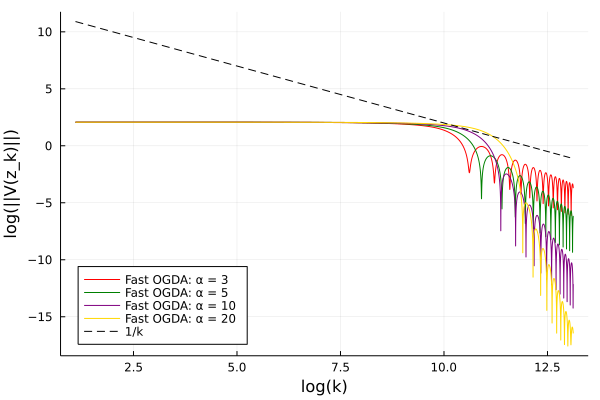}
	\caption{The parameter $\alpha$ influences the convergence behaviour of explicit Fast OGDA}
	\label{fig:alpha}
\end{figure}

Figure \ref{fig:alpha} shows that the parameter $\alpha >2$ influences significantly the convergence behaviour of the explicit Fast OGDA method. For this numerical experiment we take  $n=1000$, which means that the underlying space is ${\mathbb R}^{2000}$, and allow a maximum number of iterations of $5 \times 10^{5}$. The speed  of convergence increases with increasing $\alpha$ and seem to be much better than $o \left( 1/k \right)$.  Let us mention that the minimax problem \eqref{num:pb} was constructed to show lower complexity bounds of first order methods for convex-concave saddle point problems. 

For Nesterov's dynamical systems with $\frac{\alpha}{t}$ as damping coefficient and the corresponding numerical algorithms approaching the minimization of a smooth and convex function, it is known that $\alpha$ influences in the same way the convergence rates of the objective function values. Another intriguing similarity with Nesterov's continuous and discrete schemes is the evident oscillatory behaviour of the trajectories, however,  there for the objective function values,  while for explicit Fast OGDA for the norm of the operator along the trajectory/sequence of generated iterates. This suggests that Nesterov's acceleration approach can improve the convergence behaviour of continuous and discrete time approaches beyond the optimization setting.

In the following we complement the comparative study of the above numerical methods by following the performance profile developed by Dolan and Mor\'e (\cite{Dolan-More}). We denote by $\mathsf{S}$ the set of the algorithms/solvers (i)-(vi) from above,  where for the Fast OGDA method we take $\alpha := 3$.  We solve minmax problems of the form \eqref{num:pb} with ${\cal L} : \sR^n \times \sR^m \rightarrow \sR$,  for 10 different pairs $\left(n, m \right)$ such that $20 \leq m \leq n \leq 200$ and,  for each such pair,  for $100$ randomly generated sparse matrices $A \in \sR^{m \times n}$ and vectors $b \in \sR^{m}$ and $h \in \sR^{n}$, and $H:=2A^TA$.  For each pair $\left(n,  m \right)$ we also take $10$ initial points with normal distribution, which leads to a set of problems $\mathsf{P}$ with $\mathsf{N_{p}} = 10 \times 100 \times 10= 10000$ instances.

For each problem $\mathsf{p} \in \mathsf{P}$ and each solver $\mathsf{s} \in \mathsf{S}$, we denote by $\mathsf{t_{p,s}}$ the number of iterations needed by solver $\mathsf{s}$ to solve the problem $\mathsf{p}$ successfully, i.e.  by satsifying the following stopping criteria before $\mathtt{k_{\max}} := 10^{5}$ iterations
\begin{equation*}
	\dfrac{\left\lVert V \left( x_{k} , y_{k} \right) \right\rVert}{\left\lVert V \left( x_{0} , y_{0} \right) \right\rVert} \leq \mathtt{Tol_{op}}=10^{-6}
	\quad \textrm{ and } \quad
	\dfrac{\left\lVert \left( x_{k} , y_{k} \right) - \left( x_{k-1} , y_{k-1} \right) \right\rVert}{\left\lVert \left( x_{k} , y_{k} \right) \right\rVert + 1} \leq \mathtt{Tol_{vec}} =10^{-5}.
\end{equation*}
The two stopping criteria quantify the relative errors measured for the operator norm and the discrete velocity.  We define the \emph{performance ratio} as
\begin{equation*}
	\mathsf{r_{p,s}} := \begin{cases}
		\dfrac{\mathsf{t_{p,s}}}{\min \left\lbrace \mathsf{t_{p,s}} \colon \mathsf{s} \in \mathsf{S} \right\rbrace} & \textrm{ if } \mathsf{t_{p,s}} < \mathtt{k_{\max}}, \\
		0 		& \textrm{ otherwise, }
	\end{cases}
\end{equation*}
and the \emph{performance} of the solver $\mathsf{s}$ as
\begin{equation*}
	\mathsf{\rho_{s}} \left( \tau \right) := \dfrac{1}{\mathsf{N_{p}}} \mathrm{size} \left\lbrace \mathsf{p} \in \mathsf{P} \colon 0 < \mathsf{r_{p,s}} \leq \tau \right\rbrace,
\end{equation*}
where $\tau$ is a real factor. The performance $\mathsf{\rho_{s}} \left( \tau \right)$ for solver $\mathsf{s}$ gives the probability that the performance ratio $\mathsf{r_{p,s}}$ is within a factor $\tau \in \sR$ of the best possible ratio. Therefore,  the value of $\mathsf{\rho_{s}} \left( 1 \right)$ gives the probability that the solver $\mathsf{s}$ gives the best numerical performance when compared to the others, while $\mathsf{\rho_{s}} \left( \tau \right)$ for large values of $\tau$ measures its robustness. 

Figure \ref{fig:perpro} represents the performance profiles of the six solvers. We observe that the Fast OGDA method is the most efficient,  followed by EAG-V and Halpern-OGDA. We note that for $\tau \geq 3$ these three solvers are robust and solve $90\%$ of the problems, while Nesterov-EGA and EG solve for $\tau \geq 4$ $80\%$ of the problems.

\begin{figure}[h!]
	\centering
	\includegraphics[width=0.6 \textwidth]{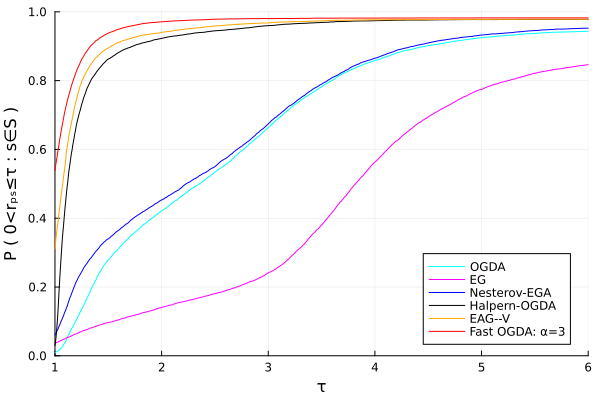}
	\caption{The performance profiles of the six solvers}
	\label{fig:perpro}
\end{figure}

\appendix

\section{Appendix}

In the first subsection of the appendix we collect some fundamental auxiliary results for the analysis carried out in the paper.  Further, we present the proof of the existence and uniqueness theorem for \eqref{ds} and also the proofs of technical lemmas used in the convergence analysis of the numerical algorithms.

\subsection{Auxiliary results}

The following result can be found in \cite[Lemma 5.1]{Abbas-Attouch-Svaiter}.
\begin{lem}
	\label{lem:lim-R}
	Let $\delta > 0$.
	Suppose that $f \colon \left[ \delta , + \infty \right) \to \sR$ is locally absolutely continuous, bounded from below, and there exists $g \in \sL^{1} \left( \left[ \delta , + \infty \right) \right)$ such that for almost every $t \geq \delta$
	\begin{equation*}
	\dfrac{d}{dt} f \left( t \right) \leq g \left( t \right) .
	\end{equation*}
	Then the limit $\lim\limits_{t \to + \infty} f \left( t \right) \in \sR$ exists.
\end{lem}

Opial’s Lemma (\cite{Opial}) in continuous form is used in the proof of the weak convergence of the trajectory of the dynamical system \eqref{ds}.

\begin{lem}
	\label{lem:Opial:cont}
	Let $\mathcal{S}$ be a nonempty subset of $\sH$ and $z \colon \left[ t_{0} , + \infty \right) \to \sH$.
	Assume that
	\begin{enumerate}
		\item
		\label{lem:Opial:cont:i}
		for every $z_{*} \in \mathcal{S}$, $\lim\limits_{t \to + \infty} \left\lVert z \left( t \right) - z_{*} \right\rVert$ exists;
		
		\item
		\label{lem:Opial:cont:ii}
		every weak sequential cluster point of the trajectory $z \left( t \right)$ as $t \to + \infty$ belongs to $\mathcal{S}$.
	\end{enumerate}
	Then $z(t)$ converges weakly to a point in $\mathcal{S}$ as $t \to + \infty$.
\end{lem}

For the convergence proof of the iterates generated by the two numerical algorithms we use the discrete counterpart of Opial’s Lemma.
\begin{lem}
	\label{lem:Opial:dis}
	Let $\mathcal{S}$ be a nonempty subset of $\sH$ and $\left( z_{k} \right) _{k \geq 1}$ be a sequence in $\sH$.
	Assume that
	\begin{enumerate}
		\item
		\label{lem:Opial:dis:i}
		for every $z_{*} \in \mathcal{S}$, $\lim\limits_{k \to + \infty} \left\lVert z_{k} - z_{*} \right\rVert$ exists;
		
		\item
		\label{lem:Opial:dis:ii}
		every weak sequential cluster point of the sequence $\left(z_{k} \right) _{k \geq 1}$ as $k \to + \infty$ belongs to $\mathcal{S}$.
	\end{enumerate}
	Then $\left(z_{k} \right)_{k \geq 1}$ converges weakly to a point in $\mathcal{S}$ as $k \to + \infty$.
\end{lem}

The following result can be found in \cite[Lemma A.2]{Attouch-Peypouquet-Redont}.
\begin{lem}
	\label{lem:lim-u}
	Let $a > 0$ and $q \colon \left[ t_{0} , + \infty \right) \rightarrow \sH$ be a continuously differentiable function such that
	\begin{equation*}
	\lim\limits_{t \to + \infty} \left(q \left( t \right) + \dfrac{t}{a} \dot{q} \left( t \right) \right) = l \in \sH .
	\end{equation*}
	Then it holds $\lim\limits_{t \to + \infty} q \left( t \right) = l$.
\end{lem}

The discrete counterpart of this result is stated below. We provide a proof for it, as we could not find any reference for this result in the literature.
\begin{lem}
	\label{lem:lim-u-k}
	Let $a \geq 1$ and $\left( q_{k} \right)_{k \geq 0}$ be a bounded sequence in $\sH$ such that
	\begin{equation*}
		\lim\limits_{k \to + \infty} \left( q_{k+1} + \dfrac{k}{a} \left( q_{k+1} - q_{k} \right) \right) = l \in \sH .
	\end{equation*}
	Then it holds $\lim\limits_{k \to + \infty} q_{k} = l$.
\end{lem}
\begin{proof}
	For every $k \geq 0$ we set $r_{k} := q_{k} - l$.  We fix $\varepsilon > 0$. Then there exists $k_{0} \geq 1$ such that for every $k \geq k_{0}$
	\begin{equation*}
		\left\lVert r_{k+1} + \dfrac{k}{a} \left( r_{k+1} - r_{k} \right) \right\rVert \leq \varepsilon .
	\end{equation*}
	Multiplying both side by $ak^{a-1}$, we obtain for every $k \geq k_{0}$
	\begin{equation*}
		\left\lVert \left( ak^{a-1} + k^{a} \right) r_{k+1} - k^{a} r_{k} \right\rVert \leq \varepsilon ak^{a-1} .
	\end{equation*}
	Then by applying the triangle inequality and using the fact that $\overline{r} := \sup_{k \geq 0} \left\lVert r_{k} \right\rVert < + \infty$, we deduce that for every $k \geq k_{0}$
	\begin{equation}	
	\label{lim-u-k:pre}	
		\left\lVert \left( k + 1 \right) ^{a} r_{k+1} - k^{a} r_{k} \right\rVert
		\leq \varepsilon ak^{a-1} + \left\lvert \left( k + 1 \right) ^{a} - k^{a} - ak^{a-1} \right\rvert \overline{r} .
	\end{equation}
	The Lagrange error bound of a Taylor series says that  for every $k \geq k_{0}$ there exists $m_k \in \left( k , k+1 \right)$ such that
	\begin{equation*}
		\left\lvert \left( k + 1 \right) ^{a} - k^{a} - ak^{a-1} \right\rvert \leq \dfrac{1}{2} a \left\lvert a - 1 \right\rvert m_k^{a-2} .
	\end{equation*}
	From here we consider two cases.
	\item[\underline{The case $1 \leq a < 2$.}]
	Then for every $k \geq k_{0}$ and every $m \in \left(k , k+1 \right)$ we have $m^{a-2} \leq 1$ and thus \eqref{lim-u-k:pre} leads to
	\begin{equation*}
		\left\lVert \left( k + 1 \right) ^{a} r_{k+1} - k^{a} r_{k} \right\rVert
		\leq \varepsilon ak^{a-1} + \dfrac{1}{2} a \left\lvert a - 1 \right\rvert \overline{r} .
	\end{equation*}
We choose $K \geq k_{0}$ and use a telescoping sum argument to get
	\begin{align*}
		\left\lVert \left( K + 1 \right) ^{a} r_{K+1} - k_{0}^{a} r_{k_{0}} \right\rVert
		& = \left\lVert \mysum_{k = k_{0}}^{K} \Bigl( \left( k + 1 \right) ^{a} r_{k+1} - k^{a} r_{k} \Bigr) \right\rVert
		\leq \mysum_{k = k_{0}}^{K} \left\lVert \left( k + 1 \right) ^{a} r_{k+1} - k^{a} r_{k} \right\rVert \nonumber \\
		& \leq \varepsilon a \mysum_{k = k_{0}}^{K} k^{a-1} + \dfrac{1}{2} a \left\lvert a - 1 \right\rvert \overline{r} \mysum_{k = k_{0}}^{K} 1
		 \leq \varepsilon a \left( K + 1 \right) ^{a} + \dfrac{1}{2} a \left\lvert a - 1 \right\rvert \overline{r} \left( K + 1 \right) .
	\end{align*}
	Once again, using the triangle inequality, we conclude that
	\begin{align*}
		\left\lVert r_{K+1} \right\rVert
		& \leq \dfrac{1}{\left( K + 1 \right) ^{a}} \left\lVert \left( K + 1 \right) ^{a} r_{K+1} - k_{0}^{a} r_{k_{0}} \right\rVert + \dfrac{k_{0}^{a}}{\left( K + 1 \right) ^{a}} \left\lVert r_{k_{0}} \right\rVert  \leq \varepsilon a + \dfrac{a \left\lvert a - 1 \right\rvert \overline{r}}{2 \left( K+1 \right) ^{a-1}} + \dfrac{k_{0}^{a}\overline{r}}{\left( K + 1 \right) ^{a}} .
	\end{align*}
	
	\item[\underline{The case $a \geq 2$.}]
For for every $k \geq k_{0}$ and every $m \in \left(k , k+1 \right)$ we have $m^{a - 2} \leq \left( k + 1 \right) ^{a-2}$, hence
	\eqref{lim-u-k:pre} leads to
	\begin{equation*}
		\left\lVert \left( k + 1 \right) ^{a} r_{k+1} - k^{a} r_{k} \right\rVert
		\leq \varepsilon ak^{a-1} + \dfrac{1}{2} a \left( a - 1 \right) \overline{r} \left( k + 1 \right) ^{a - 2} .
	\end{equation*}
We choose also in this case $K \geq k_{0}$ and by a similar argument as above we have that
	\begin{align*}
		\left\lVert \left( K + 1 \right) ^{a} r_{K+1} - k_{0}^{a} r_{k_{0}} \right\rVert
		& = \left\lVert \mysum_{k = k_{0}}^{K} \Bigl( \left( k + 1 \right) ^{a} r_{k+1} - k^{a} r_{k} \Bigr) \right\rVert
		\leq \mysum_{k = k_{0}}^{K} \left\lVert \left( k + 1 \right) ^{a} r_{k+1} - k^{a} r_{k} \right\rVert \nonumber \\
		& \leq \varepsilon a \mysum_{k = k_{0}}^{K} k^{a-1} + \dfrac{1}{2} a \left( a - 1 \right) \overline{r} \mysum_{k = k_{0}}^{K} \left( k+1 \right)^{a - 2} \nonumber \\
		& \leq \varepsilon a \left( K + 1 \right) ^{a-1} \mysum_{k = 0}^{K} 1 + \dfrac{1}{2} a \left( a - 1 \right) \overline{r} \left( K + 1 \right)^{a - 2} \mysum_{k = 0}^{K} 1 \nonumber \\
		& \leq \varepsilon a \left( K + 1 \right) ^{a} + \dfrac{1}{2} a \left( a - 1 \right) \overline{r} \left( K + 1 \right) ^{a - 1} .
	\end{align*}
	This leads to
	\begin{align*}
		\left\lVert r_{K+1} \right\rVert \leq \varepsilon a + \dfrac{a \left( a - 1 \right) \overline{r}}{2 \left( K+1 \right)} + \dfrac{k_{0}^{a} \overline{r}}{\left( K + 1 \right) ^{a}} .
	\end{align*}
Therefore, in both scenarios we obtain
	\begin{equation*}
		\limsup_{k \to + \infty} \left\lVert r_{k} \right\rVert \leq \varepsilon a ,
	\end{equation*}
	which leads to the desired conclusion, as $\varepsilon > 0$ was arbitrarily chosen.
\end{proof}

The following result is a particular instance of \cite[Lemma 5.31]{Bauschke-Combettes:book}.
\begin{lem}
	\label{lem:quasi-Fej}
	Let $\left( a_{k} \right) _{k \geq 1}$, $\left(b_{k} \right) _{k \geq 1}$ and $\left(d_{k} \right)_{k \geq 1}$ be sequences of real numbers. Assume that $\left(a_{k} \right)_{k \geq 1}$ is bounded from below, and $\left(b_{k} \right) _{k \geq 1}$ and $\left ( d_{k} \right)_{k \geq 1}$ are nonnegative  sequences such that $\sum_{k \geq 1} d_{k} < + \infty$. If
	\begin{equation*}
	a_{k+1} \leq a_{k} - b_{k} + d_{k} \quad \forall k \geq 1,
	\end{equation*}
	then the following statements are true:
	\begin{enumerate}
		\item the sequence $\left( b_{k} \right) _{k \geq 1}$ is summable, namely $\sum_{k \geq 1} b_{k} < + \infty$;
		\item the sequence $\left(a_{k} \right)_{k \geq 1}$ is convergent.
	\end{enumerate}
\end{lem}

The following elementary result is used several times in the paper.
\begin{lem}
	\label{lem:quad}
	Let $a, b, c \in \sR$ be such that $a \neq 0$ and $b^{2} - ac \leq 0$.
	The following statements are true:
	\begin{enumerate}
		\item
		\label{quad:vec-pos}
		if $a > 0$, then it holds
		\begin{equation*}
			a \left\lVert x \right\rVert ^{2} + 2b \left\langle x , y \right\rangle + c \left\lVert y \right\rVert ^{2} \geq 0 \quad \forall x, y \in \sH ;
		\end{equation*}
		
		\item
		\label{quad:vec}
		if $a < 0$, then it holds
		\begin{equation*}
		a \left\lVert x \right\rVert ^{2} + 2b \left\langle x , y \right\rangle + c \left\lVert y \right\rVert ^{2} \leq 0 \quad \forall x, y \in \sH .
		\end{equation*}
	\end{enumerate}
\end{lem}

\subsection{Proof of the existence and uniqueness theorem for the evolution equation}

In this subsection we provide the proof of the existence and uniqueness of the trajectories of \eqref{ds}.

\begin{prooff} The system \eqref{ds} can be rewritten as a first-order ordinary differential equation	
	\begin{equation}
		\label{e-u:ds}
		\begin{dcases}
			\dot{z} \left( t \right) 			& = \dfrac{1}{2t} u \left( t \right) - \dfrac{1}{t} \left( \alpha - 1 \right) z \left( t \right) - \beta \left( t \right) V \left( z \left( t \right) \right) \\
			\dot{u} \left( t \right)	& = \Bigl( t \dot{\beta} \left( t \right) + \left( 2 - \alpha \right) \beta \left( t \right) \Bigr) V \left( z \left( t \right) \right) \\
			\Bigl( z \left( t_{0} \right) , u \left( t_{0} \right) \Bigr) & = \Bigl( z^{0} , 2 \left( \alpha - 1 \right) z^{0} + 2t_{0} \dot z^{0} + 2t_{0} \beta \left( t_{0} \right) V \left( z^{0} \right) \Bigr)
		\end{dcases},
	\end{equation}
where for every $t \geq t_{0}$ we define
	\begin{equation*}
		u \left( t \right) := 2 \left( \alpha - 1 \right) z \left( t \right) + 2t \dot{z} \left( t \right) + 2t \beta \left( t \right) V \left( z \left( t \right) \right).
	\end{equation*}
We define $G \colon \left[ t_{0} , + \infty \right) \times \sH \times \sH \to \sH \times \sH$ by
	\begin{equation*}
		G \left( t , \zeta , \mu \right) := \biggl( \Bigl( t \dot{\beta} \left( t \right) + \left( 2 - \alpha \right) \beta \left( t \right) \Bigr) V \left( \zeta \right) , \dfrac{1}{2t} \mu - \dfrac{1}{t} \left( \alpha - 1 \right) \zeta - \beta \left( t \right) V \left( \zeta \right) \biggr) ,
	\end{equation*}
	so that \eqref{e-u:ds} becomes
	\begin{equation*}
		\begin{dcases}
			\Bigl( \dot{u} \left( t \right) , \dot{z} \left( t \right) \Bigr) & = G \left( t , z \left( t \right) , u \left( t \right) \right) \\
			\Bigl( z \left( t_{0} \right) , u \left( t_{0} \right) \Bigr) & = \Bigl( z^{0} , 2 \left( \alpha - 1 \right) z^{0} + 2t_{0} \dot z^{0} + 2t_{0} \beta \left( t_{0} \right) V \left( z^{0} \right) \Bigr)
		\end{dcases} .
	\end{equation*}
Since $G$ is Lipschitz continuous on bounded sets,  the local existence and uniqueness theorem (see, for instance, \cite[Theorems 46.2 and 46.3]{Sell-You}) allows to conclude that there exists a unique continuous differentiable solution $\left( z , u \right) \in \sH \times \sH$ of \eqref{e-u:ds} defined on a maximally interval $\left[ t_{0} , \Tm \right)$ where $0 < t_{0} < \Tm \leq + \infty$. Furthermore, either
	\begin{equation*}
		\Tm = + \infty \qquad \textrm{ or } \qquad \lim\limits_{t \to \Tm} \left\lVert \left( z \left( t \right) , u \left( t \right) \right) \right\rVert = + \infty .
	\end{equation*}
In the following we will show that indeed $\Tm = + \infty$.
	
 According to \eqref{rate:alpha-1}, for $z_{*} \in \sol$ fixed,  for every $t_{0} \leq t < \Tm$ it holds
	\begin{equation*}
		\E_{\alpha - 1} \left( t \right) + 2 \int_{t_{0}}^{t} \tau^{2} \beta \left( \tau \right) w \left( \tau \right) \left\lVert V \left( z \left( \tau \right) \right) \right\rVert ^{2} d \tau \leq \E_{\alpha - 1} \left( t_{0} \right) < +\infty,
	\end{equation*}
which implies that
		\begin{equation}
		\label{e-u:norm:u}
		 t \mapsto u(t)  \mbox{ is bounded on }[t_0, \Tm).
	\end{equation}
	On the other hand, inequality \eqref{rate_beta-w} implies that
	\begin{equation*}
		\int_{t_{0}}^{\Tm} \tau \beta^{2} \left( \tau \right) \left\lVert V \left( z \left( \tau \right) \right) \right\rVert ^{2} d \tau \leq \dfrac{\E_{\alpha - 1} \left( t_{0} \right)}{\varepsilon} < + \infty ,
	\end{equation*}
	for some $\varepsilon > 0$. Now for $0 < \lambda < \alpha - 1$, we have according to \eqref{rate:inq} that for every $t_{0} \leq t < \Tm$
	\begin{align}
		2 \lambda \left( \alpha - 1 - \lambda \right) \left\lVert z \left( t \right) - z_{*} \right\rVert ^{2}
		\leq \E_{\lambda} \left( t \right)
		& \leq \E_{\lambda} \left( t_{0} \right) + \dfrac{2}{\varepsilon} \left( \alpha - 1 - \lambda \right) \E_{\alpha - 1} \left( t_{0} \right) < + \infty . \label{e-u:norm:z}
	\end{align}
From \eqref{e-u:norm:u} and \eqref{e-u:norm:z} we have that $\lim_{t \to \Tm} \left\lVert \left( z \left( t \right) , u \left( t \right) \right) \right\rVert < + \infty$,  therefore $\Tm = + \infty$.
\end{prooff}

\subsection{Proof of the technical lemma used in the analysis of the implicit algorithm}

In this subsection we provide the proof of Lemma \ref{lemma8} which shows that the discrete energy \eqref{im:defi:E-k:eq} dissipates with every iteration of the implicit Fast OGDA method.

\begin{proofff} Let $0 \leq \lambda \leq \alpha-1$.  For brevity  we denote for every $k \geq 0$
	\begin{equation}
		\label{defi:u-k-lambda}
		u^{k+1}_{\lambda} := 2 \lambda \left( z^{k+1} - z_{*} \right) + 2 \left( k + 1 \right) \left( z^{k+1} - z^{k} \right) + s \left( k + 1 \right) \beta_{k} V \left( z^{k+1} \right) .
	\end{equation}
	This means that  for every $k \geq 1$ it holds
	\begin{equation*}
		u^{k}_{\lambda} = 2 \lambda \left( z^{k} - z_{*} \right) + 2k \left( z^{k} - z^{k-1} \right) + sk \beta_{k-1} V \left( z^{k} \right),
	\end{equation*}
therefore taking the difference and using \eqref{dis:d-u} we deduce that
	\begin{align}
		u^{k+1}_{\lambda} - u^{k}_{\lambda}
		= & \ 2 \left( \lambda + 1 - \alpha \right) \left( z^{k+1} - z^{k} \right) + 2 \left( k + \alpha \right) \left( z^{k+1} - z^{k} \right) - 2k \left( z^{k} - z^{k-1} \right) \nonumber \\
		& + s \Bigl( \left( k + 1 \right) \beta_{k} - k \beta_{k-1} \Bigr) V \left( z^{k+1} \right) + sk \beta_{k-1} \left( V \left( z^{k+1} \right) - V \left( z^{k} \right) \right) \nonumber \\
		 = & \ 2 \left( \lambda + 1 - \alpha \right) \left( z^{k+1} - z^{k} \right) + \left( 1 - \alpha \right) s \beta_{k} V \left( z^{k+1} \right) - sk \beta_{k-1} \left( V \left( z^{k+1} \right) - V \left( z^{k} \right) \right) . \label{defi:u-k-lambda:dif}
	\end{align}
In the following we want to use the following identity
	\begin{equation}
		\label{dec:dif:u-lambda:pre}
		\dfrac{1}{2} \left( \left\lVert u^{k+1}_{\lambda} \right\rVert ^{2} - \left\lVert u^{k}_{\lambda} \right\rVert ^{2} \right) = \left\langle u^{k+1}_{\lambda} , u^{k+1}_{\lambda} - u^{k}_{\lambda} \right\rangle - \dfrac{1}{2} \left\lVert u^{k+1}_{\lambda} - u^{k}_{\lambda} \right\rVert ^{2} \quad \forall k \geq 1.
	\end{equation}
	Using the relations \eqref{defi:u-k-lambda} and \eqref{defi:u-k-lambda:dif},  for every $k \geq 1$ we derive that
	\begin{align}
		& \left\langle u^{k+1}_{\lambda} , u^{k+1}_{\lambda} - u^{k}_{\lambda} \right\rangle \nonumber \\
		= & \ 4 \lambda \left( \lambda + 1 - \alpha \right) \left\langle z^{k+1} - z_{*} , z^{k+1} - z^{k} \right\rangle
		+ 2 \lambda \left( 1 - \alpha \right) s \beta_{k} \left\langle z^{k+1} - z_{*} , V \left( z^{k+1} \right) \right\rangle \nonumber \\
		& - 2 \lambda sk \beta_{k-1} \left\langle z^{k+1} - z_{*} , V \left( z^{k+1} \right) - V \left( z^{k} \right) \right\rangle
		+ 4 \left( \lambda + 1 - \alpha \right) \left( k + 1 \right) \left\lVert z^{k+1} - z^{k} \right\rVert ^{2} \nonumber \\
		& + 2s \left( \lambda + 2 - 2 \alpha \right) \left( k + 1 \right) \beta_{k} \left\langle z^{k+1} - z^{k} , V \left( z^{k+1} \right) \right\rangle \nonumber \\
		&  - 2s \left( k + 1 \right) k \beta_{k-1} \left\langle z^{k+1} - z^{k} , V \left( z^{k+1} \right) - V \left( z^{k} \right) \right\rangle
		+ \left( 1 - \alpha \right) s^{2} \left( k + 1 \right) \beta_{k}^{2} \left\lVert V \left( z^{k+1} \right) \right\rVert ^{2} \nonumber \\
		&  - s^{2} \left( k + 1 \right) k \beta_{k} \beta_{k-1} \left\langle V \left( z^{k+1} \right) , V \left( z^{k+1} \right) - V \left( z^{k} \right) \right\rangle , \label{dec:dif:u-lambda:inn}
	\end{align}
	and
	\begin{align}
		- \dfrac{1}{2} \left\lVert u^{k+1}_{\lambda} - u^{k}_{\lambda} \right\rVert ^{2}
		& = - 2 \left( \lambda + 1 - \alpha \right) ^{2} \left\lVert z^{k+1} - z^{k} \right\rVert ^{2}
		- \dfrac{1}{2} \left( 1 - \alpha \right) ^{2} s^{2} \beta_{k}^{2} \left\lVert V \left( z^{k+1} \right) \right\rVert ^{2} \nonumber \\
		& \quad - \dfrac{1}{2} s^{2} k^{2} \beta_{k-1}^{2} \left\lVert V \left( z^{k+1} \right) - V \left( z^{k} \right) \right\rVert ^{2} \nonumber \\
		& \quad - 2 \left( \lambda + 1 - \alpha \right) \left( 1 - \alpha \right) s \beta_{k} \left\langle z^{k+1} - z^{k} , V \left( z^{k+1} \right) \right\rangle \nonumber \\
		& \quad	+ 2 \left( \lambda + 1 - \alpha \right) sk \beta_{k-1} \left\langle z^{k+1} - z^{k} , V \left( z^{k+1} \right) - V \left( z^{k} \right) \right\rangle \nonumber \\
		& \quad + \left( 1 - \alpha \right) s^{2} k \beta_{k} \beta_{k-1} \left\langle V \left( z^{k+1} \right) , V \left( z^{k+1} \right) - V \left( z^{k} \right) \right\rangle . \label{dec:dif:u-lambda:norm}
	\end{align}
A direct computation shows that
	\begin{align}
		& \Bigl( \left( \lambda + 2 - 2 \alpha \right) \left( k + 1 \right) - \left( \lambda + 1 - \alpha \right) \left( 1 - \alpha \right) \Bigr) \beta_{k} \nonumber \\
		= \ 	& \Bigl( \left( \lambda + 1 - \alpha \right) \left( 2k + \alpha + 1 \right) - \lambda \left( k + 1 \right) \Bigr) \beta_{k} \nonumber \\
		= \ 	& \Bigl( \left( \lambda + 1 - \alpha \right) \left( 2k + \alpha + 1 \right) - \lambda \Bigr) \beta_{k} - \lambda k \left( \beta_{k} - \beta_{k-1} \right) + \lambda k \beta_{k-1} . \label{dec:coe:dz-V}
	\end{align}
By plugging \eqref{dec:dif:u-lambda:inn} and \eqref{dec:dif:u-lambda:norm} into \eqref{dec:dif:u-lambda:pre}, we get for every $k \geq 1$
	\begin{align}
		& \dfrac{1}{2} \left( \left\lVert u^{k+1}_{\lambda} \right\rVert ^{2} - \left\lVert u^{k}_{\lambda} \right\rVert ^{2} \right) \nonumber \\
		= & \ 4 \lambda \left( \lambda + 1 - \alpha \right) \left\langle z^{k+1} - z_{*} , z^{k+1} - z^{k} \right\rangle
		+ 2 \lambda \left( 1 - \alpha \right) s \beta_{k} \left\langle z^{k+1} - z_{*} , V \left( z^{k+1} \right) \right\rangle \nonumber \\
		& - 2 \lambda sk \beta_{k-1} \left\langle z^{k+1} - z_{*} , V \left( z^{k+1} \right) - V \left( z^{k} \right) \right\rangle
		+ 2 \left( \lambda + 1 - \alpha \right) \left( 2k + \alpha + 1 - \lambda \right) \left\lVert z^{k+1} - z^{k} \right\rVert ^{2} \nonumber \\
		& + 2s \Bigl( \left( \lambda + 1 - \alpha \right) \left( 2k + \alpha + 1 \right) - \lambda \left( k + 1 \right) \Bigr) \beta_{k} \left\langle z^{k+1} - z^{k} , V \left( z^{k+1} \right) \right\rangle \nonumber \\
		& - 2sk \left( k + \alpha - \lambda \right) \beta_{k-1} \! \left\langle z^{k+1} - z^{k} , V \left( z^{k+1} \right) - V \left( z^{k} \right) \right\rangle
		+ \!\dfrac{1}{2} \left( 1 - \alpha \right) s^{2} \beta_{k}^{2} \left( 2k + \alpha + 1 \right) \left\lVert V \left( z^{k+1} \right) \right\rVert ^{2} \nonumber \\
		& - s^{2} \left( k + \alpha \right) k \beta_{k} \beta_{k-1} \left\langle V \left( z^{k+1} \right) , V \left( z^{k+1} \right) - V \left( z^{k} \right) \right\rangle
		- \dfrac{1}{2} s^{2} k^{2} \beta_{k-1}^{2} \left\lVert V \left( z^{k+1} \right) - V \left( z^{k} \right) \right\rVert ^{2} . \label{dec:dif:u-lambda}
	\end{align}
Next we are going to consider the remaining terms in the difference of the discrete energy functions.  First we observe that for every $k \geq 0$
	\begin{align}
		& 2 \lambda \left( \alpha - 1 - \lambda \right) \left( \left\lVert z^{k+1} - z_{*} \right\rVert ^{2} - \left\lVert z^{k} - z_{*} \right\rVert ^{2} \right) \nonumber \\
		= \ 	& 2 \lambda \left( \alpha - 1 - \lambda \right) \left( 2 \left\langle z^{k+1} - z_{*} , z^{k+1} - z^{k} \right\rangle - \left\lVert z^{k+1} - z^{k} \right\rVert ^{2} \right) . \label{dec:dif:norm}
	\end{align}
Some algebra shows that for every $k \geq 1$
	\begin{align}
		& 2 \lambda s \left( k + 1 \right) \beta_{k} \left\langle z^{k+1} - z_{*} , V \left( z^{k+1} \right) \right\rangle - 2 \lambda s k \beta_{k-1} \left\langle z^{k} - z_{*} , V \left( z^{k} \right) \right\rangle \nonumber \\
		= \ & 2 \lambda s \Bigl( \left( k + 1 \right) \beta_{k} - k \beta_{k-1} \Bigr) \left\langle z^{k+1} - z_{*} , V \left( z^{k+1} \right) \right\rangle \nonumber \\
		& \quad + 2 \lambda sk \beta_{k-1} \left( \left\langle z^{k+1} - z_{*} , V \left( z^{k+1} \right) \right\rangle - \left\langle z^{k} - z_{*} , V \left( z^{k} \right) \right\rangle \right) \nonumber \\
		= \		& 2 \lambda s \Bigl( \left( k + 1 \right) \beta_{k} - k \beta_{k-1} \Bigr) \left\langle z^{k+1} - z_{*} , V \left( z^{k+1} \right) \right\rangle \nonumber \\
		& \quad + 2 \lambda sk \beta_{k-1} \left\langle z^{k+1} - z_{*} , V \left( z^{k+1} \right) - V \left( z^{k} \right) \right\rangle
		+ 2 \lambda sk \beta_{k-1} \left\langle z^{k+1} - z^{k} , V \left( z^{k} \right) \right\rangle \nonumber \\
		= \ 	& 2 \lambda s \Bigl( \left( k + 1 \right) \beta_{k} - k \beta_{k-1} \Bigr) \left\langle z^{k+1} - z_{*} , V \left( z^{k+1} \right) \right\rangle
		- 2 \lambda sk \beta_{k-1} \left\langle z^{k+1} - z^{k} , V \left( z^{k+1} \right) - V \left( z^{k} \right) \right\rangle \nonumber \\
		& \quad + 2 \lambda sk \beta_{k-1} \left\langle z^{k+1} - z_{*} , V \left( z^{k+1} \right) - V \left( z^{k} \right) \right\rangle
		+ 2 \lambda sk \beta_{k-1} \left\langle z^{k+1} - z^{k} , V \left( z^{k+1} \right) \right\rangle . \label{dec:dif:vi}
	\end{align}
	Finally, according to \eqref{as:beta-k:eps} and \eqref{as:beta-k}, we have for every $k \geq \left\lceil \alpha \right\rceil$
	\begin{align*}
		& \left( k + \alpha + 1 \right) \left( k + 1 \right) \beta_{k+1} - \left( k + \alpha \right) k \beta_{k-1} \nonumber \\
		= \		& \left( k + \alpha + 1 \right) \left( k + 1 \right) \bigl( \beta_{k+1} - \beta_{k} \bigr) + \left( k + \alpha \right) k \bigl( \beta_{k} - \beta_{k-1} \bigr) + \left( 2k + \alpha + 1 \right) \beta_{k} \nonumber \\
		\leq \	& \left( \alpha - 2 - \varepsilon \right) \Bigl( \left( k + \alpha + 1 \right) \beta_{k+1} + \left( k + \alpha \right) \beta_{k} \Bigr) + \left( 2k + \alpha + 1 \right) \beta_{k} \nonumber \\
		= \ 	& \left( \alpha - 2 - \varepsilon \right) \Bigl( \left( k + 1 \right) \left( \beta_{k+1} - \beta_{k} \right) + \alpha \beta_{k+1} + \left( 2k + \alpha + 1 \right) \beta_{k} \Bigr) + \left( 2k + \alpha + 1 \right) \beta_{k} \nonumber \\
		\leq \ 	& \left( \alpha - 2 - \varepsilon \right) \left( 2 \alpha - 2 - \varepsilon \right) \beta_{k+1} + \left( \alpha - 1 - \varepsilon \right) \left( 2k + \alpha + 1 \right) \beta_{k} \nonumber \\
		\leq \ 	& \dfrac{\alpha}{2 + \varepsilon} \left( \alpha - 2 - \varepsilon \right) \left( 2 \alpha - 2 - \varepsilon \right) \beta_{k} + \left( \alpha - 1 - \varepsilon \right) \left( 2k + \alpha + 1 \right) \beta_{k} 
	\end{align*}
	and thus it holds
	\begin{align}
		& \ \dfrac{1}{2} s^{2} \left( k + \alpha + 1 \right) \left( k + 1 \right) \beta_{k+1} \beta_{k} \left\lVert V \left( z^{k+1} \right) \right\rVert ^{2} - \dfrac{1}{2} s^{2} \left( k + \alpha \right) k \beta_{k} \beta_{k-1} \left\lVert V \left( z^{k} \right) \right\rVert ^{2} \nonumber \\
		= & \ \dfrac{1}{2} s^{2} \Bigl( \left( k + \alpha + 1 \right) \left( k + 1 \right) \beta_{k+1} - \left( k + \alpha \right) k \beta_{k-1} \Bigr) \beta_{k} \left\lVert V \left( z^{k+1} \right) \right\rVert ^{2} \nonumber \\
		& + \dfrac{1}{2} s^{2} \left( k + \alpha \right) k \beta_{k} \beta_{k-1} \left( \left\lVert V \left( z^{k+1} \right) \right\rVert ^{2} - \left\lVert V \left( z^{k} \right) \right\rVert ^{2} \right) \nonumber \\
		\leq \ 	& \dfrac{1}{2} \Bigl( \dfrac{\alpha}{2 + \varepsilon} \left( \alpha - 2 - \varepsilon \right) \left( 2 \alpha - 2 - \varepsilon \right) + \left( \alpha - 1 - \varepsilon \right) \left( 2k + \alpha + 1 \right) \Bigr) s^{2} \beta_{k}^{2} \left\lVert V \left( z^{k+1} \right) \right\rVert ^{2} \nonumber \\
		& + s^{2} \left( k + \alpha \right) k \beta_{k} \beta_{k-1} \left\langle V \left( z^{k+1} \right) , V \left( z^{k+1} \right) - V \left( z^{k} \right) \right\rangle \nonumber \\
		& - \dfrac{1}{2} s^{2} \left( k + \alpha \right) k \beta_{k} \beta_{k-1} \left\lVert V \left( z^{k+1} \right) - V \left( z^{k} \right) \right\rVert ^{2} . \label{dec:dif:eq}
	\end{align}
After adding the relations \eqref{dec:dif:u-lambda} - \eqref{dec:dif:eq} and by taking into consideration  \eqref{dec:coe:dz-V}, we obtain \eqref{dec:inq}.

\end{proofff}

\subsection{Proofs of the technical lemmas used in the analysis of the explicit algorithm}

In this subsection we provide the proofs of the two main technical lemmas used in the analysis of the explicit Fast OGDA method.

\begin{prooffff} Let $z_* \in {\cal Z}$, $0 < \gamma < 2$ and  $0 \leq \lambda \leq \alpha - 1$.  First we prove  that for every $k \geq 1$
\begin{align}
\E_{\lambda}^{k+1} - \E_{\lambda}^{k}
= & 2 \lambda \left( 2 - \alpha \right) s \left\langle z^{k+1} - z_{*} , V \left( \bz^{k} \right) \right\rangle
+ 2 \left( \lambda + 1 - \alpha \right) \left( 2k + \alpha + 1 \right) \left\lVert z^{k+1} - z^{k} \right\rVert ^{2} \nonumber \\
&  \!\!+ 2s \Bigl( \!\!\bigl( \left( 2 - \gamma \right) \lambda + \gamma - \alpha + \gamma \left( \lambda + 1 - \alpha \right) \bigr) k + \gamma - \alpha + \alpha \left( \lambda + 1 - \alpha \right) \Bigr) \! \left\langle z^{k+1} - z^{k} , V \left( \bz^{k} \right) \right\rangle \nonumber \\
& - 2 \left( 2 - \gamma \right) sk \left( k + \alpha \right) \left\langle z^{k+1} - z^{k} , V \left( \bz^{k} \right) - V \left( \bz^{k-1} \right) \right\rangle \nonumber \\
& + \dfrac{1}{2} \left( 2 - \alpha \right) s^{2} \left( 2 \gamma k + \alpha + \gamma \right) \left\lVert V \left( \bz^{k} \right) \right\rVert ^{2}
\!- \dfrac{1}{2} \left( 2 - \gamma \right) s^{2} k \left( 2k + \alpha \right) \left\lVert V \left( \bz^{k} \right) - V \left( \bz^{k-1} \right) \right\rVert ^{2}.\label{ex:dE}
\end{align}

For every $k \geq 1$ we have
\begin{equation}
	\label{ex:defi:u-k-lambda}
	u^{k+1}_{\lambda} := 2 \lambda \left( z^{k+1} - z_{*} \right) + 2 \left( k + 1 \right) \left( z^{k+1} - z^{k} \right) + \gamma s \left( k + 1 \right) V \left( \bz^{k} \right),
\end{equation}
and after substraction we deduce from \eqref{ex:d-u} that
\begin{align}
	u^{k+1}_{\lambda} - u^{k}_{\lambda}
	& = 2 \left( \lambda + 1 - \alpha \right) \left( z^{k+1} - z^{k} \right) + 2 \left( k + \alpha \right) \left( z^{k+1} - z^{k} \right) - 2k \left( z^{k} - z^{k-1} \right) \nonumber \\
	& \quad + \gamma s V \left( \bz^{k} \right) + \gamma sk \left( V \left( \bz^{k} \right) - V \left( \bz^{k-1} \right) \right) \nonumber \\
	& = 2 \left( \lambda + 1 - \alpha \right) \left( z^{k+1} - z^{k} \right) + \left( \gamma - \alpha \right) s V \left( \bz^{k} \right) + \left( \gamma - 2 \right) sk \left( V \left( \bz^{k} \right) - V \left( \bz^{k-1} \right) \right) . \label{ex:defi:u-k-lambda:dif}
\end{align}
Next we recall the identities in \eqref{dec:dif:u-lambda:pre} and \eqref{dec:dif:norm}
\begin{align}
\dfrac{1}{2} \left( \left\lVert u^{k+1}_{\lambda} \right\rVert ^{2} - \left\lVert u^{k}_{\lambda} \right\rVert ^{2} \right)
= & \ \left\langle u^{k+1}_{\lambda} , u^{k+1}_{\lambda} - u^{k}_{\lambda} \right\rangle - \dfrac{1}{2} \left\lVert u^{k+1}_{\lambda} - u^{k}_{\lambda} \right\rVert ^{2} \quad \forall k \geq 1, \label{ex:dif:u-lambda:pre} \\
2 \lambda \left( \alpha - 1 - \lambda \right) \left( \left\lVert z^{k+1} - z_{*} \right\rVert ^{2} - \left\lVert z^{k} - z_{*} \right\rVert ^{2} \right)
= & \ 4 \lambda \left( \alpha - 1 - \lambda \right) \left\langle z^{k+1} - z_{*} , z^{k+1} - z^{k} \right\rangle \nonumber \\
& - 2 \lambda \left( \alpha - 1 - \lambda \right) \left\lVert z^{k+1} - z^{k} \right\rVert ^{2} \quad \forall k \geq 0, \label{ex:dif:norm}
\end{align}
respectively, as they are required also in the analysis of the explicit algorithm.

We first use the relations \eqref{ex:defi:u-k-lambda} and \eqref{ex:defi:u-k-lambda:dif} to derive for every $k \geq 1$ that
\begin{align}
& \left\langle u^{k+1}_{\lambda} , u^{k+1}_{\lambda} - u^{k}_{\lambda} \right\rangle \nonumber \\
= & \ 4 \lambda \left( \lambda + 1 - \alpha \right) \left\langle z^{k+1} - z_{*} , z^{k+1} - z^{k} \right\rangle
+ 2 \lambda \left( \gamma - \alpha \right) s \left\langle z^{k+1} - z_{*} , V \left( \bz^{k} \right) \right\rangle \nonumber \\
& + 2 \lambda \left( \gamma - 2 \right) sk \left\langle z^{k+1} - z_{*} , V \left( \bz^{k} \right) - V \left( \bz^{k-1} \right) \right\rangle
+ 4 \left( \lambda + 1 - \alpha \right) \left( k + 1 \right) \left\lVert z^{k+1} - z^{k} \right\rVert ^{2} \nonumber \\
&  + 2 \Bigl( \gamma - \alpha + \gamma \left( \lambda + 1 - \alpha \right) \Bigr) s \left( k + 1 \right) \left\langle z^{k+1} - z^{k} , V \left( \bz^{k} \right) \right\rangle \nonumber \\
& + 2 \left( \gamma - 2 \right) s \left( k + 1 \right) k \left\langle z^{k+1} - z^{k} , V \left( \bz^{k} \right) - V \left( \bz^{k-1} \right) \right\rangle
+ \gamma \left( \gamma - \alpha \right) s^{2} \left( k + 1 \right) \left\lVert V \left( \bz^{k} \right) \right\rVert ^{2} \nonumber \\
& + \gamma \left( \gamma - 2 \right) s^{2} \left( k + 1 \right) k \left\langle V \left( \bz^{k} \right) , V \left( \bz^{k} \right) - V \left( \bz^{k-1} \right) \right\rangle , \label{ex:dif:u-lambda:inn}
\end{align}
and
\begin{align}
- \dfrac{1}{2} \left\lVert u^{k+1}_{\lambda} - u^{k}_{\lambda} \right\rVert ^{2}
& = - 2 \left( \lambda + 1 - \alpha \right) ^{2} \left\lVert z^{k+1} - z^{k} \right\rVert ^{2}
- \dfrac{1}{2} \left( \gamma - \alpha \right) ^{2} s^{2} \left\lVert V \left( \bz^{k} \right) \right\rVert ^{2} \nonumber \\
& \quad - \dfrac{1}{2} \left( \gamma - 2 \right) ^{2} s^{2} k^{2} \left\lVert V \left( \bz^{k} \right) - V \left( \bz^{k-1} \right) \right\rVert ^{2} \nonumber \\
& \quad - 2 \left( \lambda + 1 - \alpha \right) \left( \gamma - \alpha \right) s \left\langle z^{k+1} - z^{k} , V \left( \bz^{k} \right) \right\rangle \nonumber \\
& \quad	- 2 \left( \lambda + 1 - \alpha \right) \left( \gamma - 2 \right) sk \left\langle z^{k+1} - z^{k} , V \left( \bz^{k} \right) - V \left( \bz^{k-1} \right) \right\rangle \nonumber \\
& \quad - \left( \gamma - 2 \right) \left( \gamma - \alpha \right) s^{2} k \left\langle V \left( \bz^{k} \right) , V \left( \bz^{k} \right) - V \left( \bz^{k-1} \right) \right\rangle . \label{ex:dif:u-lambda:norm}
\end{align}
A direct computation shows that
\begin{align*}
	& \bigl( \gamma - \alpha + \gamma \left( \lambda + 1 - \alpha \right) \bigr) \left( k + 1 \right) - \left( \lambda + 1 - \alpha \right) \left( \gamma - \alpha \right) \nonumber \\
	= \ 	& \bigl( \gamma - \alpha + \gamma \left( \lambda + 1 - \alpha \right) \bigr) k + \gamma - \alpha + \alpha \left( \lambda + 1 - \alpha \right) ,
\end{align*}
therefore, by replacing \eqref{ex:dif:u-lambda:inn} and \eqref{ex:dif:u-lambda:norm} into \eqref{ex:dif:u-lambda:pre}, we get for every $k \geq 1$
\begin{align}
& \ \dfrac{1}{2} \left( \left\lVert u^{k+1}_{\lambda} \right\rVert ^{2} - \left\lVert u^{k}_{\lambda} \right\rVert ^{2} \right) \nonumber \\
= & \ 4 \lambda \left( \lambda + 1 - \alpha \right) \left\langle z^{k+1} - z_{*} , z^{k+1} - z^{k} \right\rangle
+ 2 \lambda \left( \gamma - \alpha \right) s \left\langle z^{k+1} - z_{*} , V \left( \bz^{k} \right) \right\rangle \nonumber \\
& + 2 \lambda \left( \gamma - 2 \right) sk \left\langle z^{k+1} - z_{*} , V \left( \bz^{k} \right) - V \left( \bz^{k-1} \right) \right\rangle
+ 2 \left( \lambda + 1 - \alpha \right) \left( 2k + \alpha + 1 - \lambda \right) \left\lVert z^{k+1} - z^{k} \right\rVert ^{2} \nonumber \\
& + 2s \Bigl( \bigl( \gamma - \alpha + \gamma \left( \lambda + 1 - \alpha \right) \bigr) k + \gamma - \alpha + \alpha \left( \lambda + 1 - \alpha \right) \Bigr) \left\langle z^{k+1} - z^{k} , V \left( \bz^{k} \right) \right\rangle \nonumber \\
& + 2 \left( \gamma - 2 \right) sk \left( k + \alpha - \lambda \right)\! \left\langle z^{k+1} - z^{k} , V \left( \bz^{k} \right) - V \left( \bz^{k-1} \right) \right\rangle
+ \!\dfrac{1}{2} \left( \gamma - \alpha \right) s^{2} \left( 2 \gamma k + \alpha + \gamma \right) \left\lVert V \left( \bz^{k} \right) \right\rVert ^{2} \nonumber \\
& + \left( \gamma - 2 \right) s^{2} k \left( \gamma k + \alpha \right) \left\langle V \left( \bz^{k} \right) , V \left( \bz^{k} \right) - V \left( \bz^{k-1} \right) \right\rangle
- \dfrac{1}{2} \left( \gamma - 2 \right) ^{2} s^{2} k^{2} \left\lVert V \left( \bz^{k} \right) - V \left( \bz^{k-1} \right) \right\rVert ^{2} . \label{ex:dif:u-lambda}
\end{align}
Furthermore, one can show that for every $k \geq 1$ it holds
\begin{align}
& \ 2 \lambda s \left( k + 1 \right) \left\langle z^{k+1} - z_{*} , V \left( \bz^{k} \right) \right\rangle - 2 \lambda sk \left\langle z^{k} - z_{*} , V \left( \bz^{k-1} \right) \right\rangle \nonumber \\
= & \ 2 \lambda s \left\langle z^{k+1} - z_{*} , V \left( \bz^{k} \right) \right\rangle
+ 2 \lambda sk \left( \left\langle z^{k+1} - z_{*} , V \left( \bz^{k} \right) \right\rangle - \left\langle z^{k} - z_{*} , V \left( \bz^{k-1} \right) \right\rangle \right) \nonumber \\
= & \ 2 \lambda s \left\langle z^{k+1} - z_{*} , V \left( \bz^{k} \right) \right\rangle
+ 2 \lambda sk \left\langle z^{k+1} - z_{*} , V \left( \bz^{k} \right) - V \left( \bz^{k-1} \right) \right\rangle \nonumber \\
& - 2 \lambda sk \left\langle z^{k+1} - z^{k} , V \left( \bz^{k} \right) - V \left( \bz^{k-1} \right) \right\rangle + 2 \lambda sk \left\langle z^{k+1} - z^{k} , V \left( \bz^{k} \right) \right\rangle . \label{ex:dif:vi}
\end{align}
and
\begin{align}
& \ \dfrac{1}{2} s^{2} \left( k + 1 \right) \bigl( \gamma \left( k + 1 \right) + \alpha \bigr) \left\lVert V \left( \bz^{k} \right) \right\rVert ^{2} - \dfrac{1}{2} s^{2} k \left( \gamma k + \alpha \right) \left\lVert V \left( \bz^{k-1} \right) \right\rVert ^{2} \nonumber \\
= & \ \dfrac{1}{2} s^{2} \left( 2 \gamma k + \alpha + \gamma \right) \left\lVert V \left( \bz^{k} \right) \right\rVert ^{2}
+ \dfrac{1}{2} s^{2} k \left( \gamma k + \alpha \right) \left( \left\lVert V \left( \bz^{k} \right) \right\rVert ^{2} - \left\lVert V \left( \bz^{k-1} \right) \right\rVert ^{2} \right) \nonumber \\
= & \ \dfrac{1}{2} s^{2} \left( 2 \gamma k + \alpha + \gamma \right) \left\lVert V \left( \bz^{k} \right) \right\rVert ^{2}
+ s^{2} k \left( \gamma k + \alpha \right) \left\langle V \left( \bz^{k} \right) , V \left( \bz^{k} \right) - V \left( \bz^{k-1} \right) \right\rangle \nonumber \\
& - \dfrac{1}{2} s^{2} k \left( \gamma k + \alpha \right) \left\lVert V \left( \bz^{k} \right) - V \left( \bz^{k-1} \right) \right\rVert ^{2} . \label{ex:dif:eq}
\end{align}
Hence, multiplying \eqref{ex:dif:vi} and \eqref{ex:dif:eq} by $2 - \gamma > 0$, and summing up the resulting identities with \eqref{ex:dif:norm} and \eqref{ex:dif:u-lambda}, we obtain \eqref{ex:dE}.

(i) Let $k \geq 2$ be fixed.
	By the definition of $\F_{\lambda}^{k}$ in \eqref{ex:defi:F}, we have for every $k \geq 2$
	\begin{align}
		& \ \F_{\lambda}^{k+1} - \F_{\lambda}^{k} \nonumber \\
		= & \ \E_{\lambda}^{k+1} - \E_{\lambda}^{k}
		- \dfrac{1}{2} \lambda \left( \alpha - 2 \right) s^{2} \left[ \left( 2 - \frac{\alpha}{k + \alpha + 1} \right) \left\lVert V \left( \bz^{k} \right) \right\rVert ^{2}
		- \left( 2 - \frac{\alpha}{k + \alpha} \right) \left\lVert V \left( \bz^{k-1} \right) \right\rVert ^{2} \right] \nonumber \\
		& - 2s \left( 2 - \gamma \right) \left[ \left( k + 1 \right) ^{2} \left\langle z^{k+1} - z^{k} , V \left( z^{k+1} \right) - V \left( \bz^{k} \right) \right\rangle
		- k^{2} \left\langle z^{k} - z^{k-1} , V \left( z^{k} \right) - V \left( \bz^{k-1} \right) \right\rangle \right] \nonumber \\
		& + \dfrac{1}{2} \left( 2 - \gamma \right) \alpha s^{2} \left[ \left( k+1 \right) \sqrt{k+1} \left\lVert V \left( \bz^{k} \right) - V \left( \bz^{k-1} \right) \right\rVert ^{2}
		- k \sqrt{k} \left\lVert V \left( \bz^{k-1} \right) - V \left( \bz^{k-2} \right) \right\rVert ^{2} \right] \nonumber \\
		& + \left( 2 - \gamma \right) s^{3} L \left[\left( k+1 \right)^{2} \left\lVert V \left( \bz^{k} \right) - V \left( \bz^{k-1} \right) \right\rVert ^{2}
		- k^{2} \left\lVert V \left( \bz^{k-1} \right) - V \left( \bz^{k-2} \right) \right\rVert ^{2} \right] . \label{reg:dF}
	\end{align}
By using the definition of $\omega_{0}, \omega_{1}, \omega_{2}$ and $\omega_{4}$  in \eqref{reg:const} the fact that $0 \leq \lambda \leq \alpha - 1$ and $0 < \gamma < 2$, from \eqref{ex:dE} we obtain that for every $k \geq 1$ it holds
	\begin{align}
	& \E_{\lambda}^{k+1} - \E_{\lambda}^{k} \nonumber \\
	= & \ 2 \lambda \left( 2 - \alpha \right) s \left\langle z^{k+1} - z_{*} , V \left( \bz^{k} \right) \right\rangle
	- 2 \left( 2 - \gamma \right) sk \left( k + \alpha \right) \left\langle z^{k+1} - z^{k} , V \left( \bz^{k} \right) - V \left( \bz^{k-1} \right) \right\rangle \nonumber \\
	& + 2s \left( \omega_{0} k + \omega_{1} \right) \left\langle z^{k+1} - z^{k} , V \left( \bz^{k} \right) \right\rangle
	+ 2 \bigl( \omega_{2} k + \left( \lambda + 1 - \alpha \right) \left( \alpha + 1 \right) \bigr) \left\lVert z^{k+1} - z^{k} \right\rVert ^{2} \nonumber \\
	& + \dfrac{1}{2} s^{2} \bigl( \omega_{4} k + \left( 2 - \alpha \right) \left( \alpha + \gamma \right) \bigr) \left\lVert V \left( \bz^{k} \right) \right\rVert ^{2}
	- \dfrac{1}{2} \left( 2 - \gamma \right) s^{2} k \left( 2k + \alpha \right) \left\lVert V \left( \bz^{k} \right) - V \left( \bz^{k-1} \right) \right\rVert ^{2} \nonumber \\
	\leq & \ 2 \lambda \left( 2 - \alpha \right) s \left\langle z^{k+1} - z_{*} , V \left( \bz^{k} \right) \right\rangle
	- 2 \left( 2 - \gamma \right) sk \left( k + \alpha \right) \left\langle z^{k+1} - z^{k} , V \left( \bz^{k} \right) - V \left( \bz^{k-1} \right) \right\rangle \nonumber \\
	& + 2s \left( \omega_{0} k + \omega_{1} \right) \left\langle z^{k+1} - z^{k} , V \left( \bz^{k} \right) \right\rangle
	+ 2 \omega_{2} k \left\lVert z^{k+1} - z^{k} \right\rVert ^{2}
	+ \dfrac{1}{2} s^{2} \omega_{4} k \left\lVert V \left( \bz^{k} \right) \right\rVert ^{2} \nonumber \\
	& - \dfrac{1}{2} \left( 2 - \gamma \right) s^{2} k \left( 2k + \alpha \right) \left\lVert V \left( \bz^{k} \right) - V \left( \bz^{k-1} \right) \right\rVert ^{2}.\label{reg:dE}
	\end{align}
Plugging \eqref{reg:dE} into \eqref{reg:dF}, it yields for every $k \geq 2$
	\begin{align}
	& \ \F_{\lambda}^{k+1} - \F_{\lambda}^{k} \nonumber \\
	\leq & \ 2 \lambda \left( 2 - \alpha \right) s \left\langle z^{k+1} - z_{*} , V \left( \bz^{k} \right) \right\rangle
	- 2 \left( 2 - \gamma \right) sk \left( k + \alpha \right) \left\langle z^{k+1} - z^{k} , V \left( \bz^{k} \right) - V \left( \bz^{k-1} \right) \right\rangle \nonumber \\
	& - \dfrac{1}{2} \lambda \left( \alpha - 2 \right) s^{2} \left[ \left( 2 - \frac{\alpha}{k + \alpha + 1} \right) \left\lVert V \left( \bz^{k} \right) \right\rVert ^{2}
	- \left( 2 - \frac{\alpha}{k + \alpha} \right) \left\lVert V \left( \bz^{k-1} \right) \right\rVert ^{2} \right] \nonumber \\
	& - 2s \left( 2 - \gamma \right) \left[ \left( k + 1 \right) ^{2} \left\langle z^{k+1} - z^{k} , V \left( z^{k+1} \right) - V \left( \bz^{k} \right) \right\rangle
	- k^{2} \left\langle z^{k} - z^{k-1} , V \left( z^{k} \right) - V \left( \bz^{k-1} \right) \right\rangle \right] \nonumber \\
	& + \dfrac{1}{2} \left( 2 - \gamma \right) \alpha s^{2} \left[ \left( k+1 \right) \sqrt{k+1} \left\lVert V \left( \bz^{k} \right) - V \left( \bz^{k-1} \right) \right\rVert ^{2}
	- k \sqrt{k} \left\lVert V \left( \bz^{k-1} \right) - V \left( \bz^{k-2} \right) \right\rVert ^{2} \right] \nonumber \\
	& + \left( 2 - \gamma \right) s^{3} L \left[\left( k+1 \right)^{2} \left\lVert V \left( \bz^{k} \right) - V \left( \bz^{k-1} \right) \right\rVert ^{2}
	- k^{2} \left\lVert V \left( \bz^{k-1} \right) - V \left( \bz^{k-2} \right) \right\rVert ^{2} \right] \nonumber \\
	& + 2s \left( \omega_{0} k + \omega_{1} \right) \left\langle z^{k+1} - z^{k} , V \left( \bz^{k} \right) \right\rangle
	+ 2 \omega_{2} k \left\lVert z^{k+1} - z^{k} \right\rVert ^{2}
	+ \dfrac{1}{2} s^{2} \omega_{4} k \left\lVert V \left( \bz^{k} \right) \right\rVert ^{2} \nonumber \\
	& - \dfrac{1}{2} \left( 2 - \gamma \right) s^{2} \left( 2k^{2} + \alpha k \right) \left\lVert V \left( \bz^{k} \right) - V \left( \bz^{k-1} \right) \right\rVert ^{2} . \label{reg:pre}
	\end{align}
Our next aim is to derive upper estimates for the first two terms on the right-hand side of \eqref{reg:pre}, which will eventually simplify the subsequent four terms. First we observe that from \eqref{ex:Nes-scheme} we have for every $k \geq 1$
	\begin{align}
		& \ 2 \lambda \left( 2 - \alpha \right) s \left\langle z^{k+1} - z_{*} , V \left( \bz^{k} \right) \right\rangle \nonumber \\
		= & \ 2 \lambda \left( 2 - \alpha \right) s \left\langle z^{k+1} - \bz^{k} , V \left( \bz^{k} \right) \right\rangle + 2 \lambda \left( 2 - \alpha \right) s \left\langle \bz^{k} - z_{*} , V \left( \bz^{k} \right) \right\rangle \nonumber \\
		= & \ \lambda \left( \alpha - 2 \right) s^{2} \left( 1 + \frac{k}{k + \alpha} \right) \left\langle V \left( \bz^{k} \right) - V \left( \bz^{k-1} \right) , V \left( \bz^{k} \right) \right\rangle + 2 \lambda \left( 2 - \alpha \right) s \left\langle \bz^{k} - z_{*} , V \left( \bz^{k} \right) \right\rangle \nonumber \\
		= & \ \dfrac{1}{2} \lambda \left( \alpha - 2 \right) s^{2} \left( 2 - \frac{\alpha}{k + \alpha} \right) \left\lVert V \left( \bz^{k} \right) - V \left( \bz^{k-1} \right) \right\rVert ^{2} + \dfrac{1}{2} \lambda \left( \alpha - 2 \right) s^{2} \left( 2 - \frac{\alpha}{k + \alpha} \right) \left\lVert V \left( \bz^{k} \right) \right\rVert ^{2} \nonumber \\
		&  - \dfrac{1}{2} \lambda \left( \alpha - 2 \right) s^{2} \left( 2 - \frac{\alpha}{k + \alpha} \right) \left\lVert V \left( \bz^{k-1} \right) \right\rVert ^{2} + 2 \lambda \left( 2 - \alpha \right) s \left\langle \bz^{k} - z_{*} , V \left( \bz^{k} \right) \right\rangle \nonumber \\
		\leq & \ \lambda \left( \alpha - 2 \right) s^{2} \left\lVert V \left( \bz^{k} \right) - V \left( \bz^{k-1} \right) \right\rVert ^{2} + \dfrac{1}{2} \lambda \left( \alpha - 2 \right) s^{2} \left( 2 - \frac{\alpha}{k + \alpha + 1} \right) \left\lVert V \left( \bz^{k} \right) \right\rVert ^{2} \nonumber \\
		& - \dfrac{1}{2} \lambda \left( \alpha - 2 \right) s^{2} \left( 2 - \frac{\alpha}{k + \alpha} \right) \left\lVert V \left( \bz^{k-1} \right) \right\rVert ^{2} + 2 \lambda \left( 2 - \alpha \right) s \left\langle \bz^{k} - z_{*} , V \left( \bz^{k} \right) \right\rangle . \label{reg:inn}
	\end{align}	
The monotonicity of $V$ and relation \eqref{ex:d-u} yield for every $k \geq 1$
	\begin{align}
		& - 2sk \left( k + \alpha \right) \left\langle z^{k+1} - z^{k} , V \left( \bz^{k} \right) - V \left( \bz^{k-1} \right) \right\rangle \nonumber \\
		\leq & \ 2 sk \left( k + \alpha \right) \left\langle z^{k+1} - z^{k} , \Bigl( V \left( z^{k+1} \right) - V \left( \bz^{k} \right) \Bigr) - \Bigl( V \left( z^{k} \right) - V \left( \bz^{k-1} \right) \Bigr) \right\rangle \nonumber \\
		= & \ 2 sk \left( k + \alpha \right) \left\langle z^{k+1} - z^{k} , V \left( z^{k+1} \right) - V \left( \bz^{k} \right) \right\rangle
		- 2 sk \left( k + \alpha \right) \left\langle z^{k+1} - z^{k} , V \left( z^{k} \right) - V \left( \bz^{k-1} \right) \right\rangle \nonumber \\
		= & \ 2s \left( k + 1 \right) ^{2} \left\langle z^{k+1} - z^{k} , V \left( z^{k+1} \right) - V \left( \bz^{k} \right) \right\rangle
		- 2 sk^{2} \left\langle z^{k} - z^{k-1} , V \left( z^{k} \right) - V \left( \bz^{k-1} \right) \right\rangle \nonumber \\
		& + 2s \bigl( \left( \alpha - 2 \right) k - 1 \bigr) \left\langle z^{k+1} - z^{k} , V \left( z^{k+1} \right) - V \left( \bz^{k} \right) \right\rangle
		+ \alpha s^{2} k \left\langle V \left( \bz^{k} \right) , V \left( z^{k} \right) - V \left( \bz^{k-1} \right) \right\rangle \nonumber \\
		& + 2s^{2} k^{2} \left\langle V \left( \bz^{k} \right) - V \left( \bz^{k-1} \right) , V \left( z^{k} \right) - V \left( \bz^{k-1} \right) \right\rangle . \label{reg:mono}
	\end{align}
Young's inequality together with \eqref{ex:Lip} show that for every $k \geq \left\lceil \frac{1}{\alpha - 2} \right\rceil$ it holds
	\begin{align}
	& 2s \bigl( \left( \alpha - 2 \right) k - 1 \bigr) \left\langle z^{k+1} - z^{k} , V \left( z^{k+1} \right) - V \left( \bz^{k} \right) \right\rangle \nonumber \\
	\leq \ 	& 2 \sqrt{\left( \alpha - 2 \right) k - 1} \left\lVert z^{k+1} - z^{k} \right\rVert ^{2} + \dfrac{1}{2} s^{2} \bigl( \left( \alpha - 2 \right) k - 1 \bigr) \sqrt{\left( \alpha - 2 \right) k - 1} \left\lVert V \left( z^{k+1} \right) - V \left( \bz^{k} \right) \right\rVert ^{2} \nonumber \\
	\leq \ 	& 2 \sqrt{\left( \alpha - 2 \right) k} \left\lVert z^{k+1} - z^{k} \right\rVert ^{2} + \dfrac{1}{2} \left( \alpha - 1 \right) \sqrt{\alpha - 1} s^{2} \left( k+1 \right) \sqrt{k+1} \left\lVert V \left( z^{k+1} \right) - V \left( \bz^{k} \right) \right\rVert ^{2} \nonumber \\
	\leq \ 	& 2 \sqrt{\left( \alpha - 2 \right) k} \left\lVert z^{k+1} - z^{k} \right\rVert ^{2} + \dfrac{1}{2} \left( \alpha - 1 \right) \sqrt{\alpha - 1} s^{4} L^{2} \left( k+1 \right) \sqrt{k+1} \left\lVert V \left( \bz^{k} \right) - V \left( \bz^{k-1} \right) \right\rVert ^{2} \nonumber \\
	\leq \ 	& 2 \sqrt{\left( \alpha - 2 \right) k} \left\lVert z^{k+1} - z^{k} \right\rVert ^{2} + \dfrac{1}{2} \left( \alpha - 1 \right) \alpha s^{2} \left( k+1 \right) \sqrt{k+1} \left\lVert V \left( \bz^{k} \right) - V \left( \bz^{k-1} \right) \right\rVert ^{2} , \label{reg:Young:1}
	\end{align}
	where in the second estimate we use the fact that $\left( \alpha - 2 \right) k -1 \leq \left( \alpha - 1 \right) \left( k + 1 \right)$, while in the last one we combine $\sqrt{\alpha - 1} \leq \alpha$ and $sL < 1/2 < 1$.

In addition, for every $k \geq 2$ it holds
	\begin{align}
		\alpha s^{2} k \left\langle V \left( \bz^{k} \right) , V \left( z^{k} \right) - V \left( \bz^{k-1} \right) \right\rangle 
		\leq & \	 \dfrac{1}{2} \alpha s^{2} \sqrt{k} \left\lVert V \left( \bz^{k} \right) \right\rVert ^{2} + \dfrac{1}{2} \alpha s^{2} k \sqrt{k} \left\lVert V \left( z^{k} \right) - V \left( \bz^{k-1} \right) \right\rVert ^{2} \nonumber \\
		\leq & \ \dfrac{1}{2} \alpha s^{2} \sqrt{k} \left\lVert V \left( \bz^{k} \right) \right\rVert ^{2} + \dfrac{1}{2} \alpha s^{2} k \sqrt{k} \left\lVert V \left( \bz^{k-1} \right) - V \left( \bz^{k-2} \right) \right\rVert ^{2} \label{reg:Young:2} ,
	\end{align}
and, by using the Cauchy-Schwarz inequality and \eqref{ex:Lip},
	\begin{align}
		& 2s^{2} k^{2} \left\langle V \left( \bz^{k} \right) - V \left( \bz^{k-1} \right) , V \left( z^{k} \right) - V \left( \bz^{k-1} \right) \right\rangle \nonumber \\
		\leq \ 	& s^{3} Lk^{2} \left( \left\lVert V \left( \bz^{k} \right) - V \left( \bz^{k-1} \right) \right\rVert ^{2} + \left\lVert V \left( \bz^{k-1} \right) - V \left( \bz^{k-2} \right) \right\rVert ^{2} \right) \label{reg:C-S} .
	\end{align}
By plugging \eqref{reg:Young:1} - \eqref{reg:C-S} into \eqref{reg:mono} and adding then the result to \eqref{reg:inn}, we get after rearranging the terms for every $k \geq k_0$
	\begin{align}
	& \ 2 \lambda \left( 2 - \alpha \right) s \left\langle z^{k+1} - z_{*} , V \left( \bz^{k} \right) \right\rangle - 2 \left( 2 - \gamma \right) sk \left( k + \alpha \right) \left\langle z^{k+1} - z^{k} , V \left( \bz^{k} \right) - V \left( \bz^{k-1} \right) \right\rangle \nonumber \\
	\leq & \ \dfrac{1}{2} \lambda \left( \alpha - 2 \right) s^{2} \left[ \left( 2 - \frac{\alpha}{k + \alpha + 1} \right) \left\lVert V \left( \bz^{k} \right) \right\rVert ^{2}
	- \left( 2 - \frac{\alpha}{k + \alpha} \right) \left\lVert V \left( \bz^{k-1} \right) \right\rVert ^{2} \right] \nonumber \\
	& + 2s \left( 2 - \gamma \right) \left[ \left( k + 1 \right) ^{2} \left\langle z^{k+1} - z^{k} , V \left( z^{k+1} \right) - V \left( \bz^{k} \right) \right\rangle
	- k^{2} \left\langle z^{k} - z^{k-1} , V \left( z^{k} \right) - V \left( \bz^{k-1} \right) \right\rangle \right] \nonumber \\
	& - \dfrac{1}{2} \left( 2 - \gamma \right) \alpha s^{2} \left[ \left( k+1 \right) \sqrt{k+1} \left\lVert V \left( \bz^{k} \right) - V \left( \bz^{k-1} \right) \right\rVert ^{2}
	- k \sqrt{k} \left\lVert V \left( \bz^{k-1} \right) - V \left( \bz^{k-2} \right) \right\rVert ^{2} \right] \nonumber \\
	& - \left( 2 - \gamma \right) s^{3} L \left[\left( k+1 \right)^{2} \left\lVert V \left( \bz^{k} \right) - V \left( \bz^{k-1} \right) \right\rVert ^{2}
	- k^{2} \left\lVert V \left( \bz^{k-1} \right) - V \left( \bz^{k-2} \right) \right\rVert ^{2} \right] \nonumber \\
	& + 2 \lambda \left( 2 - \alpha \right) s \left\langle \bz^{k} - z_{*} , V \left( \bz^{k} \right) \right\rangle	
	- \dfrac{1}{2} s^{2} \Bigl( \mu_{k} - \left( 2 - \gamma \right) \left( 2k^{2} + \alpha k \right) \Bigr) \left\lVert V \left( \bz^{k} \right) - V \left( \bz^{k-1} \right) \right\rVert ^{2} \nonumber \\
	& + 2 \left( 2 - \gamma \right) \sqrt{\left( \alpha - 2 \right) k} \left\lVert z^{k+1} - z^{k} \right\rVert ^{2} + \dfrac{1}{2} \left( 2 - \gamma \right) \alpha s^{2} \sqrt{k} \left\lVert V \left( \bz^{k} \right) \right\rVert ^{2} , \label{reg:sum}
	\end{align}
	where we set
	\begin{align*}
	\mu_{k}  := &  - 2 \lambda \left( \alpha - 2 \right) - \left( 2 - \gamma \right) \left( \left( \alpha - 1 \right) \alpha \left( k+1 \right) \sqrt{k+1} + \alpha \left( k+1 \right) \sqrt{k+1} + 4sL \left( k + 1 \right) ^{2} \right) \nonumber \\
	& + \left( 2 - \gamma \right) \left( 2k^{2} + \alpha k \right) \nonumber \\
	 = & \left( 2 - \gamma \right) \left( \left( 2 - 4sL \right) \left( k + 1 \right) ^{2} + \left( \alpha - 4 \right) k - 2 - \alpha^{2} \left( k+1 \right) \sqrt{k+1} \right) - 2 \lambda \left( \alpha - 2 \right) \nonumber \\
	 = &  \left( 2 - \gamma \right) \left( 2 \left( 1 - 2sL \right) \left( k + 1 \right) + \alpha^{2} \sqrt{k+1} + \alpha - 4 \right) \left( k+1 \right) - \left( 2 - \gamma \right) \left( \alpha - 2 \right) - 2 \lambda \left( \alpha - 2 \right) .
	\end{align*}
Finally, by summing up the relations \eqref{reg:pre} and \eqref{reg:sum}, we obtain the desired estimate.

(ii) By the definition of $u_{\lambda}^{k}$ in \eqref{ex:defi:u-k-1-lambda} and by using the identity \eqref{pre:sum-2}, for every $k \geq 1$ it holds
	\begin{align*}
	\E_{\lambda}^{k}
	= & \ \dfrac{1}{2} \left\lVert u^{k}_{\lambda} \right\rVert ^{2} + 2 \lambda \left( \alpha - 1 - \lambda \right) \left\lVert z^{k} - z_{*} \right\rVert ^{2} + 2 \left( 2 - \gamma \right) \lambda sk \left\langle z^{k} - z_{*} , V \left( \bz^{k-1} \right) \right\rangle \nonumber \\
	& + \dfrac{1}{2} \left( 2 - \gamma \right) s^{2} k \left( \gamma k + \alpha \right) \left\lVert V \left( \bz^{k-1} \right) \right\rVert ^{2} \nonumber \\
	= & \ \dfrac{1}{2} \left\lVert 2 \lambda \left( z^{k} - z_{*} \right) + 2k \left( z^{k} - z^{k-1} \right) + \gamma sk V \left( \bz^{k-1} \right) \right\rVert ^{2}
	+ 2 \lambda \left( \alpha - 1 - \dfrac{2 \lambda}{\gamma} \right) \left\lVert z^{k} - z_{*} \right\rVert ^{2} \nonumber \\
	& + \dfrac{1}{2} \left( 2 - \gamma \right) \alpha s^{2} k \left\lVert V \left( \bz^{k-1} \right) \right\rVert ^{2}
	+ \dfrac{2 - \gamma}{2 \gamma} \left\lVert 2 \lambda \left( z^{k} - z_{*} \right) + \gamma sk V \left( \bz^{k-1} \right) \right\rVert ^{2} \nonumber \\
	= & \ \dfrac{2 - \gamma}{2 \gamma} \left( \left\lVert 2 \lambda \left( z^{k} - z_{*} \right) + 2k \left( z^{k} - z^{k-1} \right) + \gamma sk V \left( \bz^{k-1} \right) \right\rVert ^{2} + \left\lVert 2 \lambda \left( z^{k} - z_{*} \right) + \gamma sk V \left( \bz^{k-1} \right) \right\rVert ^{2} \right) \nonumber \\
	& + 2 \lambda \left( \alpha - 1 - \dfrac{2 \lambda}{\gamma} \right) \left\lVert z^{k} - z_{*} \right\rVert ^{2}
	+ \dfrac{1}{2} \left( 2 - \gamma \right) \alpha s^{2} k \left\lVert V \left( \bz^{k-1} \right) \right\rVert ^{2} \nonumber \\
	&  + \dfrac{\gamma - 1}{\gamma} \left\lVert 2 \lambda \left( z^{k} - z_{*} \right) + 2k \left( z^{k} - z^{k-1} \right) + \gamma sk V \left( \bz^{k-1} \right) \right\rVert ^{2} \nonumber \\
	= & \ \dfrac{2 - \gamma}{\gamma} \left( \left\lVert 2 \lambda \left( z^{k} - z_{*} \right) + k \left( z^{k} - z^{k-1} \right) + \gamma sk V \left( \bz^{k-1} \right) \right\rVert ^{2} + k^{2} \left\lVert z^{k} - z^{k-1} \right\rVert ^{2} \right) \nonumber \\
	& + 2 \lambda \left( \alpha - 1 - \dfrac{2 \lambda}{\gamma} \right) \left\lVert z^{k} - z_{*} \right\rVert ^{2}
	+ \dfrac{1}{2} \left( 2 - \gamma \right) \alpha s^{2} k \left\lVert V \left( \bz^{k-1} \right) \right\rVert ^{2} \nonumber \\
	&  + \dfrac{\gamma - 1}{\gamma} \left\lVert 2 \lambda \left( z^{k} - z_{*} \right) + 2k \left( z^{k} - z^{k-1} \right) + \gamma sk V \left( \bz^{k-1} \right) \right\rVert ^{2} .
	\end{align*}
Consequently, as $1 < \gamma < 2$, for every $k \geq k_{1} = \left\lceil \frac{2 \lambda \left( \alpha - 2 \right)}{\left( 2 - \gamma \right) \alpha} \right\rceil$ we have
	\begin{align*}
	\F_{\lambda}^{k}
	 = & \ \E_{\lambda}^{k} - 2 \left( 2 - \gamma \right) sk^{2} \left\langle z^{k} - z^{k-1} , V \left( z^{k} \right) - V \left( \bz^{k-1} \right) \right\rangle \nonumber \\
	& + \dfrac{1}{2} \left( 2 - \gamma \right) s^{2} k \sqrt{k} \left( 2sL \sqrt{k} + \alpha \right) \left\lVert V \left( \bz^{k-1} \right) - V \left( \bz^{k-2} \right) \right\rVert ^{2} \nonumber \\
	& - \dfrac{1}{2} \lambda \left( \alpha - 2 \right) s^{2} \left( 2 - \frac{\alpha}{k + \alpha} \right) \left\lVert V \left( \bz^{k-1} \right) \right\rVert ^{2} \nonumber \\
	\geq & \dfrac{2 - \gamma}{\gamma} \left( \left\lVert 2 \lambda \left( z^{k} - z_{*} \right) + k \left( z^{k} - z^{k-1} \right) + \gamma sk V \left( \bz^{k-1} \right) \right\rVert ^{2} + k^{2} \left\lVert z^{k} - z^{k-1} \right\rVert ^{2} \right) \nonumber \\
	& + 2 \lambda \left( \alpha - 1 - \dfrac{2 \lambda}{\gamma} \right) \left\lVert z^{k} - z_{*} \right\rVert ^{2}
	- 2 \left( 2 - \gamma \right) sk^{2} \left\langle z^{k} - z^{k-1} , V \left( z^{k} \right) - V \left( \bz^{k-1} \right) \right\rangle \nonumber \\
	&  + \left( 2 - \gamma \right) s^{3} Lk^{2} \left\lVert V \left( \bz^{k-1} \right) - V \left( \bz^{k-2} \right) \right\rVert ^{2} .
	\end{align*}
Now we use relation \eqref{ex:Lip} and apply Lemma \ref{lem:quad} with $\left( a , b , c \right) := \Bigl( \frac{1}{2} , -s , \frac{s}{L} \Bigr)$ to verify that for every $k \geq 1$
	\begin{align*}
		& \dfrac{1}{2} k^{2} \left\lVert z^{k} - z^{k-1} \right\rVert ^{2} - 2 sk^{2} \left\langle z^{k} - z^{k-1} , V \left( z^{k} \right) - V \left( \bz^{k-1} \right) \right\rangle + s^{3} L k^{2} \left\lVert V \left( \bz^{k-1} \right) - V \left( \bz^{k-2} \right) \right\rVert ^{2} \nonumber \\
		\geq & \ k^{2} \left( \dfrac{1}{2} \left\lVert z^{k} - z^{k-1} \right\rVert ^{2} - 2s \left\langle z^{k} - z^{k-1} , V \left( z^{k} \right) - V \left( \bz^{k-1} \right) \right\rangle + \dfrac{s}{L} \left\lVert V \left( z^{k} \right) - V \left( \bz^{k-1} \right) \right\rVert ^{2} \right)\\
 \geq & \ 0.
	\end{align*}
Combining the last two estimates, one can easily conclude that for every $k \geq k_{1}$ it holds
	\begin{align*}
	\F_{\lambda}^{k}
	 \geq & \ \dfrac{2 - \gamma}{\gamma} \left\lVert 2 \lambda \left( z^{k} - z_{*} \right) + k \left( z^{k} - z^{k-1} \right) + \gamma sk V \left( \bz^{k-1} \right) \right\rVert ^{2} \nonumber \\
	& + \dfrac{\left( 2 - \gamma \right) ^{2}}{2 \gamma} k^{2} \left\lVert z^{k} - z^{k-1} \right\rVert ^{2} + 2 \lambda \left( \alpha - 1 - \dfrac{2 \lambda}{\gamma} \right) \left\lVert z^{k} - z_{*} \right\rVert ^{2} ,
	\end{align*}
	which is the desired inequality.
	\qedhere
\end{prooffff}

\begin{proofffff}  (i) First we notice that $2 \left( 1 - \frac{1}{\gamma} \right) = 2 - \frac{2}{\gamma} < 1$ and
	\begin{equation*}
		\dfrac{1}{\gamma \left( \alpha - 2 \right)} \bigl( \left( 2 - \gamma \right) \left( \alpha - 1 \right) + \left( \gamma - 1 \right) \left( \alpha - 2 \right) \bigr) < 1 \Leftrightarrow 1 + \dfrac{1}{\alpha - 1} < \gamma < 2 .
	\end{equation*}
This means, if $\gamma$ satisfies \eqref{trunc:gamma}, it holds
		\begin{equation*}
			\max \left\lbrace \sqrt{2 \left( 1 - \dfrac{1}{\gamma} \right)} , \sqrt{\dfrac{\left( 2 - \gamma \right) \left( \alpha - 1 \right) + \left( \gamma - 1 \right) \left( \alpha - 2 \right)}{\gamma \left( \alpha - 2 \right)}}  \right\rbrace < 1 ,
		\end{equation*}
	and thus one can choose $\delta$ to fulfill \eqref{trunc:delta}.
		
For the quadratic expression in $R_k$ we calculate
		\begin{align*}
			\dfrac{\Delta_{k}'}{s^{2}}
			& := \left( \omega_{0} k + \omega_{1} \right) ^{2} - \delta^{2} k \Bigl( \omega_{2} \sqrt{k} + \omega_{3} \Bigr) \Bigl( \omega_{4} \sqrt{k} + \omega_{5} \Bigr) \nonumber \\
			& = \left( \omega_{0}^{2} - \delta^{2} \omega_{2} \omega_{4} \right) k^{2} - \delta^{2} \left( \omega_{2} \omega_{5} + \omega_{3} \omega_{4} \right) k \sqrt{k} + \left( 2 \omega_{0} \omega_{1} - \delta^{2} \omega_{3} \omega_{5} \right) k + \omega_{1}^{2} .
		\end{align*}		
		Since $\left( \omega_{0}^{2} - \delta^{2} \omega_{2} \omega_{4} \right) k^{2}$ is the dominant term in the above polynomial, it suffices to guarantee that $\omega_{0}^{2} - \delta^{2} \omega_{2} \omega_{4} < 0$ in order to be sure that there exits some integer $k_{2} \left( \lambda \right) \geq 1$ such that $\Delta_{k}' \leq 0$ for every $k \geq k_{2} \left( \lambda \right)$ and to obtain from here, due to Lemma \ref{lem:quad} \ref{quad:vec}, that $R_{k} \leq 0$ for every $k \geq k_{2} \left( \lambda \right)$.
		
It remains to show that there exists a choice of $\lambda$ for which $\omega_{0}^{2} - \delta^{2} \omega_{2} \omega_{4} < 0$ holds. We set $\xi := \lambda + 1 - \alpha \leq 0$ and get
		\begin{align*}
			\omega_{0} & = 2 \lambda + \gamma - \alpha + \gamma \left( 1 - \alpha \right) = 2 \lambda - \alpha + \gamma \left( 2 - \alpha \right) = 2 \xi + \left( \gamma - 1 \right) \left( 2 - \alpha \right) , \\
			\omega_{2} \omega_{4} 	& = - 4 \gamma \left( \alpha - 2 \right) \xi.
		\end{align*}
This means that we have to guarantee that there exists a choice for $\xi$ for which
		\begin{align}
			\omega_{0}^{2} - \delta^{2} \omega_{2} \omega_{4}
			& = \Bigl( 2 \xi - \left( \gamma - 1 \right) \left( \alpha - 2 \right) \Bigr) ^{2} + 4 \delta^{2} \gamma \left( \alpha - 2 \right) \xi \nonumber \\
			& = 4 \xi^{2} + 4 \left( \alpha - 2 \right) \left( \delta^{2} \gamma - \gamma + 1 \right) \xi + \left( \gamma - 1 \right) ^{2} \left( \alpha - 2 \right) ^{2} < 0 . \label{trunc:omega-a}
		\end{align}
		A direct computation shows that, according to \eqref{trunc:delta},
		\begin{equation*}
			\Delta_{\xi}' := 4 \left( \alpha - 2 \right) ^{2} \left[ \left( \delta^{2} \gamma - \gamma + 1 \right) ^{2} - \left( \gamma - 1 \right) ^{2} \right]
			= 4 \left( \alpha - 2 \right) ^{2} \delta^{2} \gamma \left( \delta^{2} \gamma - 2 \left( \gamma - 1
			\right) \right) > 0.
		\end{equation*}
Hence,  in order to get \eqref{trunc:omega-a}, we have to choose $\xi$ between the two roots of the quadratic function arising in this formula, in other words
\begin{align*}
	\xi_{1} \left( \alpha , \gamma \right) & := - \dfrac{1}{2} \left( \alpha - 2 \right) \left( \delta^{2} \gamma - \gamma + 1 \right) - \dfrac{\sqrt{\Delta_{\xi}'}}{4} \nonumber \\
	& < \xi = \lambda + 1 - \alpha
	< \xi_{2} \left( \alpha , \gamma \right) := - \dfrac{1}{2} \left( \alpha - 2 \right) \left( \delta^{2} \gamma - \gamma + 1 \right) + \dfrac{\sqrt{\Delta_{\xi}'}}{4} .
\end{align*}
Obviously $\xi_{1} \left( \alpha , \gamma \right) < 0$ and from Viète’s formula $\xi_{1} \left( \alpha , \gamma \right) \cdot \xi_{2} \left( \alpha , \gamma \right) = \frac{\left( \gamma - 1 \right) ^{2} \left( \alpha - 2 \right) ^{2}}{4}$, it follows that $\xi_{2} \left( \alpha , \gamma \right) < 0$ as well. 

Therefore, going back to $\lambda$, in order to be sure that $\omega_{0}^{2} - \delta^{2} \omega_{2} \omega_{4} < 0$ this must be chosen such that
\begin{equation*}
\alpha - 1 + \xi_{1} \left( \alpha , \gamma \right) < \lambda < \alpha - 1 + \xi_{2} \left( \alpha , \gamma \right) .
\end{equation*}
Next we will show that
\begin{equation}
\label{trunc:check}
0 < \alpha - 1 - \dfrac{1}{2} \left( \alpha - 2 \right) \left( \delta^{2} \gamma - \gamma + 1 \right) < \dfrac{\gamma}{2} \left( \alpha - 1 \right) .
\end{equation}
Indeed, the left-hand side inequality \eqref{trunc:check} is straightforward since
\begin{equation*}
0 < \alpha - 1 - \dfrac{1}{2} \left( \alpha - 2 \right) \left( \delta^{2} \gamma - \gamma + 1 \right) \Leftrightarrow \delta^{2} < 1 + \dfrac{1}{\gamma} \left( 1 + \dfrac{2}{\alpha - 2} \right).
\end{equation*}
The right-hand side inequality \eqref{trunc:check} is equivalent to
\begin{equation*}
\alpha - 1 - \dfrac{1}{2} \left( \alpha - 2 \right) \left( \delta^{2} \gamma - \gamma + 1 \right) < \dfrac{\gamma}{2} \left( \alpha - 1 \right) \Leftrightarrow \delta^{2} > \dfrac{1}{\gamma \left( \alpha - 2 \right)} \bigl( \left( 2 - \gamma \right) \left( \alpha - 1 \right) + \left( \gamma - 1 \right) \left( \alpha - 2 \right) \bigr),
\end{equation*}
which is true according to \eqref{trunc:delta}.

From \eqref{trunc:check} we immediately deduce that
\begin{equation*}
0 < \alpha - 1 + \xi_{2} \left( \alpha , \gamma \right)
\quad \textrm{ and } \quad
\alpha - 1 + \xi_{1} \left( \alpha , \gamma \right) < \dfrac{\gamma}{2} \left( \alpha - 1 \right) ,
\end{equation*}
which allow us to choose
\begin{equation*}
\underline{\lambda} \left( \alpha , \gamma \right) := \max \left\lbrace 0 , \alpha - 1 + \xi_{1} \left( \alpha , \gamma \right) \right\rbrace
< \overline{\lambda} \left( \alpha , \gamma \right) := \min \left\lbrace \dfrac{\gamma}{2} \left( \alpha - 1 \right) , \alpha - 1 + \xi_{2} \left( \alpha , \gamma \right) \right\rbrace .
\end{equation*}

In conclusion, choosing $\lambda$ to satisfy $\underline{\lambda} \left( \alpha , \gamma \right) < \lambda < \overline{\lambda} \left( \alpha , \gamma \right)$, we have $\omega_{0}^{2} - \delta^{2} \omega_{2} \omega_{4} < 0$ and therefore there exists some integer $k_{2} \left( \lambda \right) \geq 1$ such that $R_k \leq 0$ for every $k \geq k_{2} \left( \lambda \right)$.
				
(ii) For every $k \geq 1$ we have
		\begin{align*}
		\mu_{k} - \left( 2 - \gamma \right) \left( 1 - 2sL \right) \left( k+1 \right)^{2}
		= & \left( 2 - \gamma \right) \left( 1 - 2sL \right) \left( k+1 \right)^{2} +  \left( 2 - \gamma \right) \alpha^{2} \left( k+1 \right) \sqrt{k+1}\\ 
& + \left( 2 - \gamma \right) \left(  \alpha - 4 \right) \left( k+1 \right) - \left( 2 - \gamma \right) \left( \alpha - 2 \right) - 2 \lambda \left( \alpha - 2 \right) ,
		\end{align*}
and the conclusion is obvious since $\gamma <2$ and $s < \dfrac{1}{2L}$.
		\qedhere
\end{proofffff}
	
{\bf Acknowledgements.} The authors are thankful to the handling editor and two anonymous reviewers for comments and remarks which improved the quality of the manuscript, in particular for the observation on which we elaborate in Remark \ref{rmk:intuition} and the suggestion to use the performance profiles in the numerical experiments.

%
%
%
%


\end{document}